\documentclass{article}
\usepackage{tikz, tikz-3dplot}
\usepackage{graphicx} 
\usepackage{amsmath}
\usepackage{amsfonts}
\usepackage{cases}
\usepackage{amsthm}
\usepackage{comment}
\usepackage{amssymb}
\usepackage{hyperref}
\newtheorem{theorem}{Theorem}
\newtheorem{lemma}[theorem]{Lemma}
\newtheorem{corollary}[theorem]{Corollary}
\newtheorem{prop}{Proposition}
\usepackage{tabularx}
\usepackage{pgfplots}
\usepackage{float}
\pgfplotsset{compat=1.18}
\usepackage{float}
\usepackage{authblk}
\usepackage[most]{tcolorbox}
\usetikzlibrary{calc,arrows.meta,fillbetween}
\pgfdeclarelayer{pre main}
\pgfdeclarelayer{main}
\pgfsetlayers{pre main,main}
\usepackage{subcaption}
\usepackage[para]{footmisc}
\usepgfplotslibrary{fillbetween}
\usepackage[margin=1in]{geometry}
\usepackage[backend=biber,maxnames=99, minnames=99]{biblatex}
\newtheorem{remark}{Remark}
\renewcommand{\d}{\text{d}}
\renewcommand{\dfrac}[2]{\frac{\d #1}{\d #2}}

\renewcommand{\u}{\tilde{u}}

\newcommand{\dpri}{^{\prime\prime}}
\newcommand{\pri}{^{\prime}}

\newcommand{\dd}{\,d}
\newcommand{\gap}{\hspace{0.3 cm}}

\newcommand{\eps}{\varepsilon}
\renewcommand{\epsilon}{\varepsilon}

\counterwithin{equation}{section}

\title{The Riemann Problem for a 3×3 Generalized Chaplygin Gas System with Variable Pressure}
\date{\today}
\author[1]{Thomas Allen \thanks{thallen@ncsu.edu}}
\author[2]{Ella Kim \thanks{ellakim@ucsb.edu}}
\author[3]{Jonathan Nunes \thanks{jonathan.nunes@ufl.edu}}
\author[4]{Rita Xing \thanks{rxing28@amherst.edu}}
\author[5]{Charis Tsikkou \thanks{tsikkou@math.wvu.edu}}
\affil[1]{Department of Mathematics, North Carolina State University, Raleigh, NC 27695, USA}
\affil[2]{Department of Mathematics, University of California, Santa Barbara, Santa Barbara, CA 93106-3080, USA}
\affil[3]{Department of Mathematics, University of Florida, Gainesville, FL 32611-8105, USA}
\affil[4]{Department of Mathematics and Statistics, Amherst College, Amherst, MA 01002, USA}
\affil[5]{School of Mathematical and Data Sciences, West Virginia University, Morgantown, WV 26506, USA}

\begin{document}
\maketitle
\begin{abstract}\noindent
We consider the Riemann problem for a $3\times3$ system of conservation laws with generalized Chaplygin pressure $p(\rho,v)=-\frac{A(v)}{\rho^\alpha}, \quad 0<\alpha\leq 1,$ where the pressure depends on an additional transported variable. We analyze the system's wave structure and classify the Riemann solutions. Whenever classical solutions consisting of shocks, rarefaction waves, and contact discontinuities fail to exist, singular solutions arise. We verify that these satisfy the conservation laws in the distributional sense within the classical Dirac delta framework, and compare them with Nedeljkov's shadow-wave construction \cite{MarkoShadow}, giving two complementary descriptions of the same singular solution. We further study admissibility via the Dafermos maximum entropy dissipation principle, with several examples showing how it selects the physically relevant solution. Lax–Friedrichs simulations illustrate the Riemann wave patterns and provide a comparison with the analytical results.

To construct viscous profiles for the isolated overcompressive $\delta$-shock, we assume $\alpha\in\mathbb{Q}\cap(0,1]$ and apply the Dafermos regularization together with a spherical blow-up. Working in three directional charts, we construct the reduced singular concatenation consisting of the left outer orbit, the middle orbit on the blown-up boundary, and the right outer orbit. We then prove that, for sufficiently small positive viscosity, this singular concatenation perturbs to a heteroclinic orbit. Consequently, the isolated overcompressive $\delta$-shock is realized as the zero-viscosity limit of a family of self-similar Dafermos viscous profiles.
\end{abstract}

\vspace{5mm}


\vspace{5mm}

 
\section*{Notation}
We define $\left[\cdot\right] := \left[\cdot\right]_{\text{jump}} = \cdot_R - \cdot_L$.

\section{Introduction} \label{intro}
We study the Riemann problem for a system of three nonlinear hyperbolic conservation laws modeling a generalized Chaplygin gas in one spatial dimension. The system is given by
\begin{align}
    \label{cons1}
    \rho_t+\left(\rho u\right)_x&=0, \\
    \label{cons2}
    (\rho u)_t+\left(\rho u^2-\cfrac{A(v)}{\rho^\alpha}\right)_x&=0, \\
    \label{cons3}
    (\rho v)_t+\left(\rho u v\right)_x&=0,
\end{align}
where $\rho$, $u$ and $v$ denote the density, velocity, and an additional transported scalar variable, respectively. Systems of Chaplygin gas type have been proposed as models of the accelerated expansion of the universe via an exotic fluid with negative pressure that is inversely proportional to the density \cite{Bento, Gorini}. In the classical generalized Chaplygin gas, the pressure coefficient $A$ is assumed constant \cite{Qu, Cheng, DavorMarko}. Here, we consider the more general case in which $A$ is a strictly positive, continuous function of the transported variable $v$, thus coupling the pressure law with the third conservation law. Throughout this work, we assume $0<\alpha\leq 1$, which, together with $A(v)$, determines the rate of cosmic expansion. Although the system is mathematically meaningful for arbitrary values of $\rho$, we restrict our attention to solutions satisfying $\rho>0,$ corresponding to physically meaningful density. 

Unlike the classical generalized Chaplygin gas, where the pressure coefficient is constant, the present model contains an additional conservation law governing the evolution of the transported variable $v$. Through the pressure coefficient $A(v)$, this variable directly influences the momentum equation, introducing a nonlinear coupling between the pressure law and the flow. Consequently, the pressure law evolves together with the solution rather than being prescribed \emph{a priori}. The present model can be viewed as a natural extension of the three-equation generalized Chaplygin gas with a potential studied in \cite{Cheng, DavorMarko}. Replacing the additive potential with a pressure coefficient depending on the transported variable $v$ leads to a different family of Riemann problems in which the pressure law evolves together with the solution. In particular, the transported pressure coefficient changes the geometry of the wave curves and the resulting classification of the Riemann solutions. In addition, the present formulation naturally admits an entropy-dissipation framework based on the Dafermos maximum entropy dissipation principle, which we use to investigate the admissibility of representative singular solutions.

For one-dimensional systems of conservation laws, the Riemann problem plays a central role in the mathematical theory of nonlinear hyperbolic conservation laws and in the numerical approximation of their solutions. Since the pioneering work of Riemann and the subsequent developments of Lax and Glimm, solutions of the Riemann problem have become one of the fundamental tools for describing the local behavior of solutions to the Cauchy problem and serve as the building blocks of many finite-volume and Godunov-type numerical methods \cite{Dafermos,Glimm,Lax,Lev_1,Toro}.

Motivated by these considerations, we consider the Riemann initial data 
\begin{align} \label{riemannData}
    (\rho, u, v)(x, t)\Big|_{t=0}=\begin{cases}
        (\rho_L, u_L, v_L), & x<0, \\
        (\rho_R, u_R, v_R), & x>0,
    \end{cases}
\end{align}
consisting of two constant states separated by a jump discontinuity at the origin $x=0$. 

Generalized Chaplygin gas models have attracted considerable attention from both the mathematical and physical communities. Depending on the initial data, solutions to the Riemann problem consist of combinations of shocks, rarefactions, and contact discontinuities. However, for certain regions of the admissible state space, classical solutions do not exist, and singular solutions arise naturally. 

The generalized Chaplygin gas has provided a particularly rich setting for the study of singular solutions. The Riemann problem has been investigated for several Chaplygin-type systems, including the classical generalized Chaplygin gas, models with source terms, balance laws, and systems with additional transported quantities; see, for example,
\cite{Qu,Cheng,REU2024,DavorMarko,Pang,Li,Zh_1}
and the references therein. The present work extends this line of research by allowing the pressure coefficient itself to evolve through an additional conservation law.

The study of singular solutions in systems of conservation laws originated with the work of Korchinski \cite{Korch} and the singular shock solutions introduced by Keyfitz and Kranzer \cite{Ke_Kr_1,Ke_Kr_2,Ke_Kr_3}. Since then, $\delta$-shock waves have been investigated for a wide variety of systems using generalized Rankine--Hugoniot conditions, vanishing viscosity, weak asymptotic methods, distributional formulations, and shadow wave approximations; see, for example, \cite{ChenLiu,DaShe,Hsu,Ke_4,MarkoShadow,Sc,Shelkovich2002,Yuxi}
and the references therein.   

The main goal of this paper is to provide a complete classification of the Riemann solutions for system \eqref{cons1}--\eqref{cons3} and to understand how the transported pressure coefficient affects the corresponding wave structure. We characterize all admissible classical wave patterns and identify the regions in which singular solutions arise. We verify that these singular solutions satisfy the system in the sense of distributions using the classical Dirac delta framework and compare them with the corresponding shadow wave construction introduced by Nedeljkov \cite{MarkoShadow}. We also investigate the admissibility of representative singular solutions using the Dafermos maximum entropy dissipation principle and establish the existence of the corresponding singular profiles through a Dafermos regularization and spherical blow-up analysis. To the best of our knowledge, this is the first rigorous construction of self-similar Dafermos viscous profiles for this generalized Chaplygin gas model with a transported pressure coefficient and among the first rigorous constructions of self-similar Dafermos viscous profiles for a 3×3 system of conservation laws.

The structure of the Riemann solution depends on the value of $\alpha$. In the case $\alpha=1,$ all characteristic families are linearly degenerate, and the classical solution consists entirely of contact discontinuities. In this regime, the Riemann problem admits two possible wave structures: a classical solution consisting of contact discontinuities and an overcompressive singular solution. The resulting wave structure is comparatively simple and is governed by a single set of inequalities. In contrast, when $0<\alpha<1,$ the $1$- and $3$-characteristic families are genuinely nonlinear, while the $2$-characteristic family is linearly degenerate. The corresponding Riemann solutions consist of combinations of shock waves, rarefaction waves, contact discontinuities, and overcompressive singular solutions in regions of the $(\rho,u,v)$-space where classical wave patterns do not exist. Compared with the constant-pressure case, the transported pressure coefficient modifies the wave curves and changes the partition of the Riemann data into the different solution regimes. Numerical simulations obtained using a Lax--Friedrichs scheme are presented to illustrate the Riemann wave patterns and to compare with the analytical results.

To establish the existence of the corresponding singular profiles, we consider a Dafermos regularization of the system and study the resulting singularly perturbed dynamical system by means of a spherical blow-up. For the blow-up analysis, we assume $\alpha\in\mathbb{Q}$. The dynamics are analyzed in three coordinate charts, where a singular concatenation is constructed and shown to persist as a heteroclinic orbit for sufficiently small values of the regularization parameter. Our analysis follows the general philosophy of Schecter \cite{Sc}, adapted to the present three-equation system. It also builds on our previous work on Dafermos regularizations and singular shock profiles for related systems \cite{REU2025GSPT,Ts}. 

The remainder of the paper is organized as follows. In Section~\ref{sec:classical}, we derive the classical wave curves for the case $0<\alpha<1$, including shock waves, rarefaction waves, and contact discontinuities. Section~\ref{sec:singular} is devoted to the construction and analysis of the corresponding singular solutions using both the classical Dirac delta distribution framework and the shadow wave approach of Nedeljkov \cite{MarkoShadow}. In Section~\ref{sec:combo}, we combine classical and singular wave patterns to obtain a complete classification of the Riemann solutions for the case $0<\alpha<1$. We also investigate the admissibility of the resulting singular solutions using the Dafermos maximum entropy dissipation principle. Section~\ref{sec:numerics} contains numerical simulations that illustrate the theoretical results for a particular choice of $A(v)$. The case $\alpha=1$ is treated separately in Section~\ref{sec:analytic}, where a complete analytical classification of the corresponding Riemann solutions is obtained. Finally, Section~\ref{sec:GSPT} is devoted to Dafermos regularization and blow-up analysis, establishing the existence of the corresponding overcompressive singular profiles.

\section{Classical Waves Curves: Contact Discontinuities, Shocks, and Rarefactions} \label{sec:classical}
We now analyze the classical wave structure of system \eqref{cons1}--\eqref{cons3} for the case $0<\alpha<1.$ We first compute the eigenvalues and eigenvectors of the Jacobian matrix and determine the hyperbolic structure of the system by identifying the genuinely nonlinear and linearly degenerate characteristic families. We then derive the corresponding rarefaction curves, shock curves, and contact discontinuities that form the classical solutions of the Riemann problem.

\subsection{Genuine Nonlinearity and Linear Degeneracy of Eigenvalue Families}
The system may be written in conservation form as
\begin{equation} \label{eq:sys_vec}
\partial_t H+\partial_x G=0,
\end{equation}
where 
\begin{align*}
H =
\begin{pmatrix}
\rho \\
\rho u\\
\rho v
\end{pmatrix},\qquad G =
\begin{pmatrix}
\rho u \\
\rho u^2-\frac{A(v)}{\rho^\alpha} \\
\rho u v
\end{pmatrix}.
\end{align*}
Let $D$ denote the differential operator
$D=\left(\frac{\partial}{\partial \rho},
\frac{\partial}{\partial u},
\frac{\partial}{\partial v}\right).
$ The characteristic speeds are determined by $\det(DG-\lambda DH)=0,$ where 
\begin{align*}
DH = &
\begin{pmatrix}
1 & 0 & 0 \\ 
u & \rho & 0 \\
v & 0 & \rho
\end{pmatrix}, \qquad DG =
\begin{pmatrix}
u & \rho & 0 \\ 
u^2+\frac{\alpha A(v)}{\rho^{\alpha+1}} & 2\rho u & -\frac{A'(v)}{\rho^\alpha} \\ 
uv & \rho v & \rho u
\end{pmatrix}
\end{align*} denote the Jacobian matrices of $H$ and $G$ with respect to $(\rho,u,v).$
The characteristic speeds are given by 
\begin{align}
    \lambda_1&=u-\sqrt{\frac{\alpha A(v)}{\rho^{\alpha + 1}}}, \qquad \lambda_2=u, \qquad \lambda_3=u+\sqrt{\frac{\alpha A(v)}{\rho^{\alpha + 1}}},
\end{align}
where $\lambda_1<\lambda_2<\lambda_3$ for any state $(\rho,u,v)$. Hence, the system is strictly hyperbolic. The corresponding eigenvectors are
\begin{align}
    r_1 = &
\begin{pmatrix}
\rho^\frac{\alpha+3}{2} \\ 
-\sqrt{\alpha A(v)} \\
0
\end{pmatrix}, \qquad r_2 = 
\begin{pmatrix}
\rho A'(v) \\ 
0 \\
\alpha A(v)
\end{pmatrix}, \qquad r_3 =
\begin{pmatrix}
\rho^\frac{\alpha+3}{2} \\ 
\sqrt{\alpha A(v)} \\
0
\end{pmatrix}.
\end{align}
Recall that the $i$-th characteristic family is linearly degenerate if 
$\nabla\lambda_i\cdot r_i=0,$ and genuinely nonlinear otherwise. Linearly degenerate families generate contact discontinuities, whereas genuinely nonlinear families generate shock waves and rarefaction waves. For the case $\alpha=1$, all three characteristic families are linearly degenerate. This case is treated separately in Section~\ref{sec:analytic}. Throughout the remainder of this section and Sections~\ref{sec:singular}--\ref{sec:numerics}, we assume that $0<\alpha<1$. In this case, the second characteristic family is linearly degenerate, while the first and third characteristic families are genuinely nonlinear.

\subsection{Shock Curves}
We now derive the shock curves corresponding to the first and third characteristic families. Define the state vector $U=(\rho, u, v)^T$ and let $$U_-=(\rho_-,u_-,v_-)^T
\quad\text{and}\quad U_+=(\rho_+,u_+,v_+)^T$$
denote the states on the left and right of a shock, respectively. If the shock moves with speed $x'(t)=s$, then the Rankine--Hugoniot conditions are
\begin{equation}
    \begin{cases}
    s[\rho]&=[\rho u] \\[1.5ex]
    s[\rho u]&=[\rho u^2-\frac{A(v)}{\rho^\alpha}] \\[1.5ex]
    s[\rho v]&=[\rho uv]
\end{cases}
\end{equation}
where $[\cdot]$ denotes the jump across the shock. These conditions imply
\begin{align} \label{eq:tandem_shocks}
    s=\frac{[\rho u]}{[\rho]}=\frac{[\rho u^2-\frac{A(v)}{\rho^\alpha}]}{[\rho u]}=\frac{[\rho uv]}{[\rho v]}.
\end{align}
Fix a left state $U_-=(\rho_-,u_-,v_-).$ Solving \eqref{eq:tandem_shocks} gives
$v_+=v_-,$ and
$$u_+=u_-\pm\sqrt{A(v_-)\left(\frac1{\rho_+}-\frac1{\rho_-}\right)\left(\frac1{\rho_+^\alpha}-\frac1{\rho_-^\alpha}\right)}.$$
The admissible branch is determined by the Lax entropy conditions. Consequently, the states that can be connected to the left state $U_-$ by an admissible $1$-shock lie on the curves
\begin{equation}
S_1(U_-):\quad u=u_--\sqrt{A(v_-)\left(\frac{1}{\rho}-\frac{1}{\rho_-}\right)\left(\frac1{\rho^\alpha}-\frac1{\rho_-^\alpha}\right)},
\qquad \rho>\rho_-, \quad v=v_-.
\tag{$S_1$}
\end{equation}
For later use, we also introduce the inverse $1$-shock curve. Fixing a right state $U_+$, we define $IS_1(U_+)=\{V: \ U_+\in S_1(V)\}$, that is, the set of all left states that can be connected to $U_+$ by an admissible $1$-shock. Explicitly,
\begin{equation}
IS_1(U_+):\quad u=u_++\sqrt{A(v_+)\left(\frac{1}{\rho}-\frac{1}{\rho_+}\right)\left(\frac1{\rho^\alpha}-\frac1{\rho_+^\alpha}\right)},
\qquad 0<\rho<\rho_+, \quad v=v_+.
\tag{$IS_1$}
\end{equation}
Similarly, the states that can be connected to $U_-$ by a $3$-shock lie on the curve
\begin{equation}
S_3(U_-):\quad u=u_--\sqrt{A(v_-)\left(\frac{1}{\rho}-\frac{1}{\rho_-}\right)\left(\frac1{\rho^\alpha}-\frac1{\rho_-^\alpha}\right)},\qquad 0<\rho<\rho_-,\quad v=v_-.\tag{$S_3$}
\end{equation}
For later use, we also introduce the inverse $3$-shock curve. Fixing the right state $U_+$ we define $IS_3(U_+)=\{V: U_+\in S_3(V)\},$ that is, the set of all states that can be connected to $U_+$ by an admissible $3$-shock. Explicitly,  
\begin{equation}
IS_3(U_+):\quad u=u_++\sqrt{A(v_+)\left(\frac{1}{\rho}-\frac{1}{\rho_+}\right)\left(\frac1{\rho^\alpha}-\frac1{\rho_+^\alpha}\right)},\qquad \rho>\rho_+,\quad v=v_+.\tag{$IS_3$}
\end{equation}
Throughout the remainder of the paper, we write $S_i(U)$ and $IS_i(U)$ for the $i$-shock and inverse $i$-shock curve associated with an arbitrary state $U$, respectively. 

\subsection{Rarefaction Curve}
We now derive the rarefaction curves corresponding to the first and third characteristic families. Seeking self-similar solutions of the form
$$(\rho,u,v)(x,t)=U(\xi),\qquad
\xi=\frac{x}{t},$$
the system reduces to a system of ordinary differential equations. Along the $i$-th characteristic family,
$$\frac{dU}{d\xi}=r_i(U),$$
where $r_i$ is the right eigenvector associated with $\lambda_i$. Integrating these equations gives the rarefaction curves. The states that can be connected to the left state $U_- = (\rho_-, u_-, v_-)$ by a $1$-rarefaction lie on the curve
\begin{equation}
R_1(U_-):\quad u=u_-+\frac{2\sqrt{\alpha A(v_-)}}{\alpha+1}
\left(\frac{1}{\rho^{\frac{\alpha+1}{2}}}
-\frac{1}{\rho_-^{\frac{\alpha+1}{2}}}
\right),
\qquad 0<\rho<\rho_- \qquad v = v_-.
\tag{$R_1$}
\end{equation}
For later use, we also introduce the inverse $1$-rarefaction curve. Fixing a right state $U_+$, we define $IR_1(U_+)=\{V: \ U_+\in R_1(V)\}$, that is, the set of all left states that can be connected to $U_+$ by an admissible $1$-rarefaction. Explicitly,
\begin{equation}
IR_1(U_+):\quad u=u_++\frac{2\sqrt{\alpha A(v_+)}}{\alpha+1}
\left(\frac{1}{\rho^{\frac{\alpha+1}{2}}}
-\frac{1}{\rho_+^{\frac{\alpha+1}{2}}}
\right),
\qquad \rho>\rho_+ \qquad v = v_+.
\tag{$IR_1$}
\end{equation}
Similarly, the states that can be connected to $U_-$ by a $3$-rarefaction lie on the curve
\begin{equation}
R_3(U_-):\quad
u=u_--\frac{2\sqrt{\alpha A(v_-)}}{\alpha+1}\left(\frac{1}{\rho^{\frac{\alpha+1}{2}}}
-\frac{1}{\rho_-^{\frac{\alpha+1}{2}}}\right),\qquad\rho>\rho_- \qquad v = v_-.\tag{$R3$}
\end{equation}
For later use, we also introduce the inverse $3$-rarefaction curve. Fixing a right state $U_+$, we define $IR_3(U_+)=\{V: \ U_+\in R_3(V)\}$, that is, the set of all left states that can be connected to $U_+$ by an admissible $3$-rarefaction. Explicitly,
\begin{equation}
IR_3(U_+):\quad
u=u_+-\frac{2\sqrt{\alpha A(v_+)}}{\alpha+1}\left(\frac{1}{\rho^{\frac{\alpha+1}{2}}}
-\frac{1}{\rho_+^{\frac{\alpha+1}{2}}}\right),\qquad 0<\rho<\rho_+\tag{$IR_3$}
\end{equation}

\subsection{Contact Discontinuities}
The second characteristic family is linearly degenerate. Therefore, the corresponding elementary waves are contact discontinuities. Imposing the Rankine--Hugoniot condition $\lambda_2(U_-)=s$
yields $u=u_-,$ and $A(v)=\left(\frac{\rho}{\rho_-}\right)^{\alpha}A(v_-).$ Hence, the states that can be connected to $U_-$ by a contact discontinuity lie on the curve
\begin{equation}
C_2(U_-):\quad
u=u_-,
\qquad
A(v)=\left(\frac{\rho}{\rho_-}\right)^{\alpha}A(v_-).
\tag{CD}
\end{equation}

\begin{figure}
    \centering
    \includegraphics[width=0.5\linewidth]{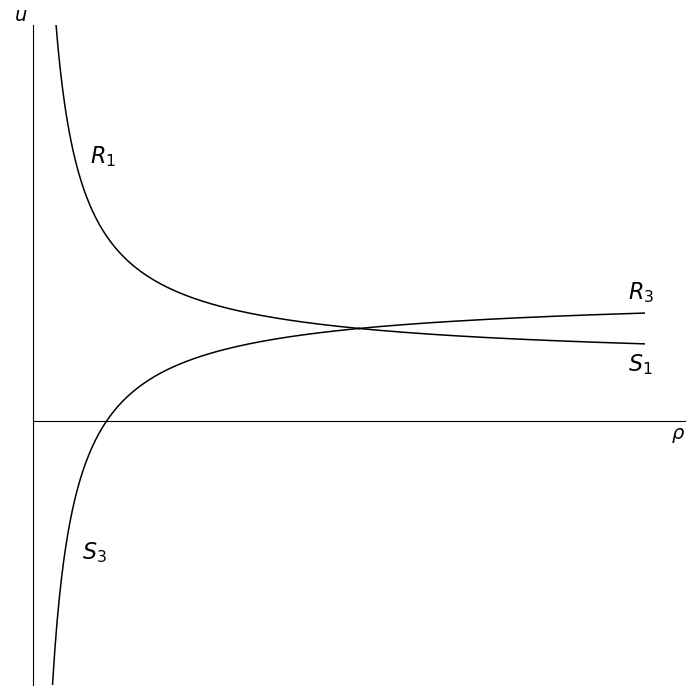}
    \caption{Graphical depiction of shock and rarefaction curve in the $\rho$-$u$ plane for fixed $v$, $A(v)$, and $\alpha$.}
    \label{fig:curves}
\end{figure}

\section{Delta Shocks and Shadow Waves} \label{sec:singular}
For a given left state $(\rho_-, u_-, v_-)$, there are regions of the $(\rho, u, v)$-space that cannot be reached by a combination of shock waves, rarefaction waves, and contact discontinuities. In these regions, we seek overcompressive singular solutions. In this section, we study these solutions using two complementary approaches. We first construct $\delta$-shock solutions in the sense of distributions and then consider the corresponding shadow-wave formulation introduced by Nedeljkov \cite{MarkoShadow}. The two approaches provide equivalent descriptions of the same singular solution. Their existence is established in Section~\ref{sec:GSPT} by means of a Dafermos regularization and spherical blow-up analysis.

We first define a weighted Dirac delta distribution supported on a smooth curve $S=\{(x(s),t(s)):a\leq s\leq b\}$
by
\begin{equation}
\left\langle \omega(\cdot)\delta(S),\varphi\right\rangle
=\int_0^\infty
\omega(t(s))\,\varphi(x(s),t(s))\,ds,\end{equation}
for all test functions
$\varphi\in C_c^\infty(\mathbb R\times(0,\infty)).$

Let $\Omega_-$ and $\Omega_+$ denote the regions to the left and right of the smooth curve $S: x=x(t)$, respectively. Following \cite{Cheng} and motivated by numerical evidence indicating a Dirac delta concentration in $\rho$ alone, we seek a $\delta$-shock solution of system
\eqref{cons1}--\eqref{cons3} of the form
\begin{equation}\label{eq:delta_sys_def}
\left\{
\begin{aligned}
&\rho(x,t)=\rho_0(x, t)+w(t)\,\delta\bigl(x-x(t)\bigr), \\[2mm]
&(u,v)(x,t)
=
\begin{cases}
(u_-,v_-)(x,t), & (x, t)\in \Omega_-,\\[1mm]
(u_+,v_+)(x,t), & (x, t)\in\Omega_+,
\end{cases}
\end{aligned}
\right.
\end{equation}
where
\begin{align}
    \rho_0(x, t)=
\begin{cases}
\rho_-(x,t), & (x, t)\in \Omega_-,\\[1mm]
\rho_+(x,t), & (x, t)\in \Omega_+.
\end{cases}
\end{align}
and the quantities $u_\delta$ and $v_\delta$ denote the velocity and transported state of the singular measure concentrated on the curve $x=x(t)$, respectively. The above distributions are required to satisfy the system
\eqref{cons1}--\eqref{cons3} in the sense of distributions, that is, 
\begin{align}
\left\langle \rho,\varphi_t\right\rangle
+\left\langle \rho u,\varphi_x\right\rangle
&=0, \label{weak1}\\
\left\langle \rho u,\varphi_t\right\rangle
+\left\langle \rho u^2-\frac{A(v)}{\rho^\alpha},\varphi_x\right\rangle
&=0, \label{weak2}\\
\left\langle \rho v,\varphi_t\right\rangle
+\left\langle \rho u v,\varphi_x\right\rangle
&=0, \label{weak3}
\end{align} for all test functions $\varphi\in C_c^\infty(\mathbb R\times(0,\infty)).$ Since the pressure term contains a negative power of the density, the expression $\frac{A(v)}{\rho^\alpha}$
is not interpreted by applying the nonlinear function directly to the measure-valued density. Instead, it is understood through regular approximations of the concentrated density. If \(\rho_\varepsilon=O(\varepsilon^{-1})\) in a layer of width \(O(\varepsilon)\) and \(v_\varepsilon=O(1)\), then
\[
\frac{A(v_\varepsilon)}{\rho_\varepsilon^\alpha}=O(\varepsilon^\alpha)\to0,
\qquad 0<\alpha\le1.
\]
Thus, the pressure contribution from the concentrated part vanishes in the distributional limit. Let $\frac{dx}{dt}=u_{\delta}(t)$ so that, using Green's Theorem and the compact support of $\varphi$, equations \eqref{weak1}--\eqref{weak3} become
\begin{align*}
    &\oint_{x(t)}\rho_- u_-\varphi\dd t-\oint_{x(t)}\rho_-\varphi\dd x+\oint_{-x(t)}\rho_+ u_+\varphi\dd t-\oint_{-x(t)}\rho_+\varphi\dd x+\int_0^\infty w(t)\dd \varphi=0,
    \\&\oint_{x(t)}\left(\rho_- u_-^2-\frac{A(v_-)}{\rho^\alpha_-}\right)\varphi\dd t-\oint_{x(t)}\left(\rho_-u_-\right)\varphi\dd x+\oint_{-x(t)}\left(\rho_+ u_+^2-\frac{A(v_+)}{\rho^\alpha_+}\right)\varphi\dd t\\&\gap-\oint_{-x(t)}\left(\rho_+u_+\right)\varphi\dd x+\int_0^\infty u_\delta(t)w(t)\dd\varphi=0,
    \\&\oint_{x(t)}\rho_- u_-v_-\varphi\dd t-\oint_{x(t)}\rho_-v_-\varphi\dd x+\oint_{-x(t)}\rho_+ u_+v_+\varphi\dd t-\oint_{-x(t)}\rho_+v_+\varphi\dd x+\int_0^\infty v_\delta w(t)\dd \varphi=0.
\end{align*}
Taken together, these equations yield the generalized Rankine--Hugoniot conditions
\begin{equation} \label{delta_ODEs}
\left\{
\begin{aligned}
\frac{dw(t)}{dt} &= [\rho]u_\delta(t)-[\rho u],\\[1mm]
\frac{\ d\bigl(w(t)u_\delta(t)\bigr)}{dt}
&= [\rho u]u_\delta(t)-\left[\rho u^2-\frac{A(v)}{\rho^\alpha}\right],\\[1mm]
\frac{d \bigl(w(t)v_\delta(t)\bigr)}{dt}
&= [\rho v]u_\delta(t)-[\rho u v].
\end{aligned}
\right.
\end{equation}
We now consider the corresponding shadow-wave construction of Nedeljkov \cite{Marko,Marko2,MarkoShadow,Daw-Marko,DavorMarko}. In this approach, the singular solution is obtained as the distributional limit of a family of piecewise constant approximations.
Define a shadow wave as the distributional limit, as $\eps\to0^+$, of
\begin{align}
(\rho_\eps,u_\eps,v_\eps)=
\begin{cases}
(\rho_-,u_-,v_-), & x<(c-\eps)t,\\[0.5ex]
(\rho_{-,\eps},u_{-,\eps},v_{-,\eps}), & (c-\eps)t<x<ct,\\[0.5ex]
(\rho_{+,\eps},u_{+,\eps},v_{+,\eps}), & ct<x<(c+\eps)t,\\[0.5ex]
(\rho_+,u_+,v_+), & x>(c+\eps)t,
\end{cases}
\end{align}
where $\eps>0$ and $c$ is the speed of the shadow wave. We assume that
$$
\rho_{-,\varepsilon},\,\rho_{+,\varepsilon}=O(\varepsilon^{-1}),\qquad
u_{-,\varepsilon},\,u_{+,\varepsilon}=O(1),\qquad
v_{-,\varepsilon},\,v_{+,\varepsilon}=O(1),
$$
and that the limits
\begin{align*}
    \xi_-&:=\lim_{\varepsilon\to0^+}\varepsilon\rho_{-,\varepsilon},
    &
    \xi_+&:=\lim_{\varepsilon\to0^+}\varepsilon\rho_{+,\varepsilon},\\
    \bar u_-&:=\lim_{\varepsilon\to0^+}u_{-,\varepsilon},
    &
    \bar u_+&:=\lim_{\varepsilon\to0^+}u_{+,\varepsilon},\\
    \bar v_-&:=\lim_{\varepsilon\to0^+}v_{-,\varepsilon},
    &
    \bar v_+&:=\lim_{\varepsilon\to0^+}v_{+,\varepsilon},
\end{align*}
exist.

Using $H$ to denote the Heaviside function, we may write
\begin{align*}
\rho_\eps
&={}
\rho_-\bigg(1-H(x-(c-\eps)t)\bigg)+\rho_{-,\eps}
\bigg(H(x-(c-\eps)t)-H(x-ct)\bigg)\\
&+\rho_{+,\eps}
\bigg(H(x-ct)-H(x-(c+\eps)t)\bigg)+\rho_+\bigg(H(x-(c+\eps)t)\bigg).
\end{align*}
Similar formulas hold for $u_\eps$ and $v_\eps$. Differentiating in the sense of distributions gives
\begin{align*}
\partial_t\rho_\varepsilon
={}&
(c-\varepsilon)(\rho_--\rho_{-,\varepsilon})
\delta(x-(c-\varepsilon)t)\\
&+c(\rho_{-,\varepsilon}-\rho_{+,\varepsilon})
\delta(x-ct)\\
&+(c+\varepsilon)(\rho_{+,\varepsilon}-\rho_+)
\delta(x-(c+\varepsilon)t).
\end{align*}
Hence, for every test function $\varphi\in C_c^\infty(\mathbb R\times\mathbb (0, \infty)),$
we obtain
\begin{align*}
\langle\partial_t\rho_\varepsilon,\varphi\rangle
={}&
(c-\varepsilon)\int_0^\infty
(\rho_--\rho_{-,\varepsilon})
\varphi((c-\varepsilon)t,t)\,dt\\
&+c\int_0^\infty
(\rho_{-,\varepsilon}-\rho_{+,\varepsilon})
\varphi(ct,t)\,dt\\
&+(c+\varepsilon)\int_0^\infty
(\rho_{+,\varepsilon}-\rho_+)
\varphi((c+\varepsilon)t,t)\,dt.
\end{align*}
Expanding the first and third terms about the interface $x=x(t)$ using Taylor's formula,
\[
\varphi(ct\pm\varepsilon t,t)
=
\varphi(ct,t)
\pm
\varepsilon t\,\partial_x\varphi(ct,t)+O(\varepsilon^2),
\]
gives
\begin{align*}
\left\langle\partial_t\rho_\varepsilon,\varphi\right\rangle
={}&
\int_0^\infty
\bigg(c(\rho_--\rho_+)
+\varepsilon(\rho_{-,\varepsilon}+\rho_{+,\varepsilon})\bigg)
\varphi(ct,t)\,dt+
\varepsilon
\int_0^\infty
(\rho_{-,\varepsilon}+\rho_{+,\varepsilon})
ct\,\partial_x\varphi(ct,t)\,dt
+O(\varepsilon),
\end{align*}
where we have kept only the terms that contribute in the limit $\varepsilon\to0^+$ and absorbed the remaining terms into $O(\varepsilon)$. Hence,
\begin{align*}
\left\langle\partial_t\rho_\varepsilon,\varphi\right\rangle
=
\left\langle
\big(c(\rho_--\rho_+)
+\varepsilon(\rho_{-,\varepsilon}+\rho_{+,\varepsilon})\big)\delta(x-ct)
-\varepsilon ct(\rho_{-,\varepsilon}+\rho_{+,\varepsilon})\delta'(x-ct),
\varphi
\right\rangle
+O(\varepsilon).
\end{align*}
Repeating the above calculations for the remaining conserved quantities gives
\begin{align*}
\left\langle\partial_x(\rho_\varepsilon u_\varepsilon),\varphi\right\rangle
={}&
\left\langle
(\rho_+u_+-\rho_-u_-)\delta(x-ct)
+\varepsilon t(\rho_{-,\varepsilon}u_{-,\varepsilon}
+\rho_{+,\varepsilon}u_{+,\varepsilon})
\delta'(x-ct),
\varphi
\right\rangle+O(\varepsilon),
\\
\left\langle\partial_t(\rho_\varepsilon u_\varepsilon),\varphi\right\rangle
={}&
\Big\langle
\big(
c(\rho_-u_--\rho_+u_+)
+\varepsilon(\rho_{-,\varepsilon}u_{-,\varepsilon}
+\rho_{+,\varepsilon}u_{+,\varepsilon})
\big)
\delta(x-ct)
\\
&\qquad
-\varepsilon ct
(\rho_{-,\varepsilon}u_{-,\varepsilon}
+\rho_{+,\varepsilon}u_{+,\varepsilon})
\delta'(x-ct),
\varphi
\Big\rangle+O(\varepsilon),
\\
\left\langle
\partial_x
\left(
\rho_\varepsilon u_\varepsilon^2
-\frac{A(v_\varepsilon)}{\rho_\varepsilon^\alpha}
\right),
\varphi
\right\rangle
={}&
\Big\langle
\left(
\rho_+u_+^2
-\frac{A(v_+)}{\rho_+^\alpha}
-\rho_-u_-^2
+\frac{A(v_-)}{\rho_-^\alpha}
\right)
\delta(x-ct)
\\
&\qquad
+\varepsilon t
\left(
\rho_{-,\varepsilon}u_{-,\varepsilon}^2
-\frac{A(v_{-,\varepsilon})}{\rho_{-,\varepsilon}^\alpha}
+\rho_{+,\varepsilon}u_{+,\varepsilon}^2
-\frac{A(v_{+,\varepsilon})}{\rho_{+,\varepsilon}^\alpha}
\right)
\delta'(x-ct),
\varphi
\Big\rangle+O(\varepsilon),
\end{align*}
\begin{align*}
\left\langle
\partial_t(\rho_\varepsilon v_\varepsilon),
\varphi
\right\rangle
={}&
\Big\langle
\big(
c(\rho_-v_--\rho_+v_+)
+\varepsilon(\rho_{-,\varepsilon}v_{-,\varepsilon}
+\rho_{+,\varepsilon}v_{+,\varepsilon})
\big)
\delta(x-ct)
\\
&\qquad
-\varepsilon ct
(\rho_{-,\varepsilon}v_{-,\varepsilon}
+\rho_{+,\varepsilon}v_{+,\varepsilon})
\delta'(x-ct),
\varphi
\Big\rangle+O(\varepsilon),
\\
\left\langle
\partial_x(\rho_\varepsilon u_\varepsilon v_\varepsilon),
\varphi
\right\rangle
={}&
\Big\langle
(\rho_+u_+v_+-\rho_-u_-v_-)\delta(x-ct)
+\varepsilon t
(\rho_{-,\varepsilon}u_{-,\varepsilon}v_{-,\varepsilon}
+\rho_{+,\varepsilon}u_{+,\varepsilon}v_{+,\varepsilon})
\delta'(x-ct),
\varphi
\Big\rangle+O(\varepsilon).
\end{align*}
Passing to the limit as $\varepsilon\to0^+$ and equating the coefficients of
$\delta(x-ct)$ and $\delta'(x-ct)$ in system \eqref{cons1}--\eqref{cons3}, we obtain
\begin{align}\label{shadow_system}
\begin{cases}
\xi_-+\xi_+=c[\rho]-[\rho u],\\[1mm]
\xi_-\bar u_-+\xi_+\bar u_+=c(\xi_-+\xi_+),\\[1mm]
\xi_-\bar u_-+\xi_+\bar u_+=c[\rho u]
-\left[\rho u^2-\frac{A(v)}{\rho^\alpha}\right],\\[2mm]
\xi_-\bar u_-^2+\xi_+\bar u_+^2
=c(\xi_-\bar u_-+\xi_+\bar u_+),\\[1mm]
\xi_-\bar v_-+\xi_+\bar v_+=c[\rho v]-[\rho uv],\\[1mm]
\xi_-\bar u_-\bar v_-+\xi_+\bar u_+\bar v_+
=c(\xi_-\bar v_-+\xi_+\bar v_+).
\end{cases}
\end{align}
Define
$$
k_1:=\xi_-+\xi_+,\qquad
k_2:=\xi_-\bar u_-+\xi_+\bar u_+,\qquad
k_3:=\xi_-\bar v_-+\xi_+\bar v_+.
$$
Then \eqref{shadow_system} gives
\begin{equation}
    \begin{cases}
        k_1=c[\rho]-[\rho u] \\[1ex]
        k_2=ck_1 \\[1ex]
        k_2=c[\rho u]-\left[\rho u^2-\frac{A(v)}{\rho^\alpha}\right] \\[1ex]
        k_3=c[\rho v]-[\rho u v]
    \end{cases} \label{rel_ki}
\end{equation}
Eliminating $k_1$ and $k_2$ gives
\begin{equation}\label{shadow_speed_eq}
c^2[\rho]-2c[\rho u]
+\left[\rho u^2-\frac{A(v)}{\rho^\alpha}\right]=0.
\end{equation}
Throughout this paper, $c$ denotes the propagation speed of the
singular solution (equivalently, the corresponding shadow wave)
connecting the left and right states $U_-$ and $U_+$. When it is
helpful to emphasize the dependence on the end states, we write
$$
c=s_{\mathrm{SDW}}(U_-,U_+)=s_{\mathrm{SDW}}(\rho_-,u_-,v_-,\rho_+,u_+,v_+).
$$
If $[\rho]\neq0$, it follows from \eqref{shadow_speed_eq} that
$$
c=\frac{[\rho u]\pm\sqrt{\Delta}}{[\rho]}=:\sigma_\pm,
$$
where
$$\Delta=[\rho u]^2-[\rho]\left[\rho u^2-\frac{A(v)}{\rho^\alpha}\right].$$
If $[\rho]=0$, then
$$c=\frac{\left[\rho u^2-\frac{A(v)}{\rho^\alpha}\right]}{2[\rho u]}=:\sigma_0.$$
The total mass contained in the two intermediate layers is
$$\int_{(c-\varepsilon)t}^{ct}\rho_{-,\varepsilon}\,dx
+\int_{ct}^{(c+\varepsilon)t}\rho_{+,\varepsilon}\,dx
=\varepsilon t(\rho_{-,\varepsilon}+\rho_{+,\varepsilon}).
$$
Passing to the limit as $\varepsilon\to0^+$ gives
$$w(t)=t(\xi_-+\xi_+)=tk_1.$$
Thus, $k_1$ is the mass accumulation rate at the singular interface. 
Since
$k_1=c[\rho]-[\rho u],$
we obtain
$$
k_1=
\begin{cases}
\sqrt{\Delta},& c=\sigma_+,\\
-\sqrt{\Delta},& c=\sigma_-.
\end{cases}
$$
Therefore, $k_1>0$ selects $\sigma_+$ (and $\sigma_0$ when $[\rho]=0$).

We now show that the aggregate quantities associated with the shadow wave satisfy the generalized Rankine--Hugoniot conditions \eqref{delta_ODEs} derived for the $\delta$-shock. Since the shadow wave propagates with constant speed $c$, we have
$\frac{dx}{dt}=c.$
Moreover, using \eqref{rel_ki}, the system \eqref{delta_ODEs} becomes
$$
\frac{dw}{dt}=k_1,\qquad
\frac{d(wu_\delta)}{dt}=k_2,\qquad
\frac{d(wv_\delta)}{dt}=k_3.
$$
Since
$\frac{dw}{dt}=k_1$
and $w(0)=0$, we recover $w(t)=tk_1,$ in agreement with the mass contained in the shadow wave. Because $u_\delta$ and $v_\delta$ are constant along the interface, $u_\delta=\frac{k_2}{k_1},\
v_\delta=\frac{k_3}{k_1}.$
Using the shadow-wave identities $k_2=ck_1$ and $k_3=\xi_-\bar v_-+\xi_+\bar v_+$,
we conclude that
$$
u_\delta=c,\qquad
v_\delta=\frac{k_3}{k_1}
=\frac{\xi_-\bar v_-+\xi_+\bar v_+}{\xi_-+\xi_+}.
$$
Thus, the aggregate quantities obtained from the shadow-wave construction satisfy exactly the generalized Rankine--Hugoniot conditions governing the corresponding $\delta$-shock. These identities will be used in Section~\ref{sec:GSPT} to identify the distributional limit of the Dafermos regularization.

The shadow-wave construction provides a refinement of the $\delta$-shock description. Although generalized Rankine--Hugoniot conditions determine only aggregate quantities $w$, $u_\delta$, and $v_\delta$, the shadow-wave formulation additionally specifies how the concentrated mass is distributed between the left and right intermediate states through $(\xi_-,\xi_+,\bar u_-,\bar u_+,\bar v_-,\bar v_+)$. Consequently, every admissible shadow wave determines a corresponding $\delta$-shock. Conversely, a $\delta$-shock specifies only the aggregate quantities $w, u_\delta, v_\delta$, and therefore does not uniquely determine the intermediate shadow-wave states.
\subsection{Overcompressive $\delta$-Shocks}\label{sec:singular_overcomp}
A $\delta$-shock (equivalently, its corresponding shadow wave) is said to be overcompressive if its propagation speed $c$ satisfies the overcompressibility condition:
\begin{align}\label{overcomp_general}
    u_--\sqrt{\frac{\alpha A(v_-)}{\rho_-^{\alpha+1}}}=\lambda_1(U_-)\geq c \geq \lambda_3(U_+)=u_++\sqrt{\frac{\alpha A(v_+)}{\rho_+^{\alpha+1}}}
\end{align}
The overcompressibility condition provides an admissibility criterion for singular solutions. However, as illustrated in \cite{MarkoShadow}, it is not always sufficient to select a unique admissible singular solution. This motivates the use of additional admissibility criteria, such as the Dafermos maximum entropy dissipation principle considered in Section~\ref{sec:combo}.

Consider first the case where $[\rho]\neq0$; then there exists a $\delta$-wave with speed $\sigma_+ =\frac{[\rho u]+\sqrt{\Delta}}{[\rho]}$. Then the first inequality in \eqref{overcomp_general} becomes
\begin{align}
    u_--\sqrt{\frac{\alpha A(v_-)}{\rho_-^{\alpha+1}}}&\geq \frac{[\rho u]+\sqrt{\Delta}}{[\rho]} \label{first_overcomp}
\end{align}
Using the simplification
\begin{align}
    \frac{[\rho u]}{[\rho]}  = \frac{\rho_+ u_+ - \rho_+ u_- + \rho_+u_- - \rho_- u_-}{[\rho]} = \frac{\rho_+[u]}{[\rho]}+ \frac{u_-[\rho]}{[\rho]} = u_- + \rho_+ \frac{[u]}{[\rho]} \label{overcomp_simp_left}
\end{align}
we can rewrite \eqref{first_overcomp} as
\begin{align}
    -\sqrt{\frac{\alpha A(v_-)}{\rho_-^{\alpha+1}}}&\geq\rho_+\frac{[u]}{[\rho]}+\frac{\sqrt{\Delta}}{[\rho]}\label{first_overcomp_result}
\end{align}
The second inequality in \eqref{overcomp_general} becomes
\begin{align}
    u_++\sqrt{\frac{\alpha A(v_+)}{\rho_+^{\alpha+1}}}&\leq \frac{[\rho u]+\sqrt{\Delta}}{[\rho]} \label{second_overcomp}
\end{align}
Using a simplification similar to \eqref{overcomp_simp_left}, we have 
\begin{align}
    \frac{[\rho u]}{[\rho]}  = \frac{\rho_+ u_+ - \rho_- u_+ + \rho_-u_+ - \rho_- u_-}{[\rho]} = \frac{u_+[\rho]}{[\rho]} + \frac{\rho_-[u]}{[\rho]}  = u_+ + \rho_- \frac{[u]}{[\rho]} \label{overcomp_simp_right}
\end{align}
which allows us to rewrite \eqref{second_overcomp} as
\begin{align}
    \sqrt{\frac{\alpha A(v_+)}{\rho_+^{\alpha+1}}}&\leq\rho_-\frac{[u]}{[\rho]}+\frac{\sqrt{\Delta}}{[\rho]}\label{second_overcomp_result}
\end{align}
If $[\rho]>0$, then \eqref{first_overcomp_result} implies $[u]<0$ since we require the right-hand side to be strictly negative. If instead $[\rho]<0$, then \eqref{second_overcomp_result} likewise implies $[u]<0$. Direct computation shows that \eqref{first_overcomp_result} is satisfied exactly when 
\begin{equation}\label{eq:overcomp__rhopositive_contradiction_1}
    u_+\geq u_--\sqrt{\frac{\alpha A(v_-)}{\rho_-^{\alpha+1}}}+\frac{1}{\sqrt{\rho_+}}\sqrt{\frac{(\alpha-1)A(v_-)}{\rho_-^\alpha}+\frac{A(v_+)}{\rho_+^\alpha}}
\end{equation}
or 
\begin{equation} \label{eq:overcomp__rhopositive_result_1}
    u_+\leq u_--\sqrt{\frac{\alpha A(v_-)}{\rho_-^{\alpha+1}}}-\frac{1}{\sqrt{\rho_+}}\sqrt{\frac{(\alpha-1)A(v_-)}{\rho_-^\alpha}+\frac{A(v_+)}{\rho_+^\alpha}}
\end{equation}
and \eqref{second_overcomp_result} is satisfied exactly when 
\begin{equation} \label{eq:overcomp__rhopositive_contradiction_2}
u_+\geq u_--\sqrt{\frac{\alpha A(v_+)}{\rho_+^{\alpha+1}}}+\frac{1}{\sqrt{\rho_-}}\sqrt{\frac{(\alpha-1)A(v_+)}{\rho_+^\alpha}+\frac{A(v_-)}{\rho_-^\alpha}}
\end{equation}
or
\begin{equation} \label{eq:overcomp__rhopositive_result_2}
u_+\leq u_--\sqrt{\frac{\alpha A(v_+)}{\rho_+^{\alpha+1}}}-\frac{1}{\sqrt{\rho_-}}\sqrt{\frac{(\alpha-1)A(v_+)}{\rho_+^\alpha}+\frac{A(v_-)}{\rho_-^\alpha}}
\end{equation}
However, inequalities \eqref{eq:overcomp__rhopositive_contradiction_1} and \eqref{eq:overcomp__rhopositive_contradiction_2} are incompatible with inequality $\lambda_1(U_-)\geq\lambda_3(U_+)$, which is necessary for \eqref{overcomp_general} to hold. Therefore, inequalities \eqref{eq:overcomp__rhopositive_result_1} and \eqref{eq:overcomp__rhopositive_result_2} characterize the overcompressive region where $\delta$-wave solutions with speed $\sigma_+$ are admissible. In the case $[\rho]=0$, a similar line of computation demonstrates that \eqref{eq:overcomp__rhopositive_result_1} and \eqref{eq:overcomp__rhopositive_result_2} also characterize the overcompressive region for $\delta$-waves of speed $\sigma_0$. 
\subsubsection{The Overcompressive Region}
Rather than expressing the overcompressive region as the intersection of two inequalities, \eqref{eq:overcomp__rhopositive_result_1} and \eqref{eq:overcomp__rhopositive_result_2}, a simpler yet equivalent characterization expresses it as a single inequality in $u$ on either side of a boundary surface depending on $\rho$ and $A(v)$. With that motivation, the remainder of this section is devoted to demonstrating that the overcompressive region can be expressed as 
\begin{equation}
    u\leq u_--\sqrt{\frac{\alpha A(v_-)}{\rho_-^{\alpha+1}}}-\frac{1}{\sqrt{\rho}}\sqrt{\frac{(\alpha-1)A(v_-)}{\rho_-^\alpha}+\frac{A(v)}{\rho^\alpha}}~,~\rho<\left(\frac{A(v)}{A(v_-)}\right)^\frac{1}{\alpha}\rho_-\label{eq:OC_Expression_1},
\end{equation}
\begin{equation}
u\leq u_--\sqrt{\frac{\alpha A(v)}{\rho^{\alpha+1}}}-\frac{1}{\sqrt{\rho_-}}\sqrt{\frac{(\alpha-1)A(v)}{\rho^\alpha}+\frac{A(v_-)}{\rho_-^\alpha}}~,~\rho\geq\left(\frac{A(v)}{A(v_-)}\right)^\frac{1}{\alpha}\rho_-\label{eq:OC_Expression_2}.
\end{equation}
Consider first the interval $0<\rho<\left(\frac{A(v)}{A(v_-)}\right)^\frac{1}{\alpha}\rho_-$ for $\delta$-shocks of speed $\sigma_+$. By \eqref{eq:overcomp__rhopositive_result_1}, $c\leq\lambda_1(U_-)$ for any $U$ with $u\leq u_--\sqrt{\frac{\alpha A(v_-)}{\rho_-^{\alpha+1}}}-\frac{1}{\sqrt{\rho}}\sqrt{\frac{(\alpha-1)A(v_-)}{\rho_-^\alpha}+\frac{A(v)}{\rho^\alpha}}$. Therefore, to show overcompressibility, it suffices to show $c\geq\lambda_3(U)$ for the desired interval. It can be shown by direct computation that 
\begin{equation*}
u_--\sqrt{\frac{\alpha A(v_-)}{\rho_-^{\alpha+1}}}-\frac{1}{\sqrt{\rho}}\sqrt{\frac{(\alpha-1)A(v_-)}{\rho_-^\alpha}+\frac{A(v)}{\rho^\alpha}}\leq u_--\sqrt{\frac{\alpha A(v)}{\rho^{\alpha+1}}}-\frac{1}{\sqrt{\rho_-}}\sqrt{\frac{(\alpha-1)A(v)}{\rho^\alpha}+\frac{A(v_-)}{\rho_-^\alpha}},
\end{equation*}
for all $\rho$ with $\left(\frac{(1-\alpha)A(v)}{A(v_-)}\right)^\frac{1}{\alpha}\rho_-\leq\rho<\left(\frac{A(v)}{A(v_-)}\right)^\frac{1}{\alpha}\rho_-$. In this region, \eqref{eq:overcomp__rhopositive_result_2} gives $c\geq\lambda_3(U)$. Now, consider $\rho<\left(\frac{(1-\alpha)A(v)}{A(v_-)}\right)^\frac{1}{\alpha}\rho_-$. 
\newline Let 
\begin{equation*}
\Phi(U)=c-\lambda_3(U),
\end{equation*}
so that
\begin{equation}
\Phi(U)=\frac{[\rho u]+\sqrt{\Delta}}{[\rho]}-u-\sqrt{\frac{\alpha A(v)}{\rho^{\alpha+1}}}.
\end{equation}
we can now rewrite $\Phi$ using the identity $\Delta=\rho\rho_-[u]^2+[\rho][\frac{A(v)}{\rho^\alpha}]$, giving us 
\begin{equation*}
\Phi(U)=\frac{(\rho u-\rho_-u_-)+\sqrt{\rho\rho_-(u-u_-)^2+(\rho-\rho_-)(\frac{A(v)}{\rho^\alpha}-\frac{A(v_-)}{\rho_-^\alpha})}}{\rho-\rho_-}-u-\sqrt{\frac{\alpha A(v)}{\rho^{\alpha+1}}}
\end{equation*}
Now, we will prove $\Phi$ is positive in the desired region. 
First, note that
\begin{equation}
    \partial_u\Phi=\frac{\rho_-}{[\rho]}\left(1+\frac{\rho(u-u_-)}{\sqrt{\Delta}}\right)\label{eq:partial_u_Phi}.
\end{equation}
Setting \eqref{eq:partial_u_Phi} equal to 0, we obtain the critical boundaries 
\begin{equation*}
    u(\rho,v)=u_-\pm\frac{1}{\sqrt{\rho}}\sqrt{\left[\frac{A(v)}{\rho^\alpha}\right]}.
\end{equation*}
As we require $[u]<0$, we choose the minus branch
\begin{equation}
    u_b(\rho,v)=u_--\frac{1}{\sqrt{\rho}}\sqrt{\left[\frac{A(v)}{\rho^\alpha}\right]}.
\end{equation}
The sign of $\Phi$ on this boundary surface can be determined
\begin{equation*}
\Phi(\rho,u_b(\rho,v),v)=-\frac{\rho_-}{\rho-\rho_-}\left(\frac{1}{\sqrt{\rho}}\sqrt{\left[\frac{A(v)}{\rho^\alpha}\right]}\right)+\frac{\sqrt{\rho\left[\frac{A(v)}{\rho^\alpha}\right]}}{\rho-\rho_-}-\frac{1}{\sqrt{\rho}}\sqrt{\frac{\alpha A(v)}{\rho^\alpha}}
\newline =\frac{1}{\sqrt{\rho}}\left(\sqrt{\left[\frac{A(v)}{\rho^\alpha}\right]}-\sqrt{\frac{\alpha A(v)}{\rho^\alpha}}\right)>0,
\end{equation*}
for $\rho<\left(\frac{(1-\alpha)A(v)}{A(v_-)}\right)^\frac{1}{\alpha}\rho_-$. All that remains now is to determine the sign of $\Phi$ for values of $u$ under the boundary. Taking $u\to-\infty$, 
\begin{equation} 
\Phi(U)\sim\frac{\rho-\sqrt{\rho\rho_-}}{\rho-\rho_-}u-u=\frac{\sqrt{\rho}(\sqrt{\rho}-\sqrt{\rho_-})}{(\sqrt{\rho}-\sqrt{\rho_-})(\sqrt{\rho}+\sqrt{\rho_-})}u-u=\left(\frac{\sqrt{\rho}}{\sqrt{\rho_-}+\sqrt{\rho}}-1\right)u\to\infty.
\end{equation}
This limit tells us that, for $u<u_b$, $\Phi$ is monotonically decreasing in $u$. Noting that 
\begin{equation}
u_--\sqrt{\frac{\alpha A(v_-)}{\rho_-^{\alpha+1}}}-\frac{1}{\sqrt{\rho}}\sqrt{\frac{(\alpha-1)A(v_-)}{\rho_-^\alpha}+\frac{A(v)}{\rho^\alpha}}<u_b,
\end{equation}
we have 
\begin{equation}
u\leq u_--\sqrt{\frac{\alpha A(v_-)}{\rho_-^{\alpha+1}}}-\frac{1}{\sqrt{\rho}}\sqrt{\frac{(\alpha-1)A(v_-)}{\rho_-^\alpha}+\frac{A(v)}{\rho^\alpha}}<u_b
\implies \Phi(\rho,u,v)>\Phi(\rho,u_b,v)>0
\end{equation}
on $\rho<\left(\frac{(1-\alpha)A(v)}{A(v_-)}\right)^\frac{1}{\alpha}\rho_L$. Therefore, 
\begin{equation}\lambda_3(U)\leq c\leq\lambda_1(U_-)\iff u\leq u_--\sqrt{\frac{\alpha A(v_-)}{\rho_L^{\alpha+1}}}-\frac{1}{\sqrt{\rho}}\sqrt{\frac{(\alpha-1)A(v_-)}{\rho_L^\alpha}+\frac{A(v)}{\rho^\alpha}}~,~\rho<\left(\frac{A(v)}{A(v_-)}\right)^\frac{1}{\alpha}\rho_L
\end{equation} proving \eqref{eq:OC_Expression_1}.
For $\delta$-waves of speed $\sigma_+$, \eqref{eq:OC_Expression_2} can be justified in a similar manner. In the case $[\rho]=0$, an analogous line of computation produces the same results. 

\section{Wave Combinations and Dissipation} \label{sec:combo}

\subsection{Solutions to the Riemann Problem} \label{subsec:combo_analytic}
In this section, we give the necessary and sufficient conditions for the existence and uniqueness of Riemann solutions connecting prescribed left and right states $U_L$ and $U_R$. These solutions consist of combinations of shocks, rarefaction waves, contact discontinuities, and, where appropriate, singular waves.

Throughout this section, we denote the shock and rarefaction curves of the $i$-th characteristic family by $S_i$ and $R_i$, respectively. Their inverse curves are denoted by $IS_i$ and $IR_i$, and the contact discontinuity associated with the second characteristic family is denoted by $CD$.

\subsubsection{$R_1 + CD + R_3$}
For this case we seek $U_{M_1},~U_{M_2}$ such that $U_{M_1}\in R_1(U_L),~U_{M_2}\in IR_3(U_R).$ In other words, we require $\rho_{M_1}<\rho_L,~\rho_{M_2}<\rho_R,$
\begin{equation}
\begin{aligned}
u_{M_1}
=
u_L
+\frac{2\sqrt{\alpha A(v_L)}}{\alpha+1}
\left(
\frac{1}{\rho_{M_1}^{\frac{\alpha+1}{2}}}
-
\frac{1}{\rho_L^{\frac{\alpha+1}{2}}}
\right), \qquad v_{M_1} = v_L.
\end{aligned}
\end{equation}
and that 
\begin{equation}
\begin{aligned}
u_{M_2}
=
u_R
-\frac{2\sqrt{\alpha A(v_{R})}}{\alpha+1}
\left(
\frac{1}{\rho_{M_2}^{\frac{\alpha+1}{2}}}
-
\frac{1}{\rho_R^{\frac{\alpha+1}{2}}}
\right), \qquad v_R = v_{M_2}.
\end{aligned}
\end{equation}
Due to the contact discontinuity, we further require 
\begin{equation}
\rho_{M_1}
=
\left(\frac{A(v_{M_1})}{A(v_{M_2})}\right)^{\frac{1}{\alpha}}
\rho_{M_2}, \qquad u_{M_2}=u_{M_1}. \label{classical_CD_condition}
\end{equation}
This gives us another inequality
$$\rho_{M_2} < \left(\frac{A(v_{M_2})}{A(v_{M_1})} \right)^{\frac{1}{\alpha}} \rho_L$$
which gives us a permitted interval of $\rho_{M_2}$, that is
\begin{equation}
    \rho_{M_2} \in (0, \rho_{\text{min}}) \label{rho_interval_R1_R3}
\end{equation}
where
$$\rho_{\text{min}} = \min\left\{\rho_R,
\left(\frac{A(v_{M_2})}{A(v_{M_1})}\right)^{\frac{1}{\alpha}}\rho_L\right\}$$
Combining these conditions gives us 
\begin{equation}
\begin{aligned}
u_L
+\frac{2\sqrt{\alpha A(v_L)}}{\alpha+1}
\left(
\frac{1}{\left(\frac{A(v_L)}{A(v_R)}\right)^{\frac{\alpha+1}{2\alpha}}
\rho_{M_2}^{\frac{\alpha+1}{2}}}
-
\frac{1}{\rho_L^{\frac{\alpha+1}{2}}}
\right)
=
u_R
-\frac{2\sqrt{\alpha A(v_{R})}}{\alpha+1}
\left(
\frac{1}{\rho_{M_2}^{\frac{\alpha+1}{2}}}
-
\frac{1}{\rho_R^{\frac{\alpha+1}{2}}}
\right).
\end{aligned}
\end{equation}
Consider 
\begin{equation}
\begin{aligned}
f(\rho_{M_2})
=
u_R-u_L
-\frac{2\sqrt{\alpha A(v_R)}}{\alpha+1}
\left(
\frac{1}{\rho_{M_2}^{\frac{\alpha+1}{2}}}
-
\frac{1}{\rho_R^{\frac{\alpha+1}{2}}}
\right)
-
\frac{2\sqrt{\alpha A(v_L)}}{\alpha+1}
\left(
\frac{1}{\left(\frac{A(v_L)}{A(v_R)}\right)^{\frac{\alpha+1}{2\alpha}}
\rho_{M_2}^{\frac{\alpha+1}{2}}}
-
\frac{1}{\rho_L^{\frac{\alpha+1}{2}}}
\right) \label{f_R1_R3}.
\end{aligned}
\end{equation}
We can verify that for values of $\rho_{M_2}$ in the interval \eqref{rho_interval_R1_R3}, we get $f' > 0$. Hence, for a solution to exist and be unique in the admissible interval, it is necessary and sufficient to have $\lim_{\rho_{M_2} \to 0}f(\rho_{M_2}) < 0$ and $f(\rho_{\text{min}}) > 0$. Trivially, $f(\rho_{M_2})\to-\infty$ as $\rho_{M_2}\to 0$. Substituting $\rho_{\text{min}}$ into \eqref{f_R1_R3}, we obtain the following desired inequality
\begin{equation}
\begin{aligned}
u_R
>
u_L
+\frac{2\sqrt{\alpha A(v_R)}}{\alpha+1}
\left(
\frac{1}{\rho_{\text{min}}^{\frac{\alpha+1}{2}}}
-
\frac{1}{\rho_R^{\frac{\alpha+1}{2}}}
\right)
+
\frac{2\sqrt{\alpha A(v_L)}}{\alpha+1}
\left(
\frac{1}{\left(\frac{A(v_L)}{A(v_R)}\right)^{\frac{\alpha+1}{2\alpha}}
\rho_{\text{min}}^{\frac{\alpha+1}{2}}}
-
\frac{1}{\rho_L^{\frac{\alpha+1}{2}}}
\right).
\end{aligned}
\end{equation}
If this inequality holds, then by the Intermediate Value Theorem, the function $f$ has a unique zero in the interval \eqref{rho_interval_R1_R3}, and hence there exists a unique intermediate state.

\subsubsection{$R_1+CD+S_3$}

For this case, we seek $U_{M_1}$, $U_{M_2}$ such that $U_{M_1}\in R_1(U_L), ~U_{M_2}\in IS_3(U_R).$ In other words, we require $\rho_{M_1} < \rho_L, ~\rho_{M_2} > \rho_R.$
\begin{equation}
    u_{M_1} = u_L + \frac{2\sqrt{\alpha A(v_L)}}{\alpha + 1} \Bigg(\frac{1}{\rho_{M_1}^{\frac{\alpha + 1}{2}}} - \frac{1}{\rho_L^{\frac{\alpha + 1}{2}}} \Bigg), \qquad v_{M_1} = v_L.
\end{equation}
and that
\begin{equation}
    u_{M_2} = u_R +\sqrt{A(v_R)\left(\frac{1}{\rho_{M_2}}-\frac{1}{\rho_R}\right)\left(\frac{1}{\rho_{M_2}^\alpha}-\frac{1}{\rho_R^\alpha}\right)}, \ \ v_{M_2}=v_R.
\end{equation}
Due to the contact discontinuity, \eqref{classical_CD_condition} still holds. This then gives us a permitted interval that $\rho_{M_2}$ can take values from, that is,
\begin{equation}
    \rho_{M_2} \in \left(\rho_R,  \bigg(\frac{A(v_{M_2})}{A(v_{M_1})}\bigg)^{\frac{1}{\alpha}} \rho_L \right) \label{rho_interval_R1_S3}
\end{equation}
Combining these conditions give us
\begin{equation}
    u_L + \frac{2\sqrt{\alpha A(v_L)}}{\alpha + 1} \bigg(\frac{1}{\left(\frac{A(v_L)}{A(v_R)}\right)^{\frac{\alpha + 1}{2\alpha}}\rho_{M_2}^{\frac{\alpha + 1}{2}}} - \frac{1}{\rho_L^{\frac{\alpha + 1}{2}}} \bigg) = u_R +\sqrt{A(v_R)\left(\frac{1}{\rho_{M_2}}-\frac{1}{\rho_R}\right)\left(\frac{1}{\rho_{M_2}^\alpha}-\frac{1}{\rho_R^\alpha}\right)}.
\end{equation}
Consider
\begin{equation}
    f(\rho_{M_2}) = u_R-u_L +\sqrt{A(v_R)\left(\frac{1}{\rho_{M_2}} -\frac{1}{\rho_R}\right)\left(\frac{1}{\rho_{M_2}^\alpha}-\frac{1}{\rho_R^\alpha}\right)} - \frac{2\sqrt{\alpha A(v_L)}}{\alpha+1}\left(\frac{1}{\left(\frac{A(v_L)}{A(v_R)}\right)^\frac{\alpha+1}{2\alpha}\rho_{M_2}^\frac{\alpha+1}{2}}-\frac{1}{\rho_L^\frac{\alpha+1}{2}}\right) \label{f_R1_S3}.
\end{equation}
We can verify that for values of $\rho_{M_2}$ in the interval \eqref{rho_interval_R1_S3}, we get $f'>0$. Hence, for a solution to exist and be unique in the admissible interval, it is necessary and sufficient to have $f(\rho_R) < 0$ and $f\left( \left(\frac{A(v_R)}{A(v_L)} \right)^{\frac{1}{\alpha}} \rho_L \right) > 0$. Substituting the endpoint values of $\rho_{M_2}$ into \eqref{f_R1_S3}, we obtain the following desired inequalities
\begin{equation}
\begin{aligned}
u_R < u_L + \frac{2\sqrt{\alpha A(v_L)}}{\alpha+1}
\left(\frac{1}{\left(\frac{A(v_L)}{A(v_R)}\right)^{\frac{\alpha+1}{2\alpha}}\rho_R^{\frac{\alpha+1}{2}}} - \frac{1}{\rho_L^{\frac{\alpha+1}{2}}}\right).
\end{aligned}
\end{equation}
and
\begin{equation}
\begin{aligned}
u_R>u_L-\sqrt{A(v_R)\bigg(\frac{1}{\big(\frac{A(v_R)}{A(v_L)}\big)^\frac{1}{\alpha}\rho_L}-\frac{1}{\rho_R}\bigg)\bigg(\frac{1}{\big(\frac{A(v_R)}{A(v_L)}\big)\rho_L^\alpha}-\frac{1}{\rho_R^\alpha}\bigg)}.
\end{aligned}
\end{equation}
If both inequalities hold, then by the Intermediate Value Theorem, the function $f$ has a unique zero in the interval \eqref{rho_interval_R1_S3}, and hence there exists a unique intermediate state. 

\subsubsection{$S_1+CD+R_3$}
For this case, we seek $U_{M_1}$, $U_{M_2}$ such that $U_{M_1}\in S_1(U_L), ~U_{M_2}\in IR_3(U_R).$ In other words, we require
$\rho_L < \rho_{M_1}, ~\rho_{M_2} < \rho_R.$
\begin{equation}
    u_{M_1} = u_L-\sqrt{A(v_L)\left(\frac{1}{\rho_{M_1}}-\frac{1}{\rho_L}\right)\left(\frac{1}{\rho_{M_1}^\alpha}-\frac{1}{\rho_L^\alpha}\right)}, \qquad v_{M_1} =v_L.
\end{equation}
and that
\begin{equation}
    u_{M_2}
=
u_R
-\frac{2\sqrt{\alpha A(v_{R})}}{\alpha+1}
\left(
\frac{1}{\rho_{M_2}^{\frac{\alpha+1}{2}}}
-
\frac{1}{\rho_R^{\frac{\alpha+1}{2}}}
\right), \qquad v_R=v_{M_2}.
\end{equation}
Due to the contact discontinuity, \eqref{classical_CD_condition} still holds. This then gives us a permitted interval that $\rho_{M_2}$ can take values from, that is,
\begin{equation}
    \rho_{M_2} \in \left(\bigg(\frac{A(v_{M_2})}{A(v_{M_1})}\bigg)^{\frac{1}{\alpha}} \rho_L, \rho_R \right) \label{rho_interval_S1_R3}
\end{equation}
Combining these conditions give us
\begin{equation}
    u_L - \sqrt{A(v_L) \left(\frac{1}{\left(\frac{A(v_L)}{A(v_R)} \right)^{\frac{1}{\alpha}}\rho_{M_2}} - \frac{1}{\rho_L} \right) \left(\frac{1}{\left(\frac{A(v_L)}{A(v_R)} \right)\rho_{M_2}^\alpha} - \frac{1}{\rho_L^\alpha} \right)} = u_R - \frac{2\sqrt{\alpha A(v_R)}}{\alpha + 1} \left( \frac{1}{\rho_{M_2}^{\frac{\alpha + 1}{2}}} - \frac{1}{\rho_R^{\frac{\alpha + 1}{2}}} \right).
\end{equation}
Consider
\begin{equation}
    f(\rho_{M_2})=u_R-u_L+\sqrt{A(v_L)\left(\frac{1}{\left(\frac{A(v_L)}{A(v_R)}\right)^\frac{1}{\alpha}\rho_{M_2}}-\frac{1}{\rho_L}\right)\left(\frac{1}{\left(\frac{A(v_L)}{A(v_R)}\right)\rho_{M_2}^\alpha}-\frac{1}{\rho_L^\alpha}\right)} \\ -  
\frac{2\sqrt{\alpha A(v_R)}}{\alpha+1}\left(\frac{1}{\rho_{M_2}^{\frac{\alpha+1}{2}}}-\frac{1}{\rho_R^\frac{\alpha+1}{2}}\right) \label{f_S1_R3}.
\end{equation}
We can verify that for values of $\rho_{M_2}$ in the interval \eqref{rho_interval_S1_R3}, we get $f' > 0$. Hence, for a solution to exist and be unique in the admissible interval, it is necessary and sufficient to have $f\left(\left(\frac{A(v_R)}{A(v_L)} \right)^{\frac{1}{\alpha}} \rho_L \right) < 0$ and $f(\rho_R) > 0$. Substituting the endpoint values of $\rho_{M_2}$ into \eqref{f_S1_R3}, we obtain the following desired inequalities
\begin{equation}
    u_R<u_L+\frac{2\sqrt{\alpha A(v_R)}}{\alpha+1}\left(\frac{1}{\left(\frac{A(v_R)}{A(v_L)}\right)^{\frac{\alpha+1}{2\alpha}}\rho_L^\frac{\alpha+1}{2}}-\frac{1}{\rho_R^{\frac{\alpha+1}{2}}}\right).
\end{equation}
and
\begin{equation}
    u_R>u_L-\sqrt{A(v_L)\left(\frac{1}{\left(\frac{A(v_L)}{A(v_R}\right)^{\frac{1}{\alpha}}\rho_R}-\frac{1}{\rho_L}\right)\left(\frac{1}{\left(\frac{A(v_L)}{A(v_R}\right)\rho_R^\alpha}-\frac{1}{\rho_L^\alpha}\right)}.
\end{equation}
If both inequalities hold, then by the Intermediate Value Theorem, the function $f$ has a unique zero in the interval \eqref{rho_interval_S1_R3}, and hence there exists a unique intermediate state.
\subsubsection{$S_1+CD+S_3$}
For this case, we seek $U_{M_1}, U_{M_2}$ such that $U_{M_1} \in S_1(U_L), ~U_{M_2}\in IS_3(U_R).$ In other words, we require
$\rho_{M_1} > \rho_L, ~\rho_{M_2} > \rho_R.$
\begin{equation}
    u_{M_1} = u_L - \sqrt{A(v_L)\left(\frac{1}{\rho_{M_1}}-\frac{1}{\rho_L}\right)\left(\frac{1}{\rho_{M_1}^\alpha}-\frac{1}{\rho_L^\alpha}\right)}, \qquad v_{M_1} =v_L.
\end{equation}
and that
\begin{equation}
    u_{M_2} = u_R + \sqrt{A(v_R)\left(\frac{1}{\rho_{M_2}}-\frac{1}{\rho_R}\right)\left(\frac{1}{\rho_{M_2}^\alpha}-\frac{1}{\rho_R^\alpha}\right)}, \qquad v_{M_2}=v_R.
\end{equation}
Due to the contact discontinuity, \eqref{classical_CD_condition} still holds. This then gives us another inequality
\[
    \rho_{M_2} > \left(\frac{A(v_{M_2})}{A(v_{M_1})} \right)^{\frac{1}{\alpha}} \rho_L
\]
which gives us a permitted interval of $\rho_{M_2}$, that is,
\begin{equation}
    \rho_{M_2} \in (\rho_{\text{max}}, \infty) \label{rho_interval_S1_S3}
\end{equation}
where $\rho_{\text{max}} = \max\left\{\rho_R,
\left(\frac{A(v_{M_2})}{A(v_{M_1})}\right)^{\frac{1}{\alpha}}\rho_L\right\}.$ Combining these conditions gives us
\begin{equation}
    u_L - \sqrt{A(v_L) \left(\frac{1}{\left(\frac{A(v_L)}{A(v_R)} \right)^{\frac{1}{\alpha}}\rho_{M_2}} - \frac{1}{\rho_L} \right) \left(\frac{1}{\left(\frac{A(v_L)}{A(v_R)}\right)\rho_{M_2}^\alpha } - \frac{1}{\rho_L^\alpha} \right)} = u_R + \sqrt{A(v_R)\left(\frac{1}{\rho_{M_2}}-\frac{1}{\rho_R}\right)\left(\frac{1}{\rho_{M_2}^\alpha}-\frac{1}{\rho_R^\alpha}\right)}.
\end{equation}
Consider
\begin{align}
    f(\rho_{M_2})=& u_R - u_L + \sqrt{A(v_R)\left(\frac{1}{\rho_{M_2}} - \frac{1}{\rho_R}\right)\left(\frac{1}{\rho_{M_2}^\alpha} - \frac{1}{\rho_R^\alpha}\right)} \\
    +& \sqrt{A(v_L)\left(\frac{1}{\left(\frac{A(v_L)}{A(v_R)}\right)^\frac{1}{\alpha}\rho_{M_2}} - \frac{1}{\rho_L}\right)\left(\frac{1}{\left(\frac{A(v_L)}{A(v_R)}\right)\rho_{M_2}^\alpha} - \frac{1}{\rho_L^\alpha}\right)} \label{f_S1_S3}.
    \end{align}
One can easily verify that for values of $\rho_{M_2}$ in the interval \eqref{rho_interval_S1_S3}, one gets $f' > 0$. Hence, for a unique solution to exist in the admissible interval, it is necessary and sufficient to have $f(\rho_{\text{max}}) < 0$ and $\lim_{\rho_{M_2} \to \infty} f(\rho_{M_2}) > 0$. Substituting the endpoint values of $\rho_{M_2}$ into \eqref{f_S1_S3}, we obtain the following desired inequalities
\begin{equation}
    u_R > u_L - \sqrt{\frac{A(v_R)}{\rho_R^{\alpha+1}}} - \sqrt{\frac{A(v_L)}{\rho_L^{\alpha+1}}}
\end{equation}
and
\begin{equation}
    u_R < u_L - \sqrt{ A(v_R) \left( \frac{1}{\rho_{\text{max}}}-\frac{1}{\rho_R} \right) \left( \frac{1}{\rho_{\text{max}}^\alpha}-\frac{1}{\rho_R^\alpha} \right)} - \sqrt{ A(v_L) \left( \frac{1}{ \left(\frac{A(v_L)}{A(v_R)}\right)^{\frac1\alpha}\rho_{\text{max}} } -\frac{1}{\rho_L} \right) \left( \frac{1} {\left(\frac{A(v_L)}{A(v_R)}\right)\rho_{\text{max}}^\alpha} -\frac{1}{\rho_L^\alpha} \right)}.
\end{equation}
If both inequalities hold, then by the Intermediate Value Theorem, the function $f$ has a unique zero in the interval \eqref{rho_interval_S1_S3}, and hence there exists a unique intermediate state.

Note that these regions are disjoint. Therefore, for a fixed left state $U_L$, each right state $U_R$ belongs to at most one region corresponding to a classical wave combination. Consequently, there exists at most one classical wave combination connecting $U_L$ to $U_R$. Furthermore, whenever such a classical solution exists, the intermediate states $U_{M_1}$ and $U_{M_2}$ are uniquely determined by the corresponding wave curves. The remaining regions of the phase space, where no classical wave combination exists, will be treated in the next section through singular solutions. The combinations
$S_1+\mathrm{SDW},\
\mathrm{SDW}+S_3,\
CD+\mathrm{SDW},\
\mathrm{SDW}+CD,$ cannot satisfy the geometric wave-curve conditions, the ordering of self-similar wave speeds, and the overcompressibility condition for the singular shock. Consequently, the only nontrivial wave combinations involving an overcompressive singular shock that remain to be considered are
$R_1+\mathrm{SDW}
\ \ \text{and}\ \ 
\mathrm{SDW}+R_3,$ together with the isolated singular shock.

The surface
\begin{equation}
u=u_L-\sqrt{\frac{A(v)}{\rho^{\alpha+1}}}
-\sqrt{\frac{A(v_L)}{\rho_L^{\alpha+1}}}
\label{eq:classical_boundary}
\end{equation}
separates regions with different types of solutions to the Riemann problem, see Figure~\ref{fig:overcompressive}. When
$u>u_L-\sqrt{\frac{A(v)}{\rho^{\alpha+1}}}
-\sqrt{\frac{A(v_L)}{\rho_L^{\alpha+1}}}$,
the Riemann problem admits classical solutions. In contrast, when
$u<u_L-\sqrt{\frac{A(v)}{\rho^{\alpha+1}}}
-\sqrt{\frac{A(v_L)}{\rho_L^{\alpha+1}}}$,
the Riemann problem admits no classical solution, and any solution must necessarily be nonclassical, namely a $\delta$-shock or, equivalently, a shadow wave. Figure~\ref{fig:overcompressive} also shows the overcompressive region, indicated by the shaded area, which extends above the surface \eqref{eq:classical_boundary}. Consequently, within the shaded region above the surface, the Riemann problem admits both classical and nonclassical solutions. In contrast, every state below the surface belongs to the overcompressive region, and the Riemann problem admits only nonclassical solutions.

\begin{figure}[H]
    \centering
    \includegraphics[width=0.5\linewidth]{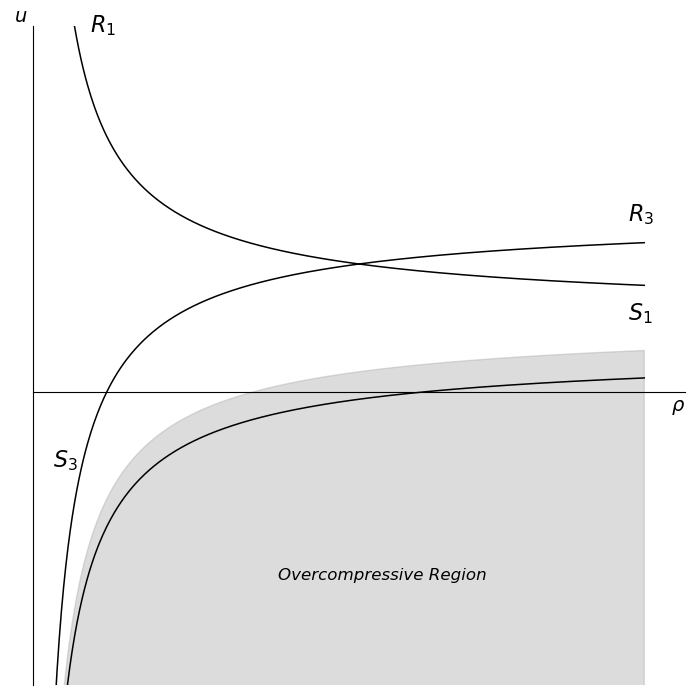}
    \caption{Graphical depiction of overcompressive region in the $\rho$-$u$ plane for fixed $v$, $A(v)$, and $\alpha$. The lowermost black surface is given by (\ref{eq:classical_boundary}).}
    \label{fig:overcompressive}
\end{figure}

\subsubsection{$\mathrm{SDW}+R_3$}
For wave combinations of this form, we seek $U_M$ such that 
\begin{equation} 
\begin{aligned}
c=\lambda_3(U_M)<\lambda_1(U_L)
\end{aligned} \label{ineq_R3_SDW}
\end{equation}
and $U_M\in IR_3(U_R)$ which immediately gives 
$v_M=v_R,$ $\rho_M<\rho_R.$ This means 
\begin{equation}
\begin{aligned}
u_M=u_R-\frac{2\sqrt{\alpha A(v_R)}}{\alpha+1}\left(\frac{1}{\rho_M^\frac{\alpha+1}{2}}-\frac{1}{\rho_R^\frac{\alpha+1}{2}}\right)
\end{aligned} \label{u_M_SDW_R3}
\end{equation}
and 
\begin{equation}
\begin{aligned}
u_M+\sqrt{\frac{\alpha A(v_R)}{\rho_M^{\alpha+1}}}=\frac{[\rho u]+\sqrt{[\rho u]^2-[\rho][\rho u^2-\frac{A(v)}{\rho^\alpha}]}}{[\rho]}
\end{aligned} \label{SDW_R3_speed}
\end{equation}
Expanding \eqref{SDW_R3_speed} and solving the quadratic for $u_M$ gives us 
\begin{equation}
\begin{aligned}
u_M=u_L-\sqrt{\frac{\alpha A(v_R)}{\rho_M^{\alpha+1}}}\pm\frac{1}{\sqrt{\rho_L}}\sqrt{\frac{(\alpha-1)A(v_R)}{\rho_M^\alpha}+\frac{A(v_L)}{\rho_L^\alpha}}
\end{aligned} \label{SDW_R3_2branch}
\end{equation}
When $[\rho]=0$, the speed formula in
\eqref{SDW_R3_speed} is replaced by $\sigma_0=\frac{\left[\rho u^2-\frac{A(v)}{\rho^\alpha}\right]}{2[\rho u]}.$ Substitution into the relation defining $U_M$ yields the same equation for the intermediate state. We will now use \eqref{SDW_R3_2branch} together with \eqref{ineq_R3_SDW} to select a branch. 

\begin{align}
\lambda_3(U_M) < \lambda_1(U_L) \notag &\implies u_M+\sqrt{\frac{\alpha A(v_R)}{\rho_M^{\alpha+1}}} < u_L-\sqrt{\frac{\alpha A(v_L)}{\rho_L^{\alpha+1}}} \notag\\ 
&\implies u_L\pm\frac{1}{\sqrt{\rho_L}}\sqrt{\frac{(\alpha-1)A(v_R)}{\rho_M^\alpha}+\frac{A(v_L)}{\rho_L^\alpha}} < u_L-\sqrt{\frac{\alpha A(v_L)}{\rho_L^{\alpha+1}}} \notag \\
&\implies \pm\frac{1}{\sqrt{\rho_L}}\sqrt{\frac{(\alpha-1)A(v_R)}{\rho_M^\alpha}+\frac{A(v_L)}{\rho_L^\alpha}} < -\sqrt{\frac{\alpha A(v_L)}{\rho_L^{\alpha+1}}} \label{SDW_R3_determination}
\end{align}

By \eqref{SDW_R3_determination}, we see that the positive branch is impossible. For the negative branch, we notice
\begin{align*}
    -\frac{1}{\sqrt{\rho_L}}\sqrt{\frac{(\alpha-1)A(v_R)}{\rho_M^\alpha}+\frac{A(v_L)}{\rho_L^\alpha}} < -\sqrt{\frac{\alpha A(v_L)}{\rho_L^{\alpha+1}}} \notag &\implies \frac{(\alpha-1)A(v_R)}{\rho_M^\alpha}+\frac{A(v_L)}{\rho_L^\alpha} > \frac{\alpha A(v_L)}{\rho_L^\alpha} \notag \\
   & \implies (\alpha-1)\left(\frac{A(v_R)}{\rho_M^\alpha}-\frac{A(v_L)}{\rho_L^\alpha}\right) > 0 \notag \\ 
   & \implies \rho_M > \left(\frac{A(v_R)}{A(v_L)}\right)^\frac{1}{\alpha}\rho_L
\end{align*}
This gives us an admissible interval for $\rho_M$, that is,
\begin{equation}
    \rho_M \in \left(\left(\frac{A(v_R)}{A(v_L)}\right)^\frac{1}{\alpha}\rho_L, \rho_R \right) \label{rho_interval_SDW_R3}
\end{equation}
Now consider 
\begin{equation}
\begin{aligned}
f(\rho_M)=u_R-u_L-\frac{2\sqrt{\alpha A(v_R)}}{\alpha+1}\left(\frac{1}{\rho_M^\frac{\alpha+1}{2}}-\frac{1}{\rho_R^\frac{\alpha+1}{2}}\right)+\sqrt{\frac{\alpha A(v_R)}{\rho_M^{\alpha+1}}}+\frac{1}{\sqrt{\rho_L}}\sqrt{\frac{(\alpha-1)A(v_R)}{\rho_M^\alpha}+\frac{A(v_L)}{\rho_L^\alpha}}
\end{aligned} \label{f_SDW_R3}
\end{equation}
It is easily shown that $f'>0$ for $\rho_M$ in the interval \eqref{rho_interval_SDW_R3}. To guarantee the existence of a unique solution within the admissible interval, we need $f\left(\left(\frac{A(v_R)}{A(v_L)}\right)^{\frac{1}{\alpha}}\rho_L\right)<0$ and $f(\rho_R)>0$. Evaluating \eqref{f_SDW_R3} at the endpoint values of $\rho_M$, we obtain the following desired inequalities
\begin{equation}
\begin{aligned}
u_R<u_L+\frac{2\sqrt{\alpha A(v_R)}}{\alpha+1}\left(\frac{1}{\left(\frac{A(v_R)}{A(v_L)}\right)^{\frac{\alpha+1}{2\alpha}}\rho_L^{\frac{\alpha+1}{2}}}-\frac{1}{\rho_R^{\frac{\alpha+1}{2}}}\right) -\frac{\sqrt{\alpha A(v_R)}}{\left(\frac{A(v_R)}{A(v_L)}\right)^{\frac{\alpha+1}{2\alpha}}\rho_L^{\frac{\alpha+1}{2}}}-\sqrt{\frac{\alpha A(v_L)}{\rho_L^{\alpha+1}}}.
\end{aligned}
\end{equation}
and
\begin{equation}
\begin{aligned}
u_R>u_L
-\sqrt{\frac{\alpha A(v_R)}{\rho_R^{\alpha+1}}}
-\frac{1}{\sqrt{\rho_L}}
\sqrt{
\frac{(\alpha-1)A(v_R)}{\rho_R^\alpha}
+\frac{A(v_L)}{\rho_L^\alpha}
}
\end{aligned}
\end{equation}
If both inequalities hold, then by the Intermediate Value Theorem, the function $f$ has a unique zero in the interval \eqref{rho_interval_SDW_R3}, and hence there exists a unique intermediate state.

\subsubsection{$R_1+\mathrm{SDW}$}
The analysis of the $R_1+\mathrm{SDW}$ configuration is analogous to that of the $\mathrm{SDW}+R_3$ case. Such a wave combination requires an intermediate state $U_M$ satisfying
\begin{equation}
    \lambda_1(U_M)=c>\lambda_3(U_R) \label{ineq_R1_SDW}
\end{equation}
and $U_M\in R_1(U_L)$ which immediately gives
$v_L = v_M,$ $\rho_M < \rho_L.$ This means
\begin{equation}
\begin{aligned}
u_M=u_L+\frac{2\sqrt{\alpha A(v_L)}}{\alpha+1}\left(\frac{1}{\rho_M^\frac{\alpha+1}{2}}-\frac{1}{\rho_L^\frac{\alpha+1}{2}}\right)
\end{aligned}
\end{equation}
and
\begin{equation}
\begin{aligned}
    u_M - \sqrt{\frac{\alpha A(v_L)}{\rho_M^{\alpha+1}}}=\frac{[\rho u]+\sqrt{[\rho u]^2-[\rho][\rho u^2-\frac{A(v)}{\rho^\alpha}]}}{[\rho]}
\end{aligned} \label{R1_SDW_speed}
\end{equation}
Expanding \eqref{R1_SDW_speed} and solving for the quadratic for $u_M$ gives us
\begin{equation}
\begin{aligned}
u_M=u_R+\sqrt{\frac{\alpha A(v_L)}{\rho_M^{\alpha+1}}}\pm\frac{1}{\sqrt{\rho_R}}\sqrt{\frac{(\alpha-1)A(v_L)}{\rho_M^\alpha}+\frac{A(v_R)}{\rho_R^\alpha}}
\end{aligned} \label{R1_SDW_2branch}
\end{equation}
When $[\rho]=0$, the speed formula in
\eqref{R1_SDW_speed} is replaced by $\sigma_0.$ Substitution into the relation defining $U_M$ produces the same equation for the intermediate state. We will now use \eqref{R1_SDW_2branch} together with \eqref{ineq_R1_SDW} to select a branch.
\begin{equation*}
\begin{aligned}
u_M-\sqrt{\frac{\alpha A(v_L)}{\rho_M^{\alpha+1}}}>u_R+\sqrt{\frac{\alpha A(v_R)}{\rho_R^{\alpha+1}}} 
\end{aligned}
\end{equation*}
\begin{equation}
\begin{aligned}
\implies \pm\frac{1}{\sqrt{\rho_R}}\sqrt{\frac{(\alpha-1)A(v_L)}{\rho_M^\alpha}+\frac{A(v_R)}{\rho_R^\alpha}}>\sqrt{\frac{\alpha A(v_R)}{\rho_R^{\alpha+1}}}
\end{aligned} \label{R1_SDW_determination}
\end{equation}
Equation~\eqref{R1_SDW_determination} shows that only the positive branch is admissible. Moreover, it gives us an admissible interval for $\rho_M$, that is,
\begin{equation}
\begin{aligned}
\rho_M \in \left(\left(\frac{A(v_L)}{A(v_R)}\right)^\frac{1}{\alpha}\rho_R, \rho_L \right)
\end{aligned} \label{rho_interval_R1_SDW}
\end{equation}
Now consider
\begin{equation}
\begin{aligned}
f(\rho_M)=u_R-u_L+\sqrt{\frac{\alpha A(v_L)}{\rho_M^{\alpha+1}}}+\frac{1}{\sqrt{\rho_R}}\sqrt{\frac{(\alpha-1)A(v_L)}{\rho_M^\alpha}+\frac{A(v_R)}{\rho_R^\alpha}}-\frac{2\sqrt{\alpha A(v_L)}}{\alpha+1}\left(\frac{1}{\rho_M^\frac{\alpha+1}{2}}-\frac{1}{\rho_L^\frac{\alpha+1}{2}}\right)
\end{aligned} \label{f_R1_SDW}
\end{equation}
As before, we have $f' > 0$ for $\rho_M$ in the interval \eqref{rho_interval_R1_SDW}. To guarantee the existence of a unique solution in the admissible interval, we need $f\left(\left(\frac{A(v_L)}{A(v_R)}\right)^\frac{1}{\alpha}\rho_R\right)<0$ and $f(\rho_L)>0$. Substituting the endpoint values of $\rho_M$ into \eqref{f_R1_SDW}, we obtain the following desired inequalities
\begin{equation}
\begin{aligned}
u_R<u_L-\frac{\sqrt{\alpha A(v_L)}}{\left(\frac{A(v_L)}{A(v_R)}\right)^\frac{\alpha+1}{2\alpha}\rho_R^\frac{\alpha+1}{2}}-\sqrt{\frac{{\alpha A(v_R)}}{\rho_R^{\alpha+1}}} +\frac{2\sqrt{\alpha A(v_L)}}{\alpha+1}\left(\frac{1}{\left(\frac{A(v_L)}{A(v_R)}\right)^\frac{\alpha+1}{2\alpha}\rho_R^\frac{\alpha+1}{2}}-\frac{1}{\rho_L^\frac{\alpha+1}{2}}\right).
\end{aligned}
\end{equation}
and
\begin{equation}
\begin{aligned}
u_R>u_L-\sqrt{\frac{\alpha A(v_L)}{\rho_L^{\alpha+1}}}-\frac{1}{\sqrt{\rho_R}}\sqrt{\frac{(\alpha-1)A(v_L)}{\rho_L^\alpha}+\frac{A(v_R)}{\rho_R^\alpha}}.
\end{aligned}
\end{equation}
If both inequalities hold, then by the Intermediate Value Theorem, the function $f$ has a unique zero in the interval \eqref{rho_interval_R1_SDW}, and hence there exists a unique intermediate state.
\subsection{Dissipation and the Dafermos Entropy Criterion} \label{subsec:dissipation}
The classification obtained in Section~\ref{subsec:combo_analytic} shows that the classical solution, whenever it exists, is unique. However, the same Riemann data may also admit a single overcompressive shadow wave (or $\delta$-shock), an $R_1+\mathrm{SDW}$ solution, a $\mathrm{SDW}+R_3$ solution, or several of these nonclassical wave patterns simultaneously. This is precisely where an additional admissibility criterion is needed to select a canonical solution. Throughout this section, we employ the Dafermos maximum entropy dissipation principle, which selects the admissible wave pattern maximizing the total entropy production \cite{Dafermos1973,Hsiao}.

An entropy--entropy flux pair $(\eta,q)$ satisfies
\begin{align}\label{eq:entropyEquivDef}
\nabla q(U)=\nabla\eta(U)D(G(U)),
\end{align}
where $G$ is the flux from equation \eqref{eq:sys_vec}. 
Consequently, $\eta(U)_t+q(U)_x=0$ for every smooth solution. For system \eqref{cons1}--\eqref{cons3}, the functions
\begin{align}\label{eq:entropy}
    \eta(U)&=\frac{\rho u^2}{2}+\frac{A(v)}{(1+\alpha)\rho^\alpha}, \\
    q(U)&=\frac{\rho u^3}{2}-\frac{\alpha u A(v)}{(1+\alpha)\rho^\alpha},\label{eq:entropyFlux}
\end{align}
satisfy \eqref{eq:entropyEquivDef}, as can be verified by direct computation. In order to ensure the entropy $\eta$ is convex---an important condition for applying Dafermos's admissibility criterion---one can compute the Hessian of $\eta$ and apply Sylvester's Criterion. It is sufficient then to assume
\begin{align}\label{eq:convex1}
    A\dpri(v)&\geq0, \\
    \alpha A(v)A\dpri(v)&\geq(1+\alpha)(A\pri(v))^2,\label{eq:convex2}
\end{align}
for the convexity of $\eta$. Therefore, for any $A(v)$ satisfying conditions \eqref{eq:convex1} and \eqref{eq:convex2}, the functions \eqref{eq:entropy} and \eqref{eq:entropyFlux} are valid entropy and entropy flux functions, respectively, for the system. 

A weak solution is said to satisfy the entropy condition if
$\eta(U)_t+q(U)_x\le0$ in the sense of distributions. This is a necessary condition for the admissibility of a weak solution, though it is not sufficiently strict to select a unique solution \cite{Dafermos1973}. For a classical discontinuity propagating with speed $s$, integrating the entropy inequality over a spacetime control volume enclosing the wave, applying Green's theorem, and shrinking the control volume onto the discontinuity yields $[q]-s[\eta]\le0.$ Equivalently, $D=s[\eta]-[q]\ge0,$ where $D$ denotes the entropy production across the wave. For a shadow wave, the entropy production consists of two contributions. The first is the classical jump contribution $c[\eta]-[q],$ where $c$ denotes the propagation speed of the shadow wave. The second contribution arises from the entropy concentrated inside the singular layer. Let
\begin{equation}
E_{\delta,\varepsilon}(t)
=
\int_{ct-\varepsilon t}^{ct+\varepsilon t}
\eta(U_\varepsilon(x,t))\,dx
\end{equation}
denote the entropy contained in the approximate shadow-wave layer. Passing to the limit as $\varepsilon\to0^+$ gives the concentrated entropy $E_\delta(t)=\frac12c^2k_1t,$ where $k_1$ is the weight of the $\delta$-measure as computed in section \ref{sec:singular}. Hence
$\frac{d}{dt}E_\delta(t)
=\frac12c^2k_1.$ The concentrated entropy contributes with the opposite sign to the entropy production, yielding
\begin{equation}
D_{\mathrm{SDW}}
=
c[\eta]-[q]
-
\frac12c^2k_1.
\end{equation} Consequently, the entropy production associated with each elementary wave is
\begin{align}
D_{R_i}=0,\ \ D_{CD}=0,\ \ D_{S_i}=s_i[\eta]-[q], \ \ D_{\mathrm{SDW}}
=c[\eta]-[q]-\frac12c^2k_1,
\end{align}
where $s_i$ is the shock speed determined by the Rankine-Hugoniot conditions and $c$ is the propagation speed of the $\delta$-shock. For a wave pattern consisting of several waves, the total entropy production is defined as the sum of the entropy productions of its constituent waves. Whenever more than one admissible wave pattern connects the same left and right states, the Dafermos maximum entropy dissipation principle selects the one with the largest total entropy production.

The geometry of the admissible regions depends on the structure of $A(v)$, making a complete analytical comparison of the entropy production of all competing wave patterns intractable for a general choice of $A(v)$. For the numerical examples in Section~\ref{sec:numerics} ($A(v)=\frac{1}{(1-v)^p}$ with $\alpha=0.66$, $p=0.146$, $v<1$), we computed the entropy production for every admissible classical, singular, and combination wave pattern. 

Our computational study revealed a clear distinction between the right states that initially lie outside and inside the overcompressive region. Outside the overcompressive region, when the density jumps across the waves were nonzero, the corresponding classical wave pattern generally produced larger entropy dissipation than the competing singular-wave solutions. In contrast, when the density jump vanished, i.e., $\rho_M=\rho_R$, the singular-wave solution typically exhibited larger entropy dissipation, with only a few exceptions, namely for sufficiently small values of $\rho_R$ in Region~7 of Case~2 and in the unshaded $\rho<\rho_L$ part of Region~5 of Case~1, where the classical solution initially produced larger entropy dissipation before the singular solution became dominant. Within the overcompressive region, the behavior was more intricate and depended on the parameter region. In Case~1, the shaded portion of Region~5 with $\rho>\rho_L$ exhibited three successive regimes as the right state varied. First, when the density jump across the shadow wave was nonzero, the singular-wave solution produced a larger entropy dissipation. Next, the classical solution became dominant. Then the $R_1+\mathrm{SDW}$ solution disappears, and subsequently the classical solutions also disappear, so the $\mathrm{SDW}$ solution eventually dominates. In the shaded portion of Region~5 with $\rho<\rho_L$, the singular-wave solution again initially produced larger entropy dissipation when the density jump was nonzero, followed by a regime where the classical solution became dominant, and finally by a regime with $\rho_M=\rho_R$, where the singular-wave solution again produced the larger entropy dissipation. In Region~6 of Case~1, where $\rho>\rho_L$, the singular-wave solution was the only admissible solution, whereas for $\rho<\rho_L$ the singular-wave solution consistently produced a greater dissipation of the entropy.
In Case~2, the shaded portion of In Region~5 with $\rho>\rho_L$, the singular-wave solution initially produced a greater dissipation of entropy when the density jump was nonzero, followed by a regime where the classical solution became dominant. Then, classical solutions and solutions $R_1+\mathrm{SDW}$ disappear, eventually leaving $\mathrm{SDW}$ as the only admissible solution. In Region~5 with $\rho<\rho_L$, the singular-wave solution initially produced larger entropy dissipation when the density jump was nonzero, followed by a regime where the classical solution became dominant and finally by a regime with $\rho_M=\rho_R$, where the singular-wave solution had higher dissipation. In the shaded portion of Region~6 with $\rho<\rho_L$, when the density jump was nonzero, the classical solution initially appeared to produce a greater dissipation of entropy, while the dissipation of entropy of the singular-wave solution was not defined due to negative radicands. Subsequently, as the jump in density becomes zero, the classical solution disappears, thus leaving the singular-wave solution as its only admissible solution. In Region~6 with $\rho>\rho_L$, only singular-wave solutions were admissible.

Despite these computational predictions, the numerical simulations of Section~\ref{sec:numerics} always selected the corresponding classical wave pattern whenever both classical and singular solutions coexisted. Explaining the discrepancy between the entropy dissipation criterion and the numerical approximation remains an interesting open question.
\section{Numerical Solutions} \label{sec:numerics}
\subsection{Numerical Preliminaries}
We utilize the standard Lax-Friedrichs (LF) method to obtain numerical approximations of our solution space and validate our analytical findings. The LF scheme is a foundational, non-oscillatory, first-order accurate method widely used to approximate solutions of hyperbolic partial differential equations \cite{Toro}. 

The method is structured by discretizing the spatial and temporal domains into a uniform grid mesh. We define a discrete mesh point $(x_j, t_n)$ by: $x_j = j h, \ (j \in \mathbb{Z}), \ \ 
    t_n = n k, \ (n \in \mathbb{N}_0)$ where $h := \Delta x$ represents the spatial width of the grid and $k := \Delta t$ represents the temporal step size. The numerical approximation of the conserved vector at this discrete point is denoted as $U^n_j \approx U(x_j, t_n)$. 

To advance the solution to the next time step, the LF scheme employs a central spatial difference combined with a localized averaging stabilizer. The update formula is given by:
\begin{equation}
    U^{n+1}_j = \frac{1}{2}\left(U^{n}_{j-1} + U^{n}_{j+1}\right) + \frac{k}{2h}\left(F^{n}_{j-1} - F^{n}_{j+1}\right)
\end{equation}
where $F^n_j = F(U^n_j)$ represents the physical flux vector. Here, $\lambda$ is the global maximum wave speed computed from the absolute maximum eigenvalue of the Jacobian matrix of flux of the system at the time step $n$. 

The Courant number, defined as $\text{CFL} = \frac{\lambda k}{h}$, scales the time step dynamically. To guarantee numerical stability and satisfy the Courant-Friedrichs-Lewy condition on this specific spatial stencil, we enforce that $\text{CFL} \leq 1$. Further discussion on the LF method can be found in \cite{Lev_1}. 
\subsection{Regions for the Solution to the Riemann Problem}
As stated above, we consider Riemann data consisting of left and right states $(\rho_L,u_L,v_L)$ and $(\rho_R,u_R,v_R)$. The admissible solution depends on
the relative position of these states in phase space. Consequently, the space of right states is partitioned into regions, each corresponding to a wave combination connecting the fixed left state $U_L$ to the prescribed right
state $U_R$.

We analyze the behavior of the solution regions for two representative choices of the left state, referred to as Case~1 and Case~2. In both cases, we take $A(v)=\frac{1}{(1-v)^p}$ for $v<1$, with $\alpha=0.66$ and $p=0.146$. This choice of $A(v)$ ensures that the entropy function \eqref{eq:entropy} is strictly convex throughout its domain, allowing the application of the Dafermos maximum entropy dissipation principle discussed in Section~\ref{subsec:dissipation} \cite{Dafermos1973,Hsiao}. For Case~1, the left state is $(\rho_L,u_L,v_L)=(4,0.7,-1.9)$, whereas for Case~2 it is $(\rho_L,u_L,v_L)=(4,0.7,0.96)$. Thus, the two cases differ only in the value of $v_L$. Nevertheless, this change dramatically alters the geometry of the admissible solution regions due to the singular behavior of $A(v)$ as $v\to1^{-}$.

For the initial data of Case 1, there are 9 identifiable solution regions. As shown in Figure \ref{fig:case1}, we label the regions for each case in the $(\rho, u)$ plane for ease of viewing; however, the true solution space is three-dimensional. We incorporate the $v$ dimension by notating the regions as follows: a $\to$ symbol denotes a shift from a higher $v_R$ state to a lower $v_R$ state. For example, $R_1+CD+S_3 \to S_1+CD+S_3$ indicates that the solution exhibits $R_1+CD+S_3$ near $v_R = 1$ and transitions to $S_1+CD+S_3$ as $v_R$ decreases in value. The regions for Case 1 are listed below:

\begin{figure}[H]
    \centering
    \includegraphics[width=0.5\linewidth]{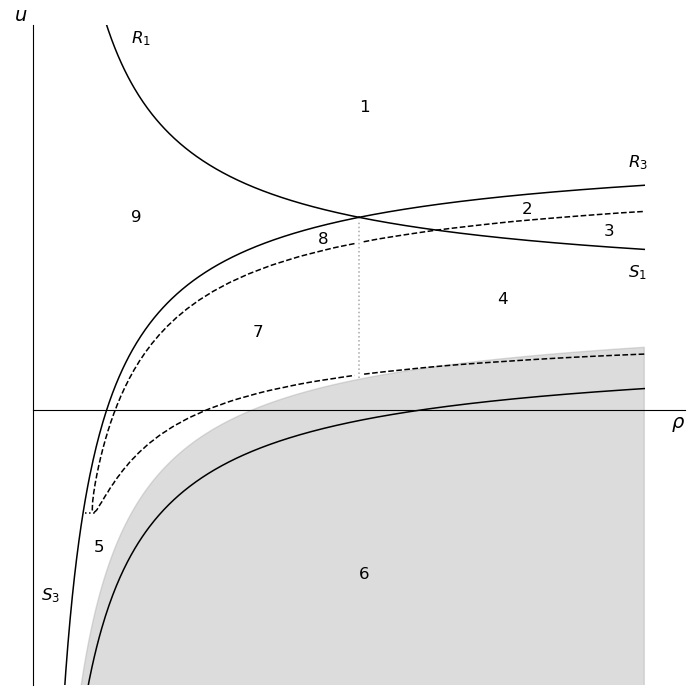}
    \caption[Case 1]{\textbf{Case 1 ($a = 0.66$, $p = 0.146$)} with left state $(4, 0.7, -1.9)$. The dashed black line represents the surface $S_{\mathrm{SDW}}(\rho_L, u_L, v_L, \rho, u, v) = \lambda_3(\rho, u, v)$. The lowermost black surface is given by (\ref{eq:classical_boundary}).}
    \label{fig:case1}
\end{figure}

\begin{itemize}
   \item Region 1: $R_1+CD+S_3\to R_1+CD+R_3\to S_1+CD+R_3$ \\
    When the right state lies in this region, the Riemann solution evolves through three distinct classical wave configurations as $v_R$ decreases. For values of $v_R$ sufficiently close to one, the solution consists of $R_1+CD+S_3$. As $v_R$ decreases, the $3$-shock is replaced by a $3$-rarefaction, yielding the configuration $R_1+CD+R_3.$ For sufficiently negative values of $v_R$, the $1$-rarefaction is subsequently replaced by a $1$-shock, resulting in the wave pattern $S_1+CD+R_3$. The waves appear in the order determined by the characteristic speeds.
    We note that the nonclassical solution $R_1+\mathrm{SDW}$ is admissible when $v_R$ is sufficiently close to $1$, and $\mathrm{SDW} + R3$ is admissible when $v_R$ is sufficiently negative. Thus, there are certain intervals of $v_R$ in this region when both classical and nonclassical solutions coexist. In Section~\ref{subsec:dissipation}, we show that Dafermos' dissipation criterion selects, among these competing solutions, the one with the largest dissipation.     
 \item Region 2: $R_1+CD+S_3 \to R_1+CD+R_3 \to S_1+CD+R_3$ \\
   When the right state lies in this region, the Riemann solution consists of classical waves, the exact types of which depend on the value of $v_R$ as described above. In addition, the nonclassical solution $R_1+\mathrm{SDW}$ is admissible when $v_R$ is sufficiently close to $1$, and $\mathrm{SDW} + R3$ is admissible when $v_R$ is sufficiently negative. We note that in a narrow neighborhood of the $S_1$ curve, the Riemann problem instead has the behavior $R_1+CD+S_3\to S_1+CD+S_3\to S_1+CD+R_3$; in this neighborhood, it also admits the nonclassical configurations $R_1+\mathrm{SDW}$ and $\mathrm{SDW} + R_3$, under the same conditions as other points in this region. 
    \item Region 3: $R_1+CD+S_3 \to R_1+CD+R_3 \to S_1+CD+R_3$ \\
   When the right state lies in this region, the Riemann solution consists of classical waves, the exact types of which depend on the value of $v_R$ as described above. In addition, the nonclassical solution $R_1+\mathrm{SDW}$ is admissible when $v_R$ is sufficiently close to $1$, and $\mathrm{SDW} + R3$ is admissible when $v_R$ is sufficiently negative. We note that in a narrow neighborhood of the $S_1$ curve, the Riemann problem instead has the behavior $R_1+CD+S_3\to S_1+CD+S_3\to S_1+CD+R_3$; in this neighborhood, it also admits the nonclassical configurations $R_1+\mathrm{SDW}$ and $\mathrm{SDW} + R_3$, under the same conditions as other points in this region. 
\item Region 4: $R_1+CD+S_3\to S_1+CD+S_3\to S_1+CD+R_3$ \\
The classical solution type changes from $R_1+CD+S_3$ for large $v_R$ to $S_1+CD+S_3$ (see Figure \ref{fig:region4}) for intermediate $v_R$ and finally to $S_1+CD+R_3$ for very small $v_R$. The Riemann problem also admits the nonclassical configuration $R_1+\mathrm{SDW}$ when $v_R$ is sufficiently big and the nonclassical solution $\mathrm{SDW}+R_3$ when $v_R$ is sufficiently negative.
 \item Region 5: $R_1+CD+S_3\to S_1+CD+S3 \to \text{No Classical}$\\
When the right state lies in this region, the classical Riemann solution initially consists of the wave pattern $R_1+CD+S_3$ for values of $v_R$ sufficiently close to $1$. As $v_R$ decreases, this transitions to the solution type $S_1+CD+S_3$. The admissible solutions for this region also depend on whether $\rho < \rho_L$ or $\rho > \rho_L$ and whether the initial right state is chosen outside of the overcompressive region or inside. When $\rho < \rho_L$ and the initial right state lies inside the overcompressibive region, then both $R_1 + \mathrm{SDW}$ and $\mathrm{SDW}$ are admissible solutions; when $\rho < \rho_L$ and the initial right state lies outside the overcompressive region, then $R_1 + \mathrm{SDW}$ is the only admissible solution; when $\rho > \rho_L$ and the initial right state lies inside the overcompressive region, then $\mathrm{SDW}$ is the only admissible solution; when $\rho > \rho_L$ and the initial right state lies outside the overcompressive region, no classical solutions, $\mathrm{SDW}$, $R_1 + \mathrm{SDW}$, or $\mathrm{SDW} + R_3$ solutions are admissible, hence making undercompressive shadow waves a likely candidate solution. Further investigation of undercompressive shadow waves will be done in future research. 
\item Region 6: $\mathrm{SDW}$ \\
When the right state lies in this region, the Riemann solution consists entirely of a singular delta wave (see Figure \ref{fig:region6}). For $\rho<\rho_L$, the solution instead exhibits $R_1+CD+S_3$ for $v_R$ close to one and then the nonclassical configuration $R_1 + \mathrm{SDW}$ for smaller $v_R$, in addition to $\mathrm{SDW}$.
\item Region 7: $R_1 + CD + S_3 \to S_1 + CD + S_3$ or $R_1 + \mathrm{SDW}$ \\
When the right state lies in this region, the wave structure differs between the upper and lower portions. In the upper portion, the classical Riemann solution initially consists of the wave pattern $R_1+CD+S_3$ for values of $v_R$ sufficiently close to $1$. As $v_R$ decreases, the solution transitions to $S_1+CD+S_3$ and subsequently to $S_1+CD+R_3$. In addition, the Riemann problem also admits the nonclassical solution $R_1+\mathrm{SDW}$ and, for smaller values of $v_R$, the nonclassical solution $\mathrm{SDW}+R_3$. In the lower portion, the Riemann problem admits the nonclassical solutions $R_1+\mathrm{SDW}$ and, for smaller values of $v_R$, $\mathrm{SDW}+R_3$. When restricting to classical wave patterns, the solution consists of $R_1+CD+S_3$ for values of $v_R$ sufficiently close to one and transitions to $S_1+CD+S_3$ as $v_R$ decreases, until the classical solution ceases to exist and is replaced by the nonclassical $R_1+\mathrm{SDW}.$
\item Region 8: $R_1+CD+S_3 \to S_1+CD+S_3 \to S_1 + CD + R_3$ or $R_1+\mathrm{SDW}$ \\
When the right state lies in this region, the classical Riemann solution initially consists of the wave pattern $R_1+CD+S_3$ for values of $v_R$ sufficiently close to one. As $v_R$ decreases, the solution transitions to $S_1+CD+S_3$ and subsequently to $S_1 + CD + R_3$ (see Figures~\ref{fig:region8part1} and \ref{fig:region8part2}). In addition, the Riemann problem also admits the nonclassical solution $R_1+\mathrm{SDW}$ and, for smaller values of $v_R$, $\mathrm{SDW} + R_3$.
\item Region 9: $R_1+CD+S_3\to S_1+CD+R_3$ or $R_1+\mathrm{SDW}$ \\
When the right state lies in this region, the classical Riemann solution has the wave pattern $R_1+CD+S_3$ and transitions to the pattern $S_1+CD+R_3$ for sufficiently small $v_R$. In addition, the Riemann problem also admits the nonclassical solution $R_1+\mathrm{SDW}$ and, for sufficiently small values of $v_R$, $\mathrm{SDW} + R_3$.
\end{itemize}
Figures~\ref{fig:region4} through \ref{fig:region8part2} demonstrate that representative points selected within each distinct region numerically validate the analytic behavior derived above. Specifically, these figures illustrate the spatial-temporal evolution of the variables $\rho$, $u$, and $v$ as a function of the self-similar variable $x/t$. 

As an example, consider the regions where $\rho$ and $u$ experience simultaneous decreases. This behavior provides numerical confirmation that the solution is progressing along the $S_3$ shock/rarefaction curve, in agreement with the analytical predictions. Although we verified all regions to confirm the results, we present only a representative selection of figures illustrating the agreement between the analytical theory and the numerical simulations. 
\begin{figure}[H]
    \centering
    \begin{minipage}[b]{0.48\textwidth}
        \centering
        \includegraphics[width=\linewidth]{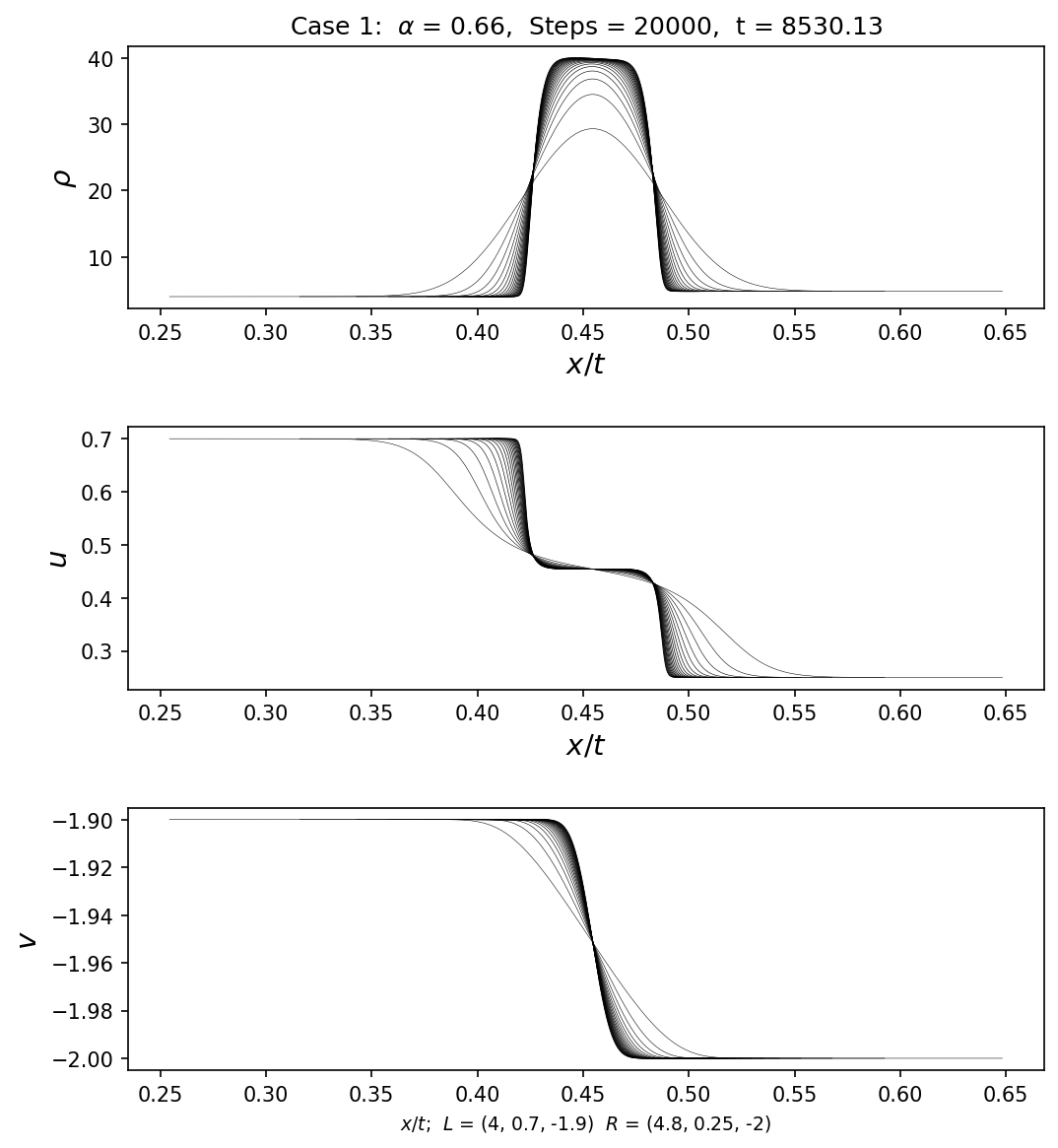}
        \caption{Region 4 in \textbf{Case 1}.}
        \label{fig:region4}
    \end{minipage}
    \hfill
    \begin{minipage}[b]{0.48\textwidth}
        \centering
        \includegraphics[width=\linewidth]{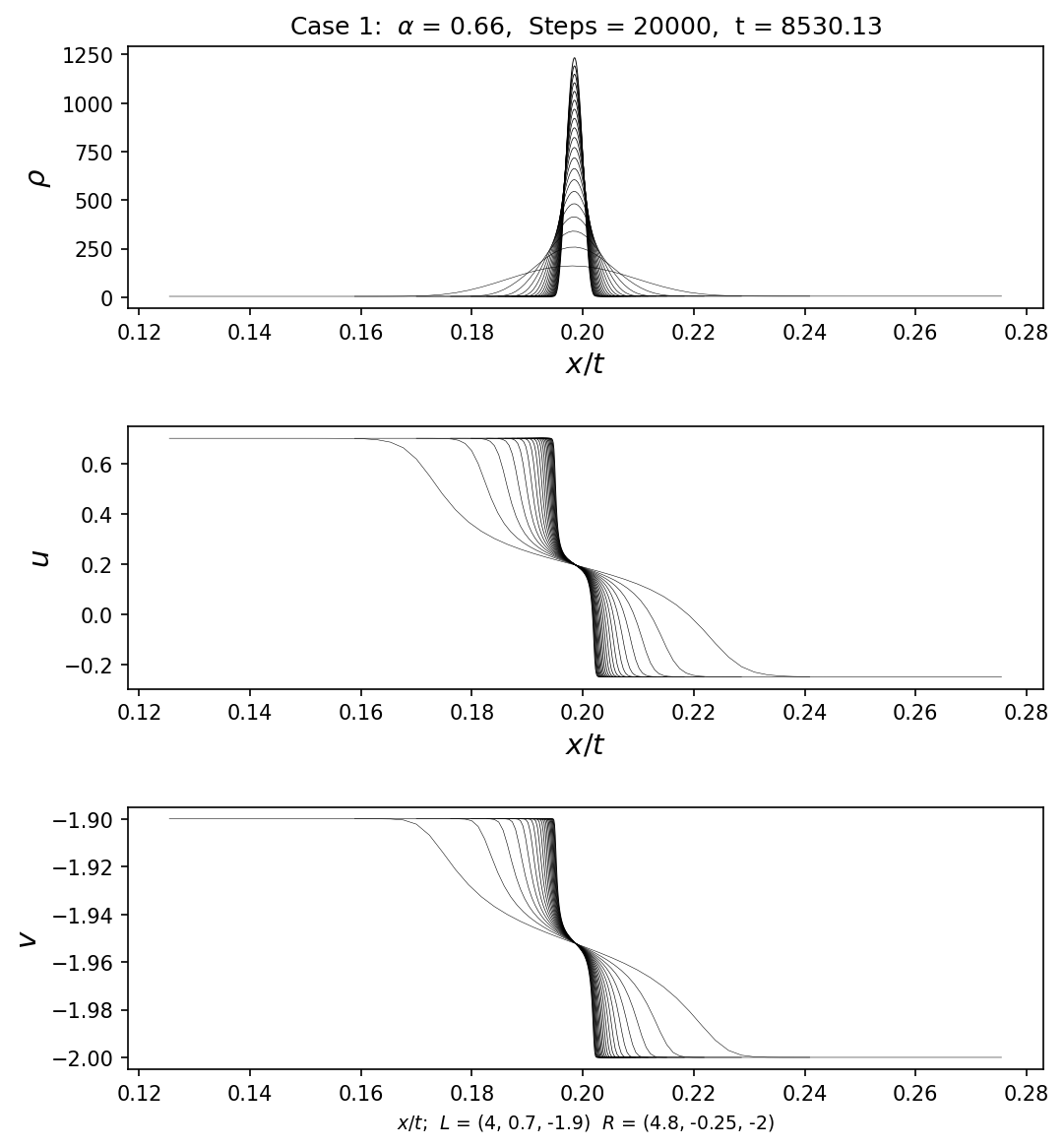}
        \caption{Region 6 in \textbf{Case 1}.}
        \label{fig:region6}
    \end{minipage}
\end{figure}

\begin{figure}[H]
    \centering
    \begin{minipage}[b]{0.48\textwidth}
        \centering
        \includegraphics[width=\linewidth]{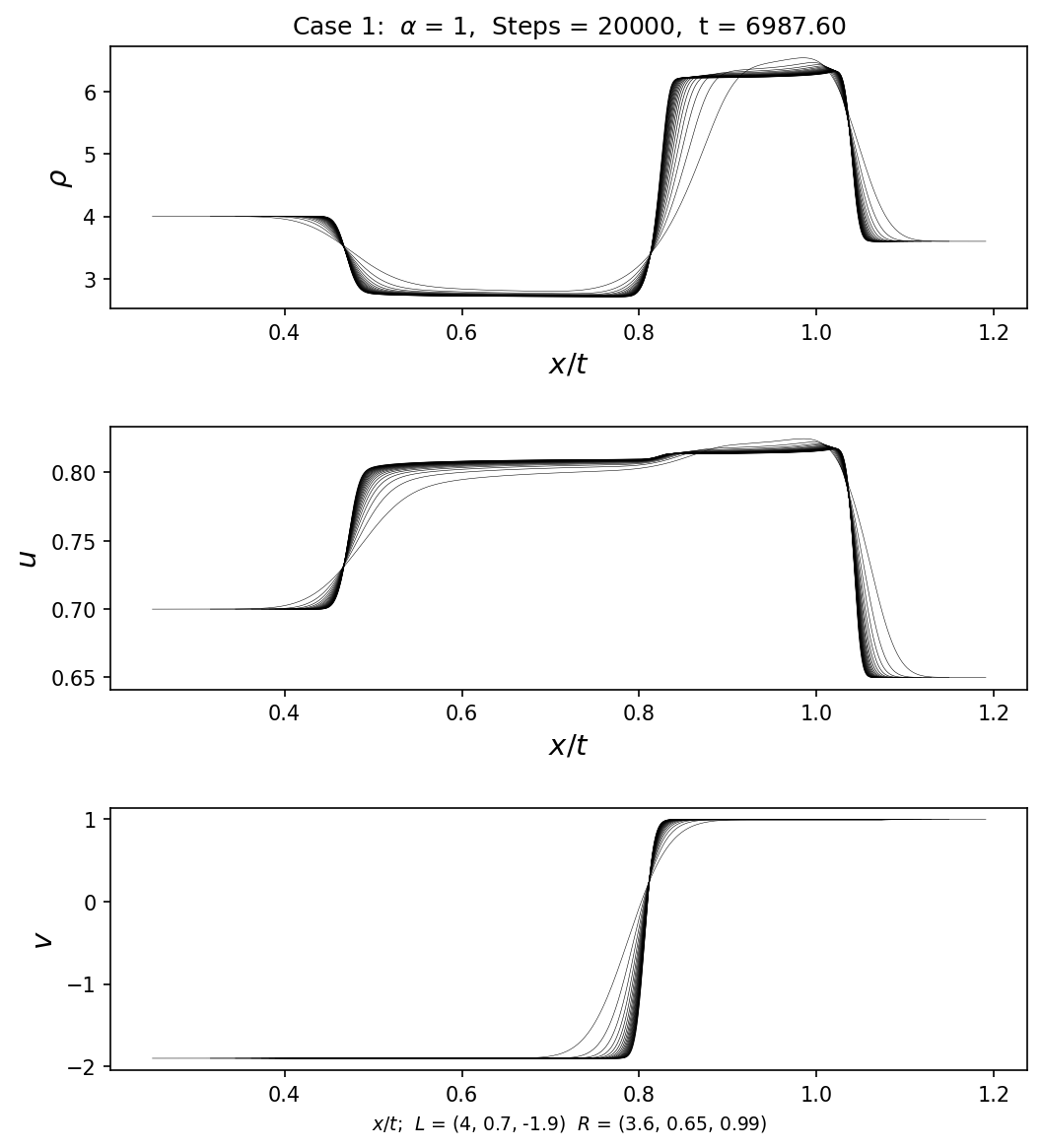}
    \caption{Region 8 with $v_R$ close to 1 in \textbf{Case 1}.}
    \label{fig:region8part1}
    \end{minipage}
    \hfill
    \begin{minipage}[b]{0.48\textwidth}
        \centering
        \includegraphics[width=\linewidth]{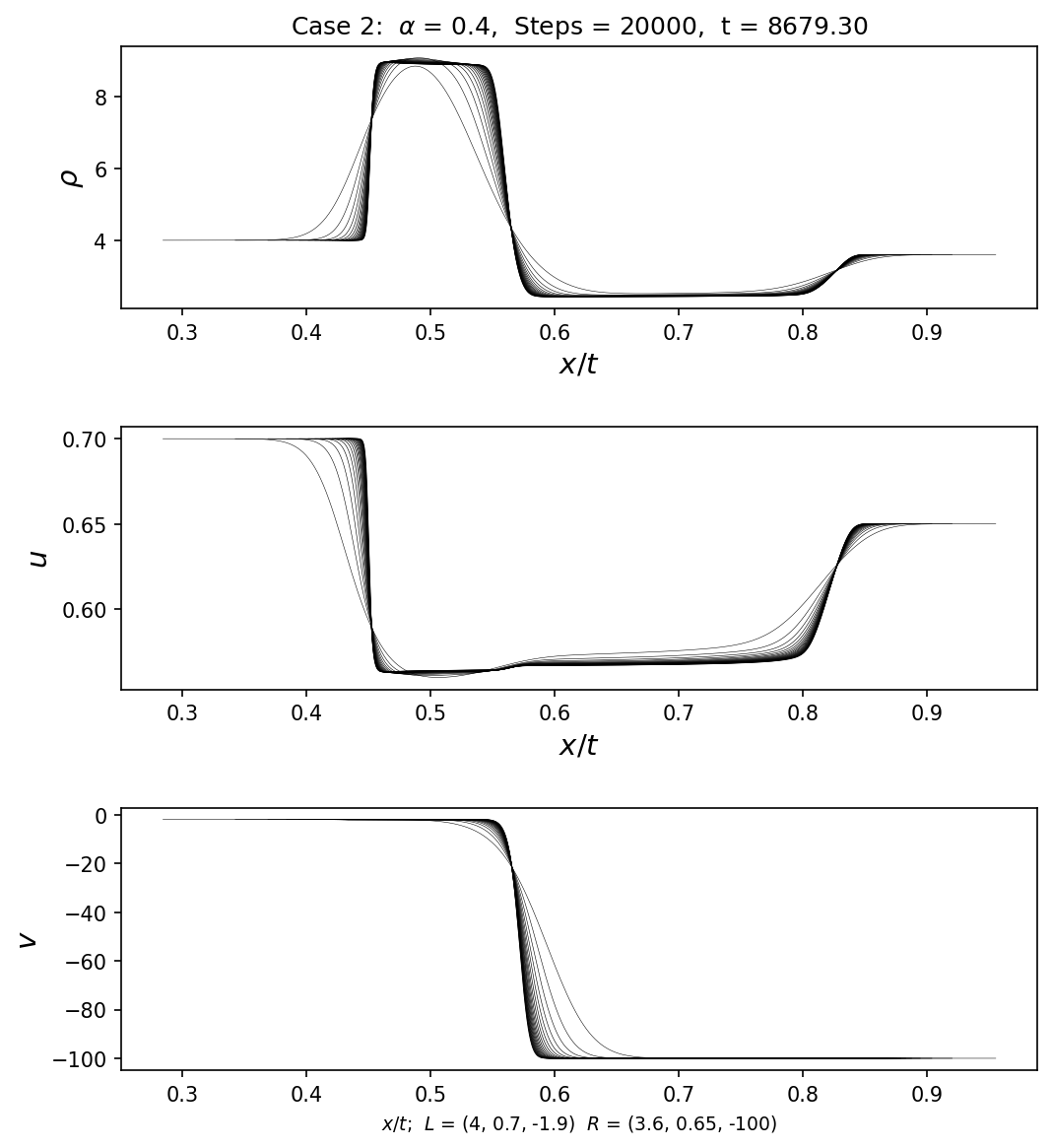}
        \caption{Region 8 with negative $v_R$ in \textbf{Case 1}.}
        \label{fig:region8part2}
    \end{minipage}
\end{figure}

The regions for Case 2 are listed below and shown in Figure \ref{fig:case2}: 

\begin{figure}[H]
    \centering
    \includegraphics[width=0.5\linewidth]{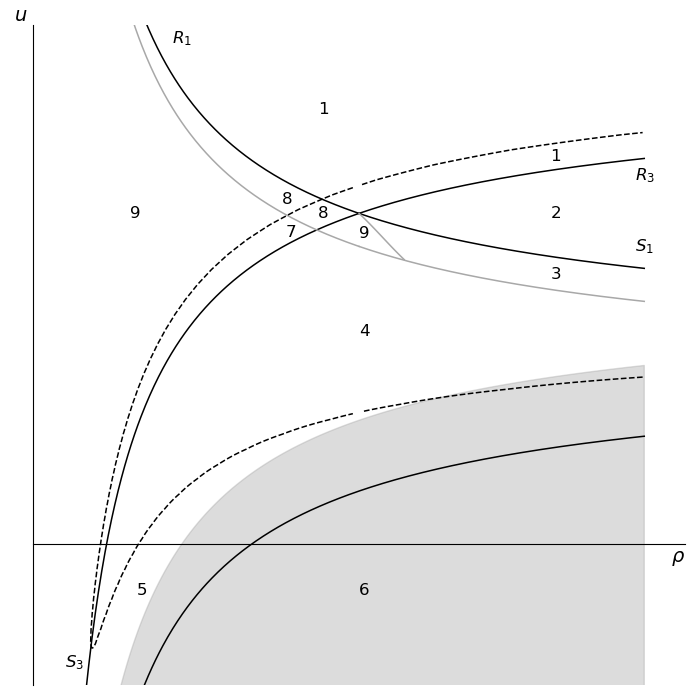}
    \caption[Case 2]{\textbf{Case 2 ($a = 0.66$, $p = 0.146$)} with left state $(4, 0.7, 0.96)$. The dashed black line and lowermost solid black line are defined as in Case 1. The light gray solid lines represent the surfaces $s_{\mathrm{SDW}}(u_L, u) = \lambda_1(u)$ and $s_{\mathrm{SDW}}(u_L, u) = \lambda_1(u_L)$, ordered from left to right.}
    \label{fig:case2}
\end{figure}

\begin{itemize}
    \item Region 1: $R_1 + CD + S_3 \to R_1+CD+R_3 \to S_1 + CD + R_3$ \\
    In this region, the Riemann solution evolves through three distinct classical wave configurations as $v_R$ decreases. For values of $v_R$ sufficiently close to $1$, $R_1 + CD + S_3$ is an admissible solution. As $v_R$ decreases, the solution changes to $R_1 + CD + R_3$ and subsequently to $S_1 + CD + R_3$. We also note that the nonclassical solution $R_1 + \mathrm{SDW}$ is admissible when $v_R$ is sufficiently close to $1$, and $\mathrm{SDW} + R_3$ is admissible when $v_R$ is sufficiently negative. Thus, there are certain intervals of $v_R$ in this region when both classical and nonclassical solutions coexist.
    \item Region 2: $R_1+CD+S_3 \to S_1+CD+R_3$ \\
    In this region, the Riemann solution begins as $R_1+CD+S_3$ for $v_R$ sufficiently close to $1$ and transitions to $S_1+CD+R_3$ solutions as $v_R$ decreases. In addition, for $v_R$ sufficiently close to $1$, the nonclassical solution $R_1 + \mathrm{SDW}$ is also admissible; for $v_R$ sufficiently negative, $\mathrm{SDW} + R_3$ is admissible.
    \item Region 3: $R_1+CD+S_3\to S_1+CD+S_3 \to S_1+CD+R_3$ \\
    In this region, the Riemann solution begins as $R_1+CD+S_3$ for $v_R$ sufficiently close to $1$ and transitions to $S_1+CD+S_3$ and then to $S_1+CD+R_3$ solutions as $v_R$ decreases. In addition, for $v_R$ sufficiently close to $1$, the nonclassical solution $R_1 + \mathrm{SDW}$ is also admissible; for $v_R$ sufficiently small, $\mathrm{SDW} + R_3$ is admissible.
    \item Region 4: $R_1+CD+S_3\to S_1+CD+S_3 \to \text{No Classical}$ \\
     In this region, the Riemann solution begins as $R_1+CD+S_3$ for $v_R$ sufficiently close to $1$ and transitions to $S_1+CD+S_3$ solutions as $v_R$ decreases. We note that for sufficiently small $v_R$, there are no classical solutions; rather, solutions depend on the value of $\rho$. When $\rho < \rho_L$, then $R_1 + \mathrm{SDW}$ is the only admissible solution. When $\rho > \rho_L$, no classical, $\mathrm{SDW}$, $R_1 + \mathrm{SDW}$, or $\mathrm{SDW} + R_3$ solutions are admissible, hence again making undercompressive shadow wave solutions a likely candidate. 
    \item Region 5: $R_1+CD+S_3\to S_1+CD+S_3 \to \text{No Classical}$ \\
    In this region, the Riemann solution begins as $R_1+CD+S_3$ for $v_R$ sufficiently close to $1$ and transitions to $S_1+CD+S_3$ solutions as $v_R$ decreases. We note that for sufficiently small $v_R$, there are no classical solutions; rather, solutions take the form $R_1+\mathrm{SDW}$ or $\mathrm{SDW}$ when $\rho < \rho_L$, and the form $\mathrm{SDW}$ when $\rho > \rho_L$. In addition, we note that when $\rho < \rho_L$ and the point is outside of the overcompressive region, the nonclassical solution $\mathrm{SDW} + R_3$ is also admissible for a short interval of $v_R$ close to $1$.
    \item Region 6: $\mathrm{SDW}$\\
    In this region, the Riemann solution is divided into two sub-regions depending whether $\rho > \rho_L$ or $\rho < \rho_L$. For $\rho>\rho_L$, the solution has the form $\mathrm{SDW}$. For $\rho<\rho_L$, the solution begins as $R_1+CD+S_3$ and transitions to $R_1 + \mathrm{SDW}$ for smaller $v_R$, in addition to $\mathrm{SDW}$.  
    \item Region 7: $R_1 + CD + S_3 \to S_1 + CD + S_3$ or $R_1 + \mathrm{SDW}$\\
    In this region, the Riemann solution is divided into two sub-regions by a threshold density $\rho_0$. For $\rho > \rho_0$, the wave pattern transitions as $R_1+\text{CD}+S_3 \rightarrow S_1+\text{CD}+S_3 \rightarrow S_1+\text{CD}+R_3$ as $v_R$ decreases, along with $R_1 + \mathrm{SDW}$ being an admissible solution for all values of $v_R$ and $\mathrm{SDW} + R_3$ being admissible for sufficiently small values of $v_R$. In contrast, for $\rho < \rho_0$, the sequence shifts to $R_1+\text{CD}+S_3 \rightarrow S_1+\text{CD}+S_3$. As $v_R$ becomes sufficiently small, classical solutions cease to exist. Instead, $R_1 + \mathrm{SDW}$ and, for a short interval of small $v_R$, $\mathrm{SDW} + R_3$ are admissible solutions.
    \item Region 8: $R_1+CD+S_3 \to R_1+CD+R_3 \to S_1+CD+R_3$ or $R_1 + \mathrm{SDW}$ \\
    In this region, the Riemann solution begins with $R_1 + CD + S_3$ for $v_R$ sufficiently close to $1$, then transitions to $R_1 + CD + R_3$ and subsequently to $S_1 + CD + R_3$ as $v_R$ decreases. In addition, the nonclassical solution $R_1+\mathrm{SDW}$ is admissible for all values of $v_R$, and $\mathrm{SDW} + R3$ is admissible when $v_R$ is sufficiently negative.  
    \item Region 9: $R_1+CD+S_3 \to S_1+CD+S_3 \to S_1+CD+R_3$ or $R_1 + \mathrm{SDW}$ \\
    In this region, the Riemann solution exhibits $R_1+CD+S_3$ for $v_R$ values close to $1$, then transitions to $S_1 + CD + S_3$ and subsequently to $S_1 + CD + R_3$ for decreasing $v_R$. In addition, the nonclassical solution $R_1 + \mathrm{SDW}$ is admissible for all values of $v_R$, and $\mathrm{SDW} + R_3$ is admissible for $v_R$ smaller than a certain threshold value.
\end{itemize}
We note that the location of the right state relative to the overcompressive region depends on the chosen value of $v_R$. For values of $v_R$ for which the right state lies inside the overcompressive region, the admissible solution is either an $R_1+\mathrm{SDW}$ or an isolated $\mathrm{SDW}$. If the right state initially lies outside the overcompressive region, then, as $v_R$ is varied, it may pass through a region where an admissible solution has the form $\mathrm{SDW}+R_3$, typically for smaller values of $v_R$.

\section{Fully Linearly Degenerate Case: $\alpha=1$} \label{sec:analytic}

We now consider the case $\alpha=1$. As noted in Section~\ref{sec:classical}, all three characteristic families are linearly degenerate. As before, let $U_-=(\rho_-,u_-,v_-)^T$ denote a fixed left state for the system. Applying the Rankine--Hugoniot conditions yields the following contact discontinuity curves:
\begin{align}
    CD_1&: \begin{cases}
        v=v_- \\
        u=u_-+\sqrt{A(v_-)}\Big(\frac{1}{\rho}-\frac{1}{\rho_-}\Big)
    \end{cases} \\
    CD_2&: \begin{cases}
        A(v)=\frac{\rho A(v_-)}{\rho_-} \\
        u=u_-
    \end{cases} \\
    CD_3&: \begin{cases}
        v=v_- \\
        u=u_--\sqrt{A(v_-)}\Big(\frac{1}{\rho}-\frac{1}{\rho_-}\Big)
    \end{cases}
\end{align}
We seek a self-similar solution consisting of three contact discontinuities connecting a given left state $U_-$ to a given right state $U_+$. A direct calculation shows that such a solution exists if and only if
\begin{equation}
     u_++\frac{\sqrt{A(v_+)}}{\rho_+}>u_--\frac{\sqrt{A(v_-)}}{\rho_-} \label{conditions}
\end{equation}
holds. In this case, the unique solution is given by
\begin{equation}
    (\rho, u, v)(x, t)=\begin{cases}
        (\rho_L, u_L, v_L), \gap-\infty<\frac{x}{t}<\tau_1 \\
        (\rho_{M_1}, u_{M_1}, v_{M_1}), \gap \tau_1<\frac{x}{t}<\tau_2 \\
        (\rho_{M_2}, u_{M_2}, v_{M_2}), \gap\tau_2<\frac{x}{t}<\tau_3 \\
        (\rho_R, u_R, v_R), \gap \tau_3<\frac{x}{t}<\infty.
    \end{cases}
\end{equation}
The propagation speeds of the three contact discontinuities are given by \begin{align}
    \tau_1= u_L-\frac{\sqrt{A(v_L)}}{\rho_L},\ \ 
    \tau_2=u_{M_1}, \ \
    \tau_3=u_R+\frac{\sqrt{A(v_R)}}{\rho_R}.
\end{align}
The intermediate states are
\begin{equation}
    \begin{cases}
        \rho_{M_1}=\cfrac{\sqrt{A(v_L)}+\frac{A(v_L)}{\sqrt{A(v_R)}}}{u_R-u_L+\frac{\sqrt{A(v_R)}}{\rho_R}+\frac{\sqrt{A(v_L)}}{\rho_L}} \\[4ex]
        u_{M_1}=u_L-\frac{\sqrt{A(v_L)}}{\rho_L}+ \cfrac{\sqrt{A(v_L)}\left(u_R-u_L+{\frac{\sqrt{A(v_R)}}{\rho_R}}+\frac{\sqrt{A(v_L)}}{\rho_L}\right)}{\sqrt{A(v_L)}+\frac{A(v_L)}{\sqrt{A(v_R)}}}\\[2.5ex]
        v_{M_1}= v_L
    \end{cases}
\end{equation}
and
\begin{equation}
    \begin{cases}
        \rho_{M_2}=\cfrac{A(v_R)\left(\sqrt{A(v_L)}+\frac{A(v_L)}{\sqrt{A(v_R)}}\right)}{A(v_L)\left(u_R-u_L+\frac{\sqrt{A(v_R)}}{\rho_R}+\frac{\sqrt{A(v_L)}}{\rho_L}\right)} \\[4ex]
        u_{M_2}=u_L-\frac{\sqrt{A(v_L)}}{\rho_L}+ \cfrac{\sqrt{A(v_L)}\left(u_R-u_L+{\frac{\sqrt{A(v_R)}}{\rho_R}}+\frac{\sqrt{A(v_L)}}{\rho_L}\right)}{\sqrt{A(v_L)}+\frac{A(v_L)}{\sqrt{A(v_R)}}}\\[2.5ex]
        v_{M_2}= v_R
    \end{cases}
\end{equation}
Therefore, condition~\eqref{conditions} completely characterizes the region of right states for which the Riemann problem admits a unique classical solution consisting of three contact discontinuities.
\subsection{Overcompressive Delta Shocks} \label{analytic_singular}
Following the calculations from Section \ref{sec:singular}, one can verify that admissible $\delta$-wave speeds satisfy
\begin{equation}
\frac{[\rho]}{2}x^2(t)
-[\rho u]\,t\,x(t)
+\frac{1}{2}
\left[\rho u^2-\frac{A(v)}{\rho}\right]t^2
=0.
\end{equation}
If $[\rho]=0$, this equation is linear in $\frac{x(t)}{t}$ and has the unique solution 
\begin{equation}
    x(t)=\cfrac{[\rho u^2-\frac{A(v)}{\rho}]}{2[\rho u]}t=:\sigma_0 t
\end{equation}
If $[\rho]\neq0$, then the equation is quadratic in $\frac{x(t)}{t}$ with discriminant
\begin{equation}
    \Delta:=[\rho u]^2-[\rho][\rho u^2-\frac{A(v)}{\rho}]
\end{equation}
and solutions 
\begin{equation}
    x(t)=\frac{[\rho u]\pm\sqrt{\Delta}}{[\rho]}t=:\sigma_\pm t
\end{equation}
If a $\delta$-wave with speed $\sigma$ satisfies the overcompressibility condition
\begin{equation} \label{overcompression}
    u_--\frac{\sqrt{A(v_-)}}{\rho_-}=\lambda_1(U_-)\geq \sigma \geq \lambda_3(U_+)=u_++\frac{\sqrt{A(v_+)}}{\rho_+}
\end{equation}
then the corresponding $\delta$-wave is an admissible overcompressive $\delta$-shock. We now characterize the regions of state space in which the above $\delta$-waves satisfy the overcompressibility condition. For convenience, introduce the abbreviations
\begin{align}
    c_-:=\frac{\sqrt{A(v_-)}}{\rho_-},\ \
    c_+:=\frac{\sqrt{A(v_+)}}{\rho_+},\ \
    L:=c_+[\rho]-\rho_-[u], \ \
    U:=-c_-[\rho]-\rho_+[u].
\end{align}
\subsubsection{Delta Shocks with Speed $\sigma_+$} \label{analytic_singular_plus}
Assume first that $[\rho]\neq0$ and that the corresponding $\delta$-shock wave travels at speed $\sigma_+ =\frac{[\rho u]+\sqrt{\Delta}}{[\rho]}$. The first inequality in \eqref{overcompression} implies
\begin{align*}
    u_--c_-&\geq \frac{\rho_+u_+-\rho_- u_-}{\rho_+ -\rho_-}+\frac{\sqrt{\Delta}}{[\rho]} \\
    u_--c_-&\geq u_-+\rho_+\frac{[u]}{[\rho]}+\frac{\sqrt{\Delta}}{[\rho]}\\
    -c_-&\geq \rho_+\frac{[u]}{[\rho]}+\frac{\sqrt{\Delta}}{[\rho]}
\end{align*}
Similarly, the second overcompressive inequality implies
\begin{align*}
    u_++c_+&\leq \frac{\rho_+u_+-\rho_- u_-}{\rho_+ -\rho_-}+\frac{\sqrt{\Delta}}{[\rho]} \\
    u_++c_+&\leq u_++\rho_-\frac{[u]}{[\rho]}+\frac{\sqrt{\Delta}}{[\rho]} \\
    c_+&\leq \rho_-\frac{[u]}{[\rho]}+\frac{\sqrt{\Delta}}{[\rho]}
\end{align*} 
Combining these inequalities yields
\begin{align}
   c_+-\rho_-\frac{[u]}{[\rho]}&\leq\frac{\sqrt{\Delta}}{[\rho]}\leq -c_--\rho_+\frac{[u]}{[\rho]}
\end{align}
Suppose first that $[\rho]>0$. Since
$\sqrt{\Delta}/[\rho]\ge0$, the upper bound implies
\begin{equation}
-c_--\rho_+\frac{[u]}{[\rho]}\ge0,
\end{equation}
and therefore $[u]<0$. Recalling the definitions
\begin{equation}
L=c_+[\rho]-\rho_-[u],
\qquad
U=-c_-[\rho]-\rho_+[u],
\end{equation}
multiplying by $[\rho]>0$ gives $L\le\sqrt{\Delta}\le U.$ In this case, $L>0$ and $U>0$, so the overcompressibility condition is equivalent to
\begin{equation}
L^2\le\Delta\le U^2.
\label{Delta_Ineq_Plus}
\end{equation}
Define $X:=[u]+c_++c_-.$ A direct computation shows that \eqref{Delta_Ineq_Plus} is equivalent to
\begin{align}
    \begin{cases}
        X(X-2c_+)\geq 0 \\
        X(X-2c_-)\geq0
    \end{cases}
\end{align}
which, given that $[u]<0$, has the unique solution $X\leq0$, which is equivalent to 
\begin{align} \label{overcomp_plus}
    u_++\frac{\sqrt{A(v_+)}}{\rho_+}\leq u_--\frac{\sqrt{A(v_-)}}{\rho_-}
\end{align}
If instead $[\rho]<0$, the same computation yields the same condition. Consequently, whenever \eqref{overcomp_plus} holds, the Riemann problem admits an overcompressive $\delta$-shock with
\begin{align}
    \begin{cases}
        x(t)=t\sigma_+
        \\w(t)=t\sqrt{\Delta}
        \\v_\delta(t)=\cfrac{\sigma_+[\rho v]-[\rho u v]}{\sqrt{\Delta}}
        \\u_\delta(t)=\sigma_+
    \end{cases}
\end{align}
where the values $w(t)$, $v_\delta$, and $u_\delta(t)$ are derived from system \eqref{delta_ODEs} by a procedure analogous to Section \ref{sec:singular}. 
\subsubsection{Delta Shocks with Speed $\sigma_-$} \label{analytic_singular_minus}
Assume again that $[\rho]\neq0$, and consider the $\delta$-wave with speed $\sigma_-=x(t)=\frac{[\rho u]-\sqrt{\Delta}}{[\rho]}t$. Inequalities \eqref{overcompression} imply
\begin{align}\label{overcomp_minus_one}
    -c_-&\geq \rho_+\frac{[u]}{[\rho]}-\frac{\sqrt{\Delta}}{[\rho]}\\
    \label{overcomp_minus_two}
    c_+&\leq \rho_-\frac{[u]}{[\rho]}-\frac{\sqrt{\Delta}}{[\rho]}
\end{align}
Suppose first that $[\rho]>0$. Then, because the left-hand side of \eqref{overcomp_minus_two} is nonnegative, it follows that $\frac{[u]}{[\rho]}>0$, which implies $[u]>0$. This implies
\begin{align}
    u_++\frac{\sqrt{A(v_+)}}{\rho_+}>u_--\frac{\sqrt{A(v_-)}}{\rho_-}
\end{align}
contradicting the overcompressibility condition. Similarly, if $[\rho]<0$, we have $[u]>0$ by reasoning similar to above, again leading to the same contradiction. Thus, no admissible overcompressive $\delta$-shock can propagate with speed $\sigma_-$.

\subsubsection{Delta Shocks with Speed $\sigma_0$} \label{analytic_singular_null}
Finally, suppose $[\rho]=0$ and $[u]\neq0$. The unique $\delta$-wave speed is $\sigma_0=\cfrac{[\rho u^2-\frac{A(v)}{\rho}]}{2[\rho u]}$. Inequalities \eqref{overcompression} imply
\begin{align*}
    \sigma_0-\lambda_1(U_-)&=\cfrac{[\rho u^2-\frac{A(v)}{\rho}]-2[\rho u]u_-+2[\rho u]c_-}{2[\rho u]}\\
    &=\cfrac{\big([u]+c_--c_+\big)\big([u]+c_-+c_+\big)}{2[u]} \leq 0
\end{align*}
and
\begin{align*}
    \sigma_0-\lambda_3(U_+)&=\cfrac{[\rho u^2-\frac{A(v)}{\rho}]-2[\rho u]u_+-2[\rho u]c_+}{2[\rho u]}\\
    &=-\cfrac{\big([u]-c_-+c_+\big)\big([u]+c_-+c_+\big)}{2[u]}\geq 0
\end{align*}
Note first that $[u]>0$ is impossible because this contradicts \eqref{overcompression}, so consider the case $[u]<0$. Then we have the system
\begin{align}
    \begin{cases}
        X(X-2c_+)\geq0 \\
        X(X-2c_-)\geq 0
    \end{cases}
\end{align}
which has the sole solution $X\leq 0$; therefore, in the case $[\rho]=0$, there exists an overcompressive $\delta$-shock if and only if \eqref{overcomp_plus} is satisfied, and this $\delta$-shock is given by 
\begin{align}
    \begin{cases}
        x(t)=\sigma_0 t
        \\w(t)=-\rho[u]t
        \\v_\delta(t)=v_++\cfrac{[v]}{[u]}(u_--\sigma_0)
        \\u_\delta(t)=\sigma_0
    \end{cases}
\end{align}
where the values $w(t)$, $v_\delta$, and $u_\delta(t)$ are again derived from system \eqref{delta_ODEs} by following a similar line of reasoning to Section \ref{sec:singular}.
Thus, if the initial data satisfy \eqref{overcomp_plus}, the solution to the Riemann problem is solvable via a single $\delta$-shock wave
\begin{equation}
\left\{
\begin{aligned}
\rho(x,t)
=
\begin{cases}
\rho_-, & x<x(t),\\[1mm]
w(t)\,\delta\bigl(x-x(t)\bigr), & x=x(t),\\[1mm]
\rho_+, & x>x(t),
\end{cases}
\\[4mm]
(u,v)(x,t)
=
\begin{cases}
(u_-,v_-), & x<x(t),\\[1mm]
(u_+,v_+), & x>x(t)
\end{cases}
\end{aligned}
\right.
\end{equation}
\subsection{Solution of the Riemann Problem}
Combining the results of the previous subsections yields the following complete characterization of the Riemann problem.

\begin{theorem}
Suppose $\alpha=1$. If $u_++\frac{\sqrt{A(v_+)}}{\rho_+}>
u_--\frac{\sqrt{A(v_-)}}{\rho_-},$
then the unique solution to the Riemann problem consists of three contact discontinuities. If $u_++\frac{\sqrt{A(v_+)}}{\rho_+}
\le
u_--\frac{\sqrt{A(v_-)}}{\rho_-},$
then the unique solution to the Riemann problem is the overcompressive $\delta$-shock constructed in Section~\ref{analytic_singular_plus} (or Section~\ref{analytic_singular_null} when $[\rho]=0$).
\end{theorem}

\section{Dafermos Regularization and Blow-Up Analysis} \label{sec:GSPT}
In this section, we establish the existence of singular profiles corresponding to isolated overcompressive $\delta$-shocks (or, equivalently, shadow waves). We do not consider composite Riemann solutions that involve both classical and singular waves. We consider a Dafermos regularization (to preserve self-similarity) of the system and rewrite the resulting equations as a first-order dynamical system. We then use a spherical blow-up on the hypersphere $S^3$ to analyze the dynamics near the singularity and construct the desired connecting orbit.
\subsection{Dafermos Regularization}
Following Dafermos \cite{Dafermos}, we consider the self-similar viscous regularization of system \eqref{cons1}--\eqref{cons3}. 
\begin{equation}
    \begin{cases}
        \rho_t + (\rho u)_x = \eps t\rho_{xx}, \\\displaystyle
        (\rho u)_t + \left(\rho u^2 - \frac{A(v)}{\rho^\alpha}\right)_x = \eps t(\rho u)_{xx}, \\\displaystyle
        (\rho v)_t + (\rho u v)_x = \eps t(\rho v)_{xx}.
    \end{cases} \label{GSPT_sys}
\end{equation}
Introducing the self-similar variable $\xi = \frac{x}{t}$ and seeking solutions depending only on $\xi,$ we obtain $$
U_x=\frac{1}{t}U_\xi,\qquad
U_{xx}=\frac{1}{t^2}U_{\xi\xi},\qquad
U_t=-\frac{\xi}{t}U_\xi.
$$ Substituting these expressions into \eqref{GSPT_sys} yields 
\begin{equation}
\begin{cases}\displaystyle
    -\xi\rho_\xi+(\rho u)_\xi=\epsilon \rho_{\xi\xi}, \\\displaystyle
    -\xi(\rho u)_\xi+\left(\rho u^2-\frac{A(v)}{\rho^{\alpha}}\right)_\xi=\epsilon (\rho u)_{\xi\xi}, \\\displaystyle
    -\xi(\rho v)_\xi+(\rho u v)_\xi=\epsilon (\rho v)_{\xi\xi}.
\end{cases}\label{eq:selfsimilar}
\end{equation}
The left and right states of the Riemann problem are
$U_L=(\rho_L,u_L,v_L),\ U_R=(\rho_R,u_R,v_R),$
and we seek a profile satisfying
$$\lim_{\xi\to-\infty}(\rho,u,v)(\xi)=U_L,
\qquad \lim_{\xi\to+\infty}(\rho,u,v)(\xi)=U_R.$$
To rewrite \eqref{eq:selfsimilar} as a first-order system, we introduce
\begin{equation}
    V = \begin{pmatrix}
        v_1 \\ v_2 \\ v_3
    \end{pmatrix} = \begin{pmatrix}
        \eps \frac{d\rho}{d\xi} \\ \eps \frac{d}{d\xi} (\rho u) \\ \eps \frac{d}{d\xi} (\rho v)
    \end{pmatrix}= \eps\frac{d}{d\xi}\begin{pmatrix}
    \rho \\ \rho u\\\rho v \end{pmatrix}
\end{equation}
and shift the self-similar variable by setting $\theta = \xi - c$
where $c$ denotes the speed of the singular shock (i.e. $\sigma_+$ or $\sigma_0$ depending on the value of $[\rho]$, see Section \ref{sec:singular}). This shift centers the profile on the $\delta$-shock, allowing us to study the singular behavior of the system more easily. With the additional variables $m = \rho u, \ n = \rho v,$
the flux takes the form $F = \begin{pmatrix}
    m \\ \frac{m^2}{\rho} - \frac{A(\frac{n}{\rho})}{\rho^\alpha} \\ \frac{mn}{\rho}
\end{pmatrix},$ which is essentially $G$ but with the change of variables. The regularized system \begin{equation}
    \begin{cases}\displaystyle
        \rho_t + m_x = \eps t \rho_{xx}, \\\displaystyle
        m_t + \left(\frac{m^2}{\rho} - A\left(\frac{n}{\rho}\right) \frac{1}{\rho^\alpha}\right)_x = \eps t m_{xx}, \\\displaystyle
        n_t + \left(\frac{mn}{\rho}\right)_x = \eps t n_{xx}.
    \end{cases}
\end{equation} becomes 
\begin{align}\label{GSPT_slow_sys}
        \epsilon\frac{dH}{d\theta}=V, \ \ \
        \epsilon\frac{dV}{d\theta}=(DF-\xi I)V, \ \ \
        \frac{d\xi}{d\theta}=1.
\end{align}
\subsection{Fast and Slow Systems}
We define the fast time variable
$\tau=\frac{\theta}{\varepsilon}.$
For $\varepsilon>0$, the slow system \eqref{GSPT_slow_sys} is equivalent to
\begin{align} \label{GSPT_fast_sys}
        \frac{dH}{d\tau}=V, \ \ \
        \frac{dV}{d\tau}=(DF-\xi I)V, \ \ \
        \frac{d\xi}{d\tau}=\epsilon.
\end{align}
When $\varepsilon=0$, the set
$S=\{(H,V,\xi):V=0\}$
is a four-dimensional manifold of equilibria. The linearization at points of $S$ has four zero eigenvectors tangent to $S$, while the three normal eigenvalues are 
$\lambda_i(H)-\xi,\ i=1,2,3,$
where $\lambda_i(H)$ are the characteristic speeds of the original system. Thus $S$ is normally hyperbolic away from the surfaces $\xi=\lambda_i(H).$ We define
$S_0=\{(H,V,\xi):V=0,\ \xi<\lambda_1(H)\},$
and
$S_2=\{(H,V,\xi):V=0,\ \xi>\lambda_3(H)\}.$
On $S_0$, all nonzero eigenvalues are positive, while on $S_2$, all nonzero eigenvalues are negative. Therefore, $S_0$ and $S_2$ are normally hyperbolic invariant manifolds. Consequently, for sufficiently small $\varepsilon>0$, they admits locally invariant perturbations together with their associated local stable and unstable manifolds.

For the fixed left state $H_L$, we define
$S_0(H_L)=\{(H,V,\xi):H=H_L,V=0,\xi<\lambda_1(H_L)\}.$
For sufficiently small $\varepsilon$, Fenichel theory yields a
four-dimensional unstable manifold
$\mathcal W_\varepsilon^u(S_0(H_L))$,
which is a smooth perturbation of
$\mathcal W_0^u(S_0(H_L))
=\Bigl\{(H,V,\xi):H\in\Omega_\xi,V=V(H),
\xi<\lambda_1(H_L)\Bigr\},$
where $\Omega_\xi$ is an open neighborhood of $H_L$ depending on $\xi$, and $V(H)$ is determined by the fast system \eqref{GSPT_fast_sys}.
Similarly, for the fixed right state $H_R$, we define
$S_2(H_R)=\{(H,V,\xi):H=H_R,\;V=0,\;\xi>\lambda_3(H_R)\}.$
For sufficiently small $\varepsilon$, Fenichel theory yields a
four-dimensional stable manifold
$\mathcal W_\varepsilon^s(S_2(H_R))$,
which is a smooth perturbation of
$\mathcal W_0^s(S_2(H_R))
=\Bigl\{(H,V,\xi):H\in\Omega_\xi,\;V=V(H),\;
\xi>\lambda_3(H_R)\Bigr\}.$
Since $\varepsilon$ is treated as a state variable in the blown-up
system, we also use the extended manifolds
\begin{equation}
\widehat{\mathcal W}_L^u
=\bigcup_{0\le\varepsilon\le\varepsilon_0}
\mathcal W_\varepsilon^u(S_0(H_L)),\qquad\widehat{\mathcal W}_R^s
=\bigcup_{0\le\varepsilon\le\varepsilon_0}\mathcal W_\varepsilon^s(S_2(H_R)).
\label{eq:extended-outer-manifolds}
\end{equation}
Each is a five-dimensional invariant manifold. Its intersection with a
fixed level set $\{\varepsilon=\mathrm{constant}\}$ is the corresponding
four-dimensional stable or unstable manifold.
\subsection{Blow-Up Transformation}
Following \cite{REU2025GSPT,Sc}, we introduce
$W=F(H)-\xi H-V,$
where $W=(w_1,w_2,w_3)^T$. This change of variables, inspired by the Rankine--Hugoniot relation, yields an equivalent first-order formulation of the self-similar system that is convenient for the subsequent blow-up analysis. Treating \(\varepsilon\) as an additional state variable, system \eqref{GSPT_fast_sys} becomes
\begin{equation}
    \frac{dH}{d\tau} = (F(H) - \xi H) - W, \ \ \ 
        \frac{dW}{d\tau} = -\epsilon H, \ \ \ 
        \frac{d\xi}{d\tau} = \epsilon, \ \ \
        \frac{d\epsilon}{d\tau} = 0.
\end{equation}
Expressed in the variables
$(\rho,u,v,w_1,w_2,w_3,\xi,\varepsilon)$,
this system is
\begin{equation}
    \begin{cases}
    \displaystyle
        \frac{d\rho}{d\tau} = \rho u - \xi \rho  - w_1, \\[1.5ex]\displaystyle
        \frac{du}{d\tau} = \frac{uw_1 - w_2}{\rho} - \frac{A(v)}{\rho^{\alpha + 1}}, \\[1.5ex]\displaystyle
        \frac{dv}{d\tau} = \frac{vw_1 - w_3}{\rho}, \\[1.5ex]\displaystyle
        \frac{dw_1}{d\tau} = -\eps \rho, \\[1.5ex]\displaystyle
        \frac{dw_2}{d\tau} = -\epsilon \rho u, \\[1.5ex]\displaystyle
        \frac{dw_3}{d\tau} = -\eps \rho v, \\[1.5ex]\displaystyle
        \frac{d\xi}{d\tau} = \eps, \\[1.5ex]\displaystyle
        \frac{d\eps}{d\tau} = 0.
    \end{cases} \label{GSPT_diffeq_sys}
\end{equation}
At $\varepsilon=0$, equations $\frac{dw_1}{d\tau}=0, \ \frac{dw_2}{d\tau}=0, \ \frac{dw_3}{d\tau}=0, \ \frac{d\xi}{d\tau}=0$
show that the reduced unstable manifold associated with the left state
has the local representation
\begin{equation}
\mathcal W_0^u(S_0(H_L))=\left\{\left(H,\,F(H_L)-\xi H_L,\,\xi,\,0\right):H\in\Omega_\xi\right\}.
\label{eq:left-outer-manifold-explicit}
\end{equation}
Consequently, its tangent space is generated by the three independent
$H$-variation vectors $E_j^L=\left(e_j,\,0,\,0,\,0\right),$ $\ j=1,2,3,$
and the equilibrium-line vector $E_\xi^L=\left(0,\,-H_L,\,1,\,0\right).$
Here the entries are ordered as $(H,W,\xi,\varepsilon)$.
The extended manifold has one additional tangent vector obtained by
variation of $\varepsilon$. Similarly,
\begin{equation}
\mathcal W_0^s(S_2(H_R))=\left\{\left(H,\,F(H_R)-\xi H_R,\,\xi,\,0\right):H\in\Omega_\xi\right\},
\label{eq:right-outer-manifold-explicit}
\end{equation}
with the corresponding tangent generators $E_j^R=(e_j,0,0,0),\ E_\xi^R=(0,-H_R,1,0).$

Our objective is to construct an orbit of the blown-up system joining the unstable manifold associated with the left state to the stable manifold associated with the right state. Such an orbit is a heteroclinic connection satisfying $H(\tau)\to H_L$ as $\tau\to-\infty$ and $H(\tau)\to H_R$ as $\tau\to+\infty$, and therefore corresponds to a self-similar viscous profile of the Dafermos regularization.

The analysis of the singular shock in Section~\ref{sec:singular} shows that the density along the singular profile satisfies
$\rho=O(\varepsilon^{-1})$ as $\varepsilon\to0^+,$ therefore, we introduce the rescaled density $\rho_0=\varepsilon\rho,$ which remains bounded in the singular limit. In terms of the variables $(\rho_0,u,v,w_1,w_2,w_3,\xi,\varepsilon)$,
system \eqref{GSPT_diffeq_sys} becomes
\begin{equation}
    \begin{cases}
    \displaystyle
        \frac{d\rho_0}{d\tau} = \rho_0 u - \xi \rho_0  - \eps w_1, \\[1.5ex]\displaystyle
        \frac{du}{d\tau} = \frac{(uw_1 - w_2)\eps}{\rho_0} - \frac{A(v)\eps^{\alpha + 1}}{\rho_0^{\alpha + 1}}, \\[1.5ex]\displaystyle
        \frac{dv}{d\tau} = \frac{(vw_1 - w_3)\eps}{\rho_0}, \\[1.5ex]\displaystyle
        \frac{dw_1}{d\tau} = -\rho_0, \\[1.5ex]\displaystyle
        \frac{dw_2}{d\tau} = -\rho_0 u, \\[1.5ex]\displaystyle
        \frac{dw_3}{d\tau} = -\rho_0 v, \\[1.5ex]\displaystyle
        \frac{d\xi}{d\tau} = \eps, \\[1.5ex]\displaystyle
        \frac{d\eps}{d\tau} = 0.
    \end{cases} \label{GSPT_eps_sys}
\end{equation}
The introduction of the variable $\rho_0$ removes the unbounded growth of the density and allows the vector field to be desingularized by multiplying the equations by a suitable power of $\rho_0$. Nevertheless, the resulting system has a nonhyperbolic equilibrium set at
$(\rho_0,\varepsilon)=(0,0),$
where the linearization contains additional zero eigenvalues. Consequently, the standard invariant manifold theory cannot be applied directly in this region. To resolve this loss of normal hyperbolicity, we perform a spherical blow-up following \cite{Sc},
\begin{equation}
    \rho_0 = r\overline{\rho}, \quad u - \xi = r\overline{u}, \quad v - \frac{k_3}{k_1} = r\overline{v}, \quad \sqrt{\eps} = r\overline{\eps}
\end{equation}
such that they lie on the hypersphere $S^3$, i.e. $    \overline{\rho}^2 + \overline{u}^2 + \overline{v}^2 + \overline{\eps}^2 = 1 \label{spherical}.$ The shift in the variables $u$ and $v$ centers the blow-up at the singular shock, where $u=u_\delta = \frac{k_2}{k_1} = \xi$ and $v=v_\delta = \frac{k_3}{k_1}$, respectively.

The blow-up replaces the singular point $(\rho_0,\varepsilon)=(0,0)$ by the hypersphere $S^3$, thereby magnifying the neighborhood of the concentrated shock layer and separating the different directions along which trajectories approach the singularity. This transformation resolves the singularity and enables analysis of the internal structure of the viscous profile via the associated reduced dynamical system. The reduced system obtained by setting $r=0$ describes the leading-order dynamics of the inner viscous profile corresponding to the singular shock.

The blow-up separates the analysis into two distinct parts. Away from $r=0$, the dynamics are governed by the persistent unstable and stable manifolds $\mathcal W_\varepsilon^u(S_0(H_L))$ and
$\mathcal W_\varepsilon^s(S_2(H_R))$, whose existence follows from
Fenichel theory. The role of the blow-up is solely to resolve the
loss of normal hyperbolicity occurring as $\rho\rightarrow\infty$, corresponding to $(\rho_0,\varepsilon)=(0,0)$. The local dynamics in this neighborhood are analyzed in three directional charts. The corresponding transition maps describe the passage between the incoming and outgoing outer trajectories through the blown-up neighborhood.

Substituting the blow-up transformation into system \eqref{GSPT_eps_sys} and differentiating with respect to $\tau$ yields
\begin{equation}
\renewcommand{\arraystretch}{1.5}
\begin{cases}
\displaystyle
\frac{d\overline{\rho}}{d\tau}=-\frac{\overline{\rho}}{r}\frac{dr}{d\tau}
+r\overline{\rho}\,\overline{u}-r\overline{\eps}^{2}w_{1},
\\\displaystyle
\frac{d\overline{u}}{d\tau}=-\frac{\overline{u}}{r}\frac{dr}{d\tau}
+\frac{\big[(\xi+r\overline{u})w_{1}-w_{2}\big]\overline{\eps}^{2}}{\overline{\rho}}-\frac{A\!\left(\frac{k_{3}}{k_{1}}+r\overline{v}\right)r^{\alpha}\overline{\eps}^{2\alpha+2}}{\overline{\rho}^{\alpha+1}}-r\overline{\eps}^{2},\\
\displaystyle
\frac{d\overline{v}}{d\tau}=
-\frac{\overline{v}}{r}\frac{dr}{d\tau}
+\frac{\big[\left(\frac{k_{3}}{k_{1}}+r\overline{v}\right)w_{1}-w_{3}\big]\overline{\eps}^{2}}{\overline{\rho}},
\\
\displaystyle
\frac{d\overline{\eps}}{d\tau}=-\frac{\overline{\eps}}{r}\frac{dr}{d\tau}.
\end{cases}
\end{equation}
By \eqref{spherical}, the derivatives of $\overline{\rho}$, $\overline{u}$, $\overline{v}$, $\overline{\eps}$ should satisfy
$2\overline{\rho} \cdot \frac{d\overline{\rho}}{d\tau} + 2\overline{u} \cdot \frac{d\overline{u}}{d\tau} + 2\overline{v} \cdot \frac{d\overline{v}}{d\tau} + 2\overline{\eps} \cdot \frac{d\overline{\eps}}{d\tau} = 0$
allowing us to solve for $\frac{dr}{d\tau}$ explicitly. After doing so, we have the following system of equations, in which $r$ appears as a dependent variable
\begin{align} \label{general_GSPT_long}
\left\{
\begin{aligned}
\frac{dr}{d\tau}
&= r^2\bar{\rho}^2\bar{u}-r^2\bar{\rho}\bar{\eps}^2w_1
+\frac{r\bar{u}\bar{\eps}^2[(\xi+r\bar{u})w_1-w_2]}{\bar{\rho}}
-\frac{r^{\alpha+1}\bar{u}\bar{\eps}^{2\alpha+2}A(\frac{k_3}{k_1}+r\bar{v})}{\bar{\rho}^{\alpha+1}}
\\
&\qquad
-r^2\bar{\eps}^2\bar{u}
+\frac{r\bar{v}\bar{\eps}^2[(\frac{k_3}{k_1}+r\bar{v})w_1-w_3]}{\bar{\rho}},
\\[0.7ex]
\frac{d\bar{\rho}}{d\tau}
&=-r\bar{\rho}^3\bar{u}
+r\bar{\rho}^2\bar{\eps}^2w_1
-\bar{u}\bar{\eps}^2[(\xi+r\bar{u})w_1-w_2]
+\frac{r^\alpha\bar{u}\bar{\eps}^{2\alpha+2}A(\frac{k_3}{k_1}+r\bar{v})}{\bar{\rho}^\alpha}
\\
&\qquad
+\bar{\rho}r\bar{\eps}^2\bar{u}
-\bar{\eps}^2\bar{v}\!\left[(\tfrac{k_3}{k_1}+r\bar{v})w_1-w_3\right]
+r\bar{\rho}\bar{u}-r\bar{\eps}^2w_1,
\\[0.7ex]
\frac{d\bar{u}}{d\tau}
&=-r\bar{u}^2\bar{\rho}^2
+r\bar{u}\bar{\rho}\bar{\eps}^2w_1
-\frac{\bar{u}^2\bar{\eps}^2[(\xi+r\bar{u})w_1-w_2]}{\bar{\rho}}
+\frac{r^\alpha\bar{u}^2\bar{\eps}^{2\alpha+2}A(\frac{k_3}{k_1}+r\bar{v})}{\bar{\rho}^{\alpha+1}}
+r\bar{\eps}^2\bar{u}^2
\\
&\qquad
-\frac{\bar{u}\bar{v}\bar{\eps}^2[(\frac{k_3}{k_1}+r\bar{v})w_1-w_3]}{\bar{\rho}}
+\frac{\bar{\eps}^2[(\xi+r\bar{u})w_1-w_2]}{\bar{\rho}}
-\frac{r^\alpha\bar{\eps}^{2\alpha+2}A(\frac{k_3}{k_1}+r\bar{v})}{\bar{\rho}^{\alpha+1}}
-r\bar{\eps}^2,
\\[0.7ex]
\frac{d\bar{v}}{d\tau}
&=-r\bar{v}\bar{\rho}^2\bar{u}
+r\bar{v}\bar{\rho}\bar{\eps}^2w_1
-\frac{\bar{v}\bar{u}\bar{\eps}^2[(\xi+r\bar{u})w_1-w_2]}{\bar{\rho}}
+\frac{r^\alpha\bar{u}\bar{v}\bar{\eps}^{2\alpha+2}A(\frac{k_3}{k_1}+r\bar{v})}{\bar{\rho}^{\alpha+1}}
+r\bar{v}\bar{\eps}^2\bar{u}
\\
&\qquad
-\frac{\bar{\eps}^2\bar{v}^2[(\frac{k_3}{k_1}+r\bar{v})w_1-w_3]}{\bar{\rho}}
+\frac{\bar{\eps}^2[(\frac{k_3}{k_1}+r\bar{v})w_1-w_3]}{\bar{\rho}},
\\[0.7ex]
\frac{d\bar{\eps}}{d\tau}
&=-r\bar{\rho}^2\bar{u}\bar{\eps}
+r\bar{\rho}\bar{\eps}^3w_1
-\frac{\bar{u}\bar{\eps}^3[(\xi+r\bar{u})w_1-w_2]}{\bar{\rho}}
+\frac{r^\alpha\bar{u}\bar{\eps}^{2\alpha+3}A(\frac{k_3}{k_1}+r\bar{v})}{\bar{\rho}^{\alpha+1}}
+r\bar{\eps}^3\bar{u}
\\
&\qquad
-\frac{\bar{\eps}^3\bar{v}[(\frac{k_3}{k_1}+r\bar{v})w_1-w_3]}{\bar{\rho}},
\\[0.7ex]
\
\frac{dw_1}{d\tau}
&=-r\bar{\rho},
\qquad
\frac{dw_2}{d\tau}
=-r\bar{\rho}(\xi+r\bar{u}),
\\
\frac{dw_3}{d\tau}
&=-r\bar{\rho}\left(\frac{k_3}{k_1}+r\bar{v}\right),
\qquad
\frac{d\xi}{d\tau}
=r^2\bar{\eps}^2.
\end{aligned}
\right.
\end{align}
We next identify the points on the blown-up boundary selected by the
left and right reduced outer trajectories. Define
\begin{equation}
D_L:=cw_{1,L}-w_{2,L}=-B_L,
\qquad
\Gamma_L:=v_{\delta}w_{1,L}-w_{3,L},
\qquad v_{\delta}=\frac{k_3}{k_1} \qquad
R_L:=\sqrt{D_L^2+\Gamma_L^2},
\label{eq:left-direction-constants}
\end{equation}
where
\begin{equation}
W_L=(w_{1,L},w_{2,L},w_{3,L})^T
=
F(H_L)-cH_L.
\end{equation}
Similarly, define
\begin{equation}
D_R:=cw_{1,R}-w_{2,R}=-B_R,
\qquad
\Gamma_R:=v_{\delta}w_{1,R}-w_{3,R},
\qquad v_{\delta}=\frac{k_3}{k_1}\qquad
R_R:=\sqrt{D_R^2+\Gamma_R^2},
\label{eq:right-direction-constants}
\end{equation}
where
\begin{equation}
W_R=(w_{1,R},w_{2,R},w_{3,R})^T
=
F(H_R)-cH_R.
\end{equation}
\begin{lemma}[Generalized Rankine--Hugoniot relation for $W$]
\label{lem:outer-W-RH}
Let
$$
W_L=F(H_L)-cH_L,
\qquad
W_R=F(H_R)-cH_R.
$$
Then the generalized Rankine--Hugoniot conditions imply
\begin{equation}
W_R-W_L
=
-k_1
\begin{pmatrix}
1\\
c\\
v_\delta
\end{pmatrix},
\label{eq:outer-W-RH}
\end{equation}
where $k_1=c[\rho]-[\rho u]$ is the mass accumulation rate of the singular shock, see Section \ref{sec:singular}. Moreover, 
$D_L=D_R,
\ \ 
\Gamma_L=\Gamma_R$
\end{lemma}

\begin{proof}
By definition, $W_R-W_L
=
[F(H)]-c[H].$
Hence
\begin{equation}
W_R-W_L
=
\begin{pmatrix}
[\rho u]-c[\rho]\\[1mm]
\left[\rho u^2-\dfrac{A(v)}{\rho^\alpha}\right]
-c[\rho u]\\[2mm]
[\rho uv]-c[\rho v]
\end{pmatrix}.
\end{equation}
The generalized Rankine--Hugoniot conditions give
$
k_1=c[\rho]-[\rho u],
\
k_2=ck_1,
\
k_3=v_\delta k_1.
$ Therefore,
\begin{equation}
[\rho u]-c[\rho]=-k_1,
\end{equation}
\begin{equation}
\left[\rho u^2-\frac{A(v)}{\rho^\alpha}\right]
-c[\rho u]
=
-k_2
=
-ck_1,
\end{equation}
and
\begin{equation}
[\rho uv]-c[\rho v]
=
-k_3
=
-v_\delta k_1.
\end{equation}
This proves \eqref{eq:outer-W-RH}. Writing \eqref{eq:outer-W-RH} componentwise gives
\begin{equation}
w_{1,R}-w_{1,L}=-k_1, \ w_{2,R}-w_{2,L}=-ck_1, \ w_{3,R}-w_{3,L}=-v_\delta k_1.
\end{equation}
Consequently,
\begin{align}
D_R-D_L=
c(w_{1,R}-w_{1,L})
-
(w_{2,R}-w_{2,L})=
-c k_1+c k_1=
0,
\end{align}
and similarly,
\begin{align}
\Gamma_R-\Gamma_L=
v_\delta(w_{1,R}-w_{1,L})
-
(w_{3,R}-w_{3,L})=
-v_\delta k_1+v_\delta k_1=
0.
\end{align}
Thus $D_L=D_R,
\ \Gamma_L=\Gamma_R.$
\end{proof}
\subsection{Outline of the geometric construction}
The construction of the viscous profile naturally separates into two parts. We first construct a singular concatenated orbit of the reduced blown-up problem corresponding to $\varepsilon=0$. It consists of three pieces:
\begin{enumerate}
\item a left reduced outer orbit issuing from the equilibrium manifold associated with the left state $H_L$ and approaching the blown-up boundary, where $\rho\to+\infty$, $u\to c$, and $v\to v_\delta$;
\item   an explicit reduced orbit lying on the blown-up boundary and connecting the limiting point selected by the left reduced outer orbit to the limiting point selected by the right reduced outer orbit;
\item a right reduced outer orbit leaving the singular region and converging to the equilibrium manifold associated with the right state $H_R$.
\end{enumerate}

The left and right reduced outer orbits are obtained from the reduced system corresponding to $\varepsilon=0$, for which $\xi=c$ and $W$ remains constant along trajectories. The middle orbit is obtained from the reduced dynamics on the blown-up boundary. By Lemma~\ref{lem:outer-W-RH}, the left and right reduced outer trajectories approach limiting directions on the blown-up boundary that share the same values of the invariants $D:=cw_1-w_2$ and $\Gamma:=v_{\delta}w_1-w_3$. By the weighted blow-down relations, the asymptotic limits of the reduced outer trajectories are expressed in the weighted coordinates $(\beta,y,h)$, where they converge to the incoming and outgoing equilibria of the weighted corner system. The remaining mismatch is encoded solely by the scalar quantity $k_L-k_R$, which is determined by the middle matching condition. Consequently, the reduced middle orbit provides the continuation of the left reduced outer trajectory to the right reduced outer trajectory through the blown-up boundary. The construction of the reduced outer orbits is first carried out for $0<\alpha<1$. The endpoint case $\alpha=1$ requires a separate argument because some of the strict-concavity properties used become affine.

For positive $\varepsilon$, the outer orbit segments persist up to
fixed sections away from the singular corner by Fenichel theory and
smooth dependence of the flow. The entry and exit charts resolve the
approach to and departure from the blown-up boundary. The distinguished
middle orbit is constructed in the inner scaling and represented in the
positive-density chart $K_\rho$. The central singular orbit is nonhyperbolic in the positive-density coordinates. Consequently, its positive-$\varepsilon$ persistence cannot be obtained from ordinary transverse heteroclinic persistence. A weighted central-corner analysis is required to identify incoming and outgoing invariant curves, and determine the associated tangent spaces to resolve the two endpoint transitions and provide the setting for the application of Schecter's Corner Lemma.

Persistence will be established for $\alpha\in\mathbb{Q}\cap(0,1].$ For $\alpha=\frac12$ and $\alpha=1$, the weighted central-corner
vector field is smooth without any additional change of variables.
For a general rational exponent, smoothness is restored by the
ramified weighted coordinate introduced in
Subsection~\ref{subsubsec:weighted-central-corner}. The outer-to-corner transversality conditions are verified in
Subsection~\ref{subsubsec:outer-corner-transversality}. Schecter's
Corner Lemma then propagates the incoming and outgoing manifolds
through the two corner neighborhoods. Finally, the nondegeneracy of
the scalar middle matching condition permits the two propagated
manifolds to be matched for sufficiently small positive
$\varepsilon$.
\begin{figure}[H]
    \centering
    \begin{minipage}[b]{0.48\textwidth}
        \centering
    \includegraphics[width=\linewidth]{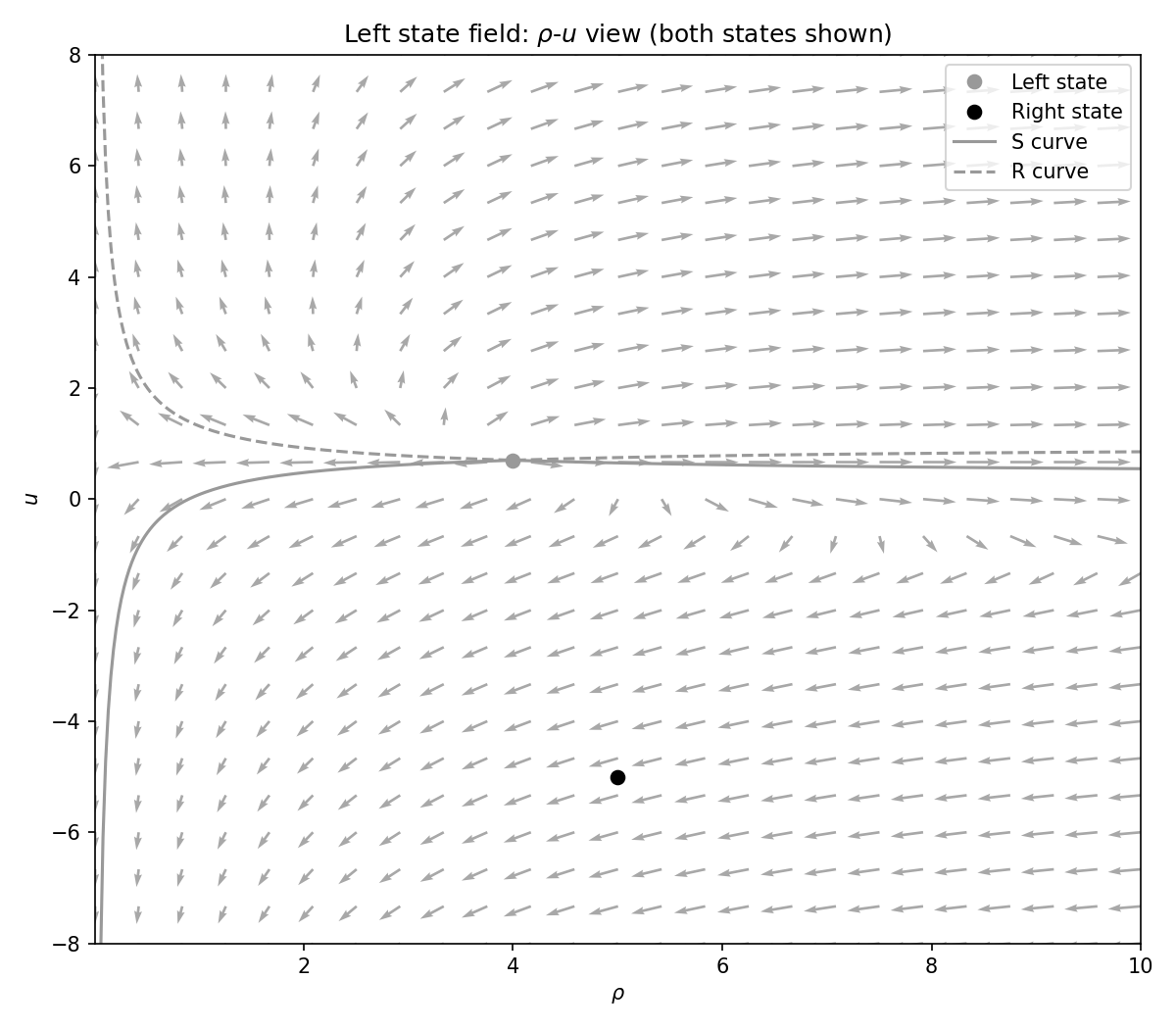}
    \caption{Left state invariant region of Region 6, Case 1: $(u_L,  \rho_L, v_L) = (0.7, 4, -1.9)$.}
    \label{fig:combinedleft}
    \end{minipage}
    \hfill
    \begin{minipage}[b]{0.48\textwidth}
        \centering
        \includegraphics[width=\linewidth]{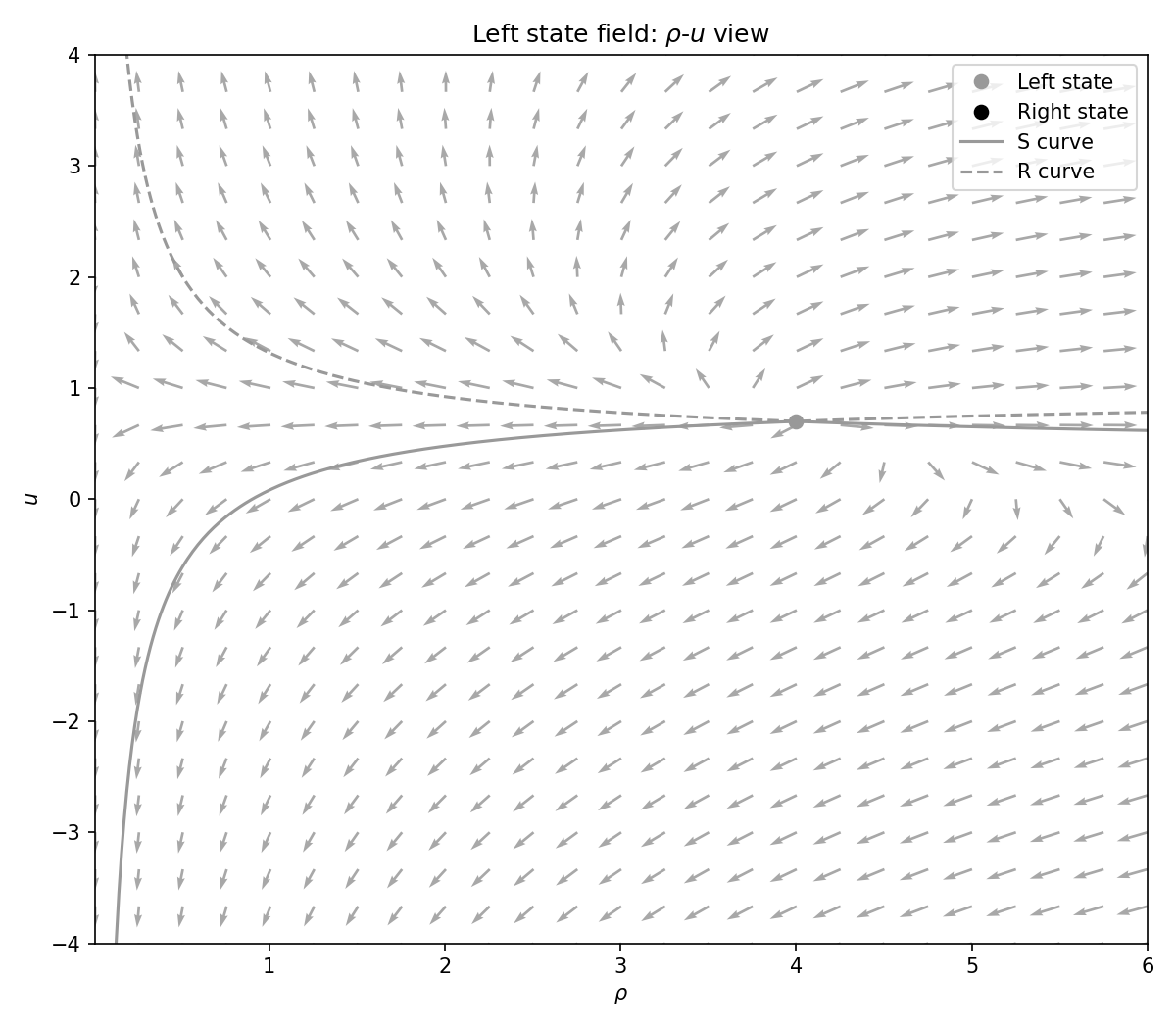}
        \caption{Zoom view of left state invariant region of Case 1.}
        \label{fig:zoomleft}
    \end{minipage}
\end{figure}
\begin{lemma}
\label{lem:sign_xiw}

Assume that the singular-shock speed $c$ satisfies the stronger condition:
\begin{equation}
u_R+\sqrt{\frac{A(v_R)}{\rho_R^{\alpha+1}}}<c<
u_L-\sqrt{\frac{A(v_L)}{\rho_L^{\alpha+1}}}.
\label{eq:overcompressive-speed}
\end{equation}
Since $0<\alpha\le1$, this condition implies the overcompressivity
inequalities
$\lambda_3(H_R)<c<\lambda_1(H_L)$.
Let $W_L=F(H_L)-cH_L, \ W_R=F(H_R)-cH_R,$ and define $D_L=cw_{1,L}-w_{2,L}, \ D_R=cw_{1,R}-w_{2,R}.$ Then $D_L=-\rho_L(u_L-c)^2
+\frac{A(v_L)}{\rho_L^\alpha}<0$ and $D_R=-\rho_R(u_R-c)^2+
\frac{A(v_R)}{\rho_R^\alpha}<0.$ Equivalently, the quantities $B_L
:=\rho_L(u_L-c)^2
-\frac{A(v_L)}{\rho_L^\alpha}$ and $B_R
:=\rho_R(u_R-c)^2
-\frac{A(v_R)}{\rho_R^\alpha}$ are strictly positive.
\end{lemma}

\begin{proof}
Since
\begin{equation}
H_L
=
\begin{pmatrix}
\rho_L\\
\rho_Lu_L\\
\rho_Lv_L
\end{pmatrix}\ \ 
\text{and} \ \ 
F(H_L)
=
\begin{pmatrix}
\rho_Lu_L\\
\rho_Lu_L^2-\dfrac{A(v_L)}{\rho_L^\alpha}\\
\rho_Lu_Lv_L
\end{pmatrix},
\end{equation}
the identity $W_L=F(H_L)-cH_L$ gives $w_{1,L}
=\rho_L(u_L-c)$ and $w_{2,L}
=\rho_Lu_L(u_L-c)
-\frac{A(v_L)}{\rho_L^\alpha}.$ Therefore,
\begin{align}
D_L&=cw_{1,L}-w_{2,L}=
c\rho_L(u_L-c)
-\rho_Lu_L(u_L-c)
+\frac{A(v_L)}{\rho_L^\alpha}=
-\rho_L(u_L-c)^2
+\frac{A(v_L)}{\rho_L^\alpha}.
\label{eq:DL-formula}
\end{align}
The right-hand inequality in \eqref{eq:overcompressive-speed} gives
$u_L-c>
\sqrt{\frac{A(v_L)}{\rho_L^{\alpha+1}}}.$ Both sides are positive, thus squaring and multiplying by $\rho_L>0$ yields $\rho_L(u_L-c)^2
>\frac{A(v_L)}{\rho_L^\alpha}.$ Consequently, \eqref{eq:DL-formula} implies $D_L<0.$ Equivalently, $$B_L=\rho_L(u_L-c)^2
-\frac{A(v_L)}{\rho_L^\alpha}
=-D_L>0.$$ The corresponding computation at the right state gives
$w_{1,R}=
\rho_R(u_R-c)$ and $w_{2,R}
=\rho_Ru_R(u_R-c)-\frac{A(v_R)}{\rho_R^\alpha}.$ Hence
\begin{align}
D_R
&=cw_{1,R}-w_{2,R}=
c\rho_R(u_R-c)
-\rho_Ru_R(u_R-c)
+\frac{A(v_R)}{\rho_R^\alpha}=
-\rho_R(u_R-c)^2
+\frac{A(v_R)}{\rho_R^\alpha}.
\label{eq:DR-formula}
\end{align}
The left-hand inequality in \eqref{eq:overcompressive-speed} gives
$c-u_R>\sqrt{\frac{A(v_R)}{\rho_R^{\alpha+1}}}.$ Since $(u_R-c)^2=(c-u_R)^2$, squaring and multiplying by $\rho_R>0$
gives $\rho_R(u_R-c)^2
>\frac{A(v_R)}{\rho_R^\alpha}.$ It follows from \eqref{eq:DR-formula} that $D_R<0.$ Equivalently,
\begin{equation}
B_R=\rho_R(u_R-c)^2
-\frac{A(v_R)}{\rho_R^\alpha}
=-D_R>0.
\end{equation}
\end{proof}
The existence of a speed $c$ satisfying
\eqref{eq:overcompressive-speed} is equivalent to the strict inequality
\begin{equation}
u_R
+\sqrt{\frac{A(v_R)}{\rho_R^{\alpha+1}}}
<u_L-\sqrt{\frac{A(v_L)}{\rho_L^{\alpha+1}}}\implies
u_R<u_L-
\sqrt{\frac{A(v_L)}{\rho_L^{\alpha+1}}}
-
\sqrt{\frac{A(v_R)}{\rho_R^{\alpha+1}}}.
\label{eq:overcompressive-right-state-region}
\end{equation}
Thus, for a fixed left state $U_L$, the overcompressive right states lie
strictly below the surface
\begin{equation}
u=u_L-\sqrt{\frac{A(v_L)}{\rho_L^{\alpha+1}}}-\sqrt{\frac{A(v)}{\rho^{\alpha+1}}},
\end{equation}
shown in Figure~\ref{fig:case1}.
\begin{prop}[Existence of the left reduced outer orbit]
\label{prop:left-reduced-orbit}
Assume that $0<\alpha\leq1$ and
\begin{equation}
A(v)=(1-v)^{-p},
\qquad
0<p\leq\alpha,
\label{eq:A-form-left}
\end{equation}
with $v_L<1$ and $v_\delta<1$. Suppose that
\begin{equation}
y_L:=u_L-c>0
\label{eq:yL-positive}
\end{equation}
and
\begin{equation}
B_L
:=
\rho_L(u_L-c)^2
-
\frac{A(v_L)}{\rho_L^\alpha}
>0.
\label{eq:BL-positive}
\end{equation}
The latter inequality follows from the conditions in Lemma~\ref{lem:sign_xiw}. Let $W_L=F(H_L)-cH_L,$ then the reduced system
\begin{equation}
\frac{dH}{d\tau}
=
F(H)-cH-W_L
\label{eq:left-reduced-H}
\end{equation}
has a trajectory
\begin{equation}
H_L(\zeta)
=
\bigl(\rho_L(\zeta),u_L(\zeta),v_L(\zeta)\bigr),
\qquad
-\infty<\zeta<0,
\end{equation}
such that
\begin{equation}
H_L(\zeta)\longrightarrow
(\rho_L,u_L,v_L)
\qquad
\text{as }\zeta\to-\infty.
\label{eq:left-H-limit}
\end{equation}
Moreover,
\begin{equation}
\rho_L(\zeta)\longrightarrow+\infty,
\qquad
u_L(\zeta)\longrightarrow c,
\qquad
v_L(\zeta)\longrightarrow v_\delta
\qquad
\text{as }\zeta\to0^-.
\label{eq:left-inner-limit-original}
\end{equation}
More precisely,
\begin{equation}
u_L(\zeta)-c
=
-B_L\zeta+o(|\zeta|),
\qquad
\frac{1}{\rho_L(\zeta)}
=
\frac{B_L}{2}\zeta^2+o(\zeta^2),
\label{eq:left-inner-asymptotics-rho-u}
\end{equation}
and
\begin{equation}
v_L(\zeta)-v_\delta
=
\Gamma_L\zeta+o(|\zeta|)
\label{eq:left-inner-asymptotics-v}
\end{equation}
as $\zeta\to0^-$, where
\begin{equation}
\Gamma_L
=
v_\delta w_{1,L}-w_{3,L}
=
\rho_L(u_L-c)(v_\delta-v_L).
\label{eq:left-Gamma-identity}
\end{equation}
\end{prop}

\begin{proof}
Writing \eqref{eq:left-reduced-H} in primitive variables gives
\begin{align}\label{eq:left-rho-fast}
\frac{d\rho}{d\tau}
=\rho(u-c)-w_{1,L},\ \ \
\frac{du}{d\tau}=
\frac{uw_{1,L}-w_{2,L}}{\rho}
-\frac{A(v)}{\rho^{\alpha+1}},\ \ \
\frac{dv}{d\tau}
=\frac{vw_{1,L}-w_{3,L}}{\rho}.
\end{align}
Since $W_L=F(H_L)-cH_L$, we have
\begin{equation}
w_{1,L}
=
\rho_L(u_L-c)
=
\rho_Ly_L.
\label{eq:left-w1}
\end{equation}
Furthermore, by the definition of $B_L$, $cw_{1,L}-w_{2,L}=-B_L.$
Consequently,
\begin{align}
uw_{1,L}-w_{2,L}=
(u-c)w_{1,L}
+cw_{1,L}-w_{2,L}=
\rho_Ly_L(u-c)-B_L.
\label{eq:left-momentum-reduction}
\end{align}
Similarly,
\begin{align}
v_\delta w_{1,L}-w_{3,L}=
\rho_L(u_L-c)(v_\delta-v_L)=
\rho_Ly_L(v_\delta-v_L)
=\Gamma_L.
\label{eq:left-Gamma-computation}
\end{align}
It follows that 
\begin{align}\label{eq:left-rho-tau}
\frac{d\rho}{d\tau}
=\rho(u-c)-\rho_Ly_L,
\ \  \ \frac{du}{d\tau}=
\frac{\rho_Ly_L(u-c)-B_L}{\rho}
-\frac{A(v)}{\rho^{\alpha+1}},\ \ \ 
\frac{dv}{d\tau}
=\frac{\rho_Ly_L(v-v_L)}{\rho}.
\end{align}
We now introduce a new independent variable $\zeta$ by $\frac{d\tau}{d\zeta}=\rho.$ Since $\rho>0$, this change of independent variable preserves the orientation of trajectories. Equivalently, $\frac{d}{d\zeta}
=\rho\frac{d}{d\tau}.$ The system \eqref{eq:left-rho-tau} therefore becomes
\begin{align}\label{eq:left-rho-zeta}
\frac{d\rho}{d\zeta}=
\rho^2(u-c)-\rho\rho_Ly_L,\ \ \ 
\frac{du}{d\zeta}
=\rho_Ly_L(u-c)-B_L
-\frac{A(v)}{\rho^\alpha},\ \ \ 
\frac{dv}{d\zeta}
=\rho_Ly_L(v-v_L).
\end{align}
The third equation is linear and the solution satisfying $v(0)=v_\delta$ is
\begin{equation}
v(\zeta)
=
v_L+(v_\delta-v_L)
e^{\rho_Ly_L\zeta}.
\label{eq:left-v-zeta-explicit}
\end{equation}
Thus, $v(\zeta)\longrightarrow v_L
\quad \text{as }\zeta\to-\infty,$ and
$v(\zeta)\longrightarrow v_\delta
\quad \text{as }\zeta\to0^-.$ Introduce $\mu=\frac{1}{\rho},
\quad y=u-c, \quad \mu_L=\frac{1}{\rho_L}.$ We reverse the independent variable by setting $\eta=-\zeta.$ Finally, define
\begin{equation}
z=
e^{\rho_Ly_L\zeta}=
e^{-\frac{y_L}{\mu_L}\eta}.
\label{eq:left-z-definition}
\end{equation}
Then $v(z)
=v_L+(v_\delta-v_L)z,$ and
\begin{equation}
B_L
=
\frac{y_L^2}{\mu_L}
-
A(v_L)\mu_L^\alpha.
\label{eq:left-B-new}
\end{equation}
From \eqref{eq:left-rho-zeta}, we obtain the transformed system is
\begin{equation}
\frac{d\mu}{d\eta}
=
y-\frac{y_L}{\mu_L}\mu,
\ \ \ 
\frac{dy}{d\eta}
=B_L+A(v(z))\mu^\alpha
-
\frac{y_L}{\mu_L}y,\ \ \ 
\frac{dz}{d\eta}
=-\frac{y_L}{\mu_L}z.
\label{eq:left-reversed-system}
\end{equation}
The desired left reduced orbit is therefore equivalent to a solution of
\eqref{eq:left-reversed-system} satisfying
\begin{equation}
(\mu(0),y(0),z(0))=(0,0,1)
\label{eq:left-reversed-initial}
\end{equation}
and
\begin{equation}
(\mu(\eta),y(\eta),z(\eta))
\longrightarrow
(\mu_L,y_L,0)
\qquad
\text{as }\eta\to+\infty.
\label{eq:left-reversed-target}
\end{equation}

The vector field in \eqref{eq:left-reversed-system} is continuous on

\begin{equation}
\mathcal{D}
=
\left\{
(\mu,y,z):
\mu\geq0,\quad 0\leq z\leq1
\right\}.
\end{equation}
For $0<\alpha<1$, the map $\mu\mapsto\mu^\alpha$ is not locally
Lipschitz at $\mu=0$. Thus uniqueness of solutions through
\eqref{eq:left-reversed-initial} does not follow from the standard
Picard--Lindel\"of theorem. Nevertheless, Peano's theorem gives at least one local solution through \eqref{eq:left-reversed-initial}, which is sufficient for the present existence argument. At $\eta=0$, system \eqref{eq:left-reversed-system} gives
\begin{equation}
\frac{d\mu}{d\eta}\bigg|_{\eta=0}=0,
\qquad
\frac{dy}{d\eta}\bigg|_{\eta=0}=B_L>0,
\qquad
\frac{dz}{d\eta}\bigg|_{\eta=0}
=
-\frac{y_L}{\mu_L}<0.
\label{eq:left-initial-derivatives}
\end{equation}
Consequently,
\begin{equation}
y(\eta)=B_L\eta+o(\eta)
\label{eq:left-local-y}
\end{equation}
and
\begin{equation}
z(\eta)
=
1-\frac{y_L}{\mu_L}\eta+o(\eta)
\label{eq:left-local-z}
\end{equation}
as $\eta\to0^+$. Substituting the expansion for $y$ into the first
equation of \eqref{eq:left-reversed-system} yields
\begin{equation}
\mu(\eta)
=
\frac{B_L}{2}\eta^2+o(\eta^2).
\label{eq:left-local-mu}
\end{equation}
In particular,
\begin{align}
y(\eta)-\frac{y_L}{\mu_L}\mu(\eta)
&=
B_L\eta
-
\frac{y_LB_L}{2\mu_L}\eta^2
+
o(\eta^2)>0
\label{eq:left-local-entry}
\end{align}
for all sufficiently small $\eta>0$. We now construct a bounded positively invariant region. The appropriate
region depends on the relative positions of $v_L$ and $v_\delta$.

\medskip

\noindent
\textit{Case 1: $v_\delta>v_L$.}
Define
\begin{equation}
M(z)
=
\mu_L
\left(
\frac{A(v_L)}{A(v(z))}
\right)^{1/\alpha}.
\label{eq:left-M-definition}
\end{equation}
Since $A$ is increasing and
\begin{equation}
v_L\leq v(z)\leq v_\delta,
\qquad
0\leq z\leq1,
\end{equation}
we have
\begin{equation}
0<M(z)\leq\mu_L.
\label{eq:left-M-bounds}
\end{equation}
Moreover,
\begin{equation}
A(v(z))M(z)^\alpha
=
A(v_L)\mu_L^\alpha.
\label{eq:left-M-identity}
\end{equation}
Consider
\begin{equation}
\mathcal{R}_L
=
\left\{
(\mu,y,z):
0\leq z\leq1,\quad
0\leq\mu\leq M(z),\quad
\frac{y_L}{\mu_L}\mu\leq y\leq y_L
\right\}.
\label{eq:left-region-one}
\end{equation}
We verify positive invariance by examining each boundary component. On the lower boundary $y=\frac{y_L}{\mu_L}\mu,$ define $G(\mu,y)
=y-\frac{y_L}{\mu_L}\mu.$ Along $G=0$,
\begin{align}
\frac{dG}{d\eta}
&=
\frac{dy}{d\eta}
-
\frac{y_L}{\mu_L}\frac{d\mu}{d\eta}=
B_L+A(v(z))\mu^\alpha
-
\frac{y_L^2}{\mu_L^2}\mu.
\label{eq:left-G-derivative}
\end{align}
For fixed $z\in[0,1]$, set $\Phi_z(\mu)
=B_L+A(v(z))\mu^\alpha-\frac{y_L^2}{\mu_L^2}\mu.$ At the endpoints of the interval $[0,M(z)]$, $\Phi_z(0)=B_L>0$ and, by \eqref{eq:left-M-identity} and \eqref{eq:left-B-new},
\begin{align}
\Phi_z(M(z))
&=
B_L+A(v_L)\mu_L^\alpha
-
\frac{y_L^2}{\mu_L^2}M(z)=
\frac{y_L^2}{\mu_L^2}
\bigl(\mu_L-M(z)\bigr)
\geq0.
\label{eq:left-Phi-M}
\end{align}
For $0<\alpha<1$,
\begin{equation}
\Phi_z''(\mu)
=
\alpha(\alpha-1)A(v(z))\mu^{\alpha-2}
<0,
\qquad
\mu>0,
\end{equation}
so $\Phi_z$ is strictly concave. Since a concave function lies above the
chord joining any two points of its graph, and since
$$
\Phi_z(0)=B_L>0,
\qquad
\Phi_z(M(z))
=
\frac{y_L^2}{\mu_L^2}
\bigl(\mu_L-M(z)\bigr)
\ge0,
$$
we obtain $\Phi_z(\mu)\ge0,
\quad
0\le\mu\le M(z).$ For $\alpha=1$, the function $\Phi_z$ is affine. Therefore, it coincides
with the chord joining its endpoint values. Since both endpoint values
are nonnegative, we again obtain
\begin{equation}
\Phi_z(\mu)\ge0,
\qquad
0\le\mu\le M(z).
\label{eq:left-Phi-positive}
\end{equation}
Hence $\frac{dG}{d\eta}\geq0$ on the boundary $G=0$. On the upper boundary $y=y_L$, we have
\begin{align}
\frac{dy}{d\eta}=
B_L+A(v(z))\mu^\alpha
-
\frac{y_L^2}{\mu_L}=
A(v(z))\mu^\alpha
-
A(v_L)\mu_L^\alpha\leq0,
\label{eq:left-upper-y-inward}
\end{align}
where the last inequality follows from $\mu\leq M(z)$ and
\eqref{eq:left-M-identity}.
On $\mu=0$,
\begin{equation}
\frac{d\mu}{d\eta}=y\geq0.
\label{eq:left-mu-zero}
\end{equation}
It remains to consider the moving boundary $\mu=M(z)$. Define $J(\mu,z)=M(z)-\mu.$ The region corresponds to $J\geq0$. Along $J=0$,
\begin{align}
\frac{dJ}{d\eta}
&=
M'(z)\frac{dz}{d\eta}
-
\frac{d\mu}{d\eta}=
M'(z)\frac{dz}{d\eta}
-y+\frac{y_L}{\mu_L}M(z).
\label{eq:left-J-derivative}
\end{align}
Since $y\leq y_L$, it is sufficient to prove
\begin{equation}
M'(z)\frac{dz}{d\eta}
\geq
\frac{y_L}{\mu_L}
\bigl(\mu_L-M(z)\bigr).
\label{eq:left-moving-sufficient}
\end{equation}
Using $A(v)=(1-v)^{-p}$, define
\begin{equation}
q:=\frac{p}{\alpha},
\qquad
x(z):=\frac{1-v(z)}{1-v_L}.
\label{eq:left-q-x-definition}
\end{equation}
Then $0<q\leq1$ and $M(z)=\mu_Lx(z)^q.$ Since $v(z)=v_L+(v_\delta-v_L)z,$ we have $x(z)
=1-\frac{v_\delta-v_L}{1-v_L}z$ and therefore

\begin{equation}
1-x(z)
=
\frac{v_\delta-v_L}{1-v_L}z.
\label{eq:left-one-minus-x}
\end{equation}
Using $\frac{dz}{d\eta}
=-\frac{y_L}{\mu_L}z,$ a direct calculation gives
\begin{equation}
M'(z)\frac{dz}{d\eta}
=
y_Lq(1-x)x^{q-1}.
\label{eq:left-M-eta}
\end{equation}
On the other hand,
\begin{equation}
\frac{y_L}{\mu_L}
\bigl(\mu_L-M(z)\bigr)
=
y_L(1-x^q).
\label{eq:left-muL-minus-M}
\end{equation}
Because $s\mapsto s^q$ is concave for $0<q\leq1$, its tangent line at
$s=x$ lies above its graph. Evaluating this tangent-line inequality at
$s=1$ yields $1\leq
x^q+qx^{q-1}(1-x),$ or equivalently,
\begin{equation}
q(1-x)x^{q-1}
\geq
1-x^q.
\label{eq:left-key-moving}
\end{equation}
Equations \eqref{eq:left-M-eta},
\eqref{eq:left-muL-minus-M}, and
\eqref{eq:left-key-moving} imply
\eqref{eq:left-moving-sufficient}. Hence, $\frac{dJ}{d\eta}\geq0$
on $\mu=M(z)$. Finally, on $z=0$ we have $dz/d\eta=0$, whereas on $z=1$,
\begin{equation}
\frac{dz}{d\eta}
=
-\frac{y_L}{\mu_L}<0.
\end{equation}
Thus the interval $0\leq z\leq1$ is positively invariant. In view of
\eqref{eq:left-local-entry}, the local solution enters
$\mathcal{R}_L$ and remains there for as long as it exists.

\medskip

\noindent
\textit{Case 2: $v_\delta\leq v_L$.} In this case,
\begin{equation}
v_\delta\leq v(z)\leq v_L,
\qquad
0\leq z\leq1.
\end{equation}
Consider the rectangular region
\begin{equation}
\widetilde{\mathcal{R}}_L
=
\left\{
(\mu,y,z):
0\leq z\leq1,\quad
0\leq\mu\leq\mu_L,\quad
0\leq y\leq y_L
\right\}.
\label{eq:left-region-two}
\end{equation}
On $\mu=0$, $\frac{d\mu}{d\eta}=y\geq0,$ whereas on $\mu=\mu_L$, $\frac{d\mu}{d\eta}
=y-y_L
\leq0.$ On $y=0$, $\frac{dy}{d\eta}
=
B_L+A(v(z))\mu^\alpha
>0.$ On $y=y_L$,
\begin{align}
\frac{dy}{d\eta}=
B_L+A(v(z))\mu^\alpha
-
\frac{y_L^2}{\mu_L}=
A(v(z))\mu^\alpha
-
A(v_L)\mu_L^\alpha\leq0,
\end{align}
because $A(v(z))\leq A(v_L)$ and $\mu\leq\mu_L$. As in the first case, the interval $0\leq z\leq1$ is positively invariant. Therefore, $\widetilde{\mathcal{R}}_L$ is positively invariant.

In either case, the solution remains in a compact subset of the domain
on which the vector field is continuous. Since the vector field is
locally bounded there, the solution can be continued for every
$\eta\geq0$.

The third equation of \eqref{eq:left-reversed-system} has the explicit solution $z(\eta)
=\exp\left(
-\frac{y_L}{\mu_L}\eta
\right).$ Consequently,
\begin{equation}
z(\eta)\longrightarrow0
\qquad
\text{as }\eta\to+\infty,
\label{eq:left-z-to-zero}
\end{equation}
and hence
\begin{equation}
v(z(\eta))\longrightarrow v_L
\qquad
\text{as }\eta\to+\infty.
\label{eq:left-v-to-vL}
\end{equation}
The limiting planar system is
\begin{equation}
\begin{cases}
\displaystyle
\frac{d\mu}{d\eta}
=
y-\frac{y_L}{\mu_L}\mu,
\\[1.2ex]
\displaystyle
\frac{dy}{d\eta}
=
B_L+A(v_L)\mu^\alpha
-
\frac{y_L}{\mu_L}y.
\end{cases}
\label{eq:left-limiting-planar}
\end{equation}
At an equilibrium, $y=\frac{y_L}{\mu_L}\mu,$ and therefore $\mu$ must satisfy
\begin{equation}
\mathcal{H}(\mu)
:=
B_L+A(v_L)\mu^\alpha
-
\frac{y_L^2}{\mu_L^2}\mu
=
0.
\label{eq:left-H-function}
\end{equation}
By \eqref{eq:left-B-new},
\begin{equation}
\mathcal{H}(0)=B_L>0,
\qquad
\mathcal{H}(\mu_L)=0.
\label{eq:left-H-endpoint-values}
\end{equation}
For $0<\alpha<1$,
\begin{equation}
\mathcal{H}''(\mu)
=
\alpha(\alpha-1)A(v_L)\mu^{\alpha-2}
<0,
\qquad
\mu>0,
\end{equation}
so $\mathcal{H}$ is strictly concave. For $\alpha=1$, $\mathcal{H}$ is affine with $\mathcal{H}'(\mu)
=A(v_L)-\frac{y_L^2}{\mu_L^2}<0$ by \eqref{eq:left-stability-inequality}. Since $\mathcal{H}(0)=B_L>0,
\ \mathcal{H}(\mu_L)=0,$ it follows in either case that
$
\mathcal{H}(\mu)>0,
\
0<\mu<\mu_L.
$
The inequality $B_L>0$ gives
\begin{equation}
\frac{y_L^2}{\mu_L^2}
>A(v_L)\mu_L^{\alpha-1}
>\alpha A(v_L)\mu_L^{\alpha-1},
\label{eq:left-stability-inequality}
\end{equation}
and hence
\begin{equation}
\mathcal{H}'(\mu_L)
=
\alpha A(v_L)\mu_L^{\alpha-1}
-
\frac{y_L^2}{\mu_L^2}
<0.
\label{eq:left-H-prime-muL}
\end{equation}
For $0<\alpha<1$, the derivative $\mathcal{H}'$ is strictly decreasing,
whereas for $\alpha=1$ it is constant. In either case, $\mathcal{H}'(\mu)\le \mathcal{H}'(\mu_L)<0,
\ \mu\ge\mu_L,$ so $\mathcal{H}$ is strictly decreasing on $[\mu_L,\infty)$. Therefore,
$\mu_L$ is the unique positive zero of $\mathcal{H}$. It follows that $(\mu,y)=(\mu_L,y_L)$ is the unique equilibrium of the limiting system in either trapping region. The Jacobian matrix of \eqref{eq:left-limiting-planar} at this
equilibrium is
\begin{equation}
\mathcal{J}_L
=
\begin{pmatrix}
-\frac{y_L}{\mu_L} & 1
\\[1.2ex]
\alpha A(v_L)\mu_L^{\alpha-1}
&
-\frac{y_L}{\mu_L}
\end{pmatrix}.
\label{eq:left-Jacobian}
\end{equation}
Its eigenvalues are
\begin{equation}
\lambda_\pm
=
-\frac{y_L}{\mu_L}
\pm
\sqrt{\alpha A(v_L)\mu_L^{\alpha-1}}.
\label{eq:left-eigenvalues}
\end{equation}
By \eqref{eq:left-stability-inequality}, $\lambda_-<\lambda_+<0.$
Thus $(\mu_L,y_L)$ is a hyperbolic sink. The divergence of the limiting vector field is
\begin{equation}
-\frac{2y_L}{\mu_L}<0.
\label{eq:left-negative-divergence}
\end{equation}
Hence the Bendixson--Dulac criterion excludes nonconstant periodic
orbits in the trapping region. Since $z(\eta)\to0$, system \eqref{eq:left-reversed-system} is
asymptotically autonomous with limiting system
\eqref{eq:left-limiting-planar}. The bounded trajectory therefore has a nonempty, compact, connected omega-limit set that is invariant under the limiting planar flow. If this omega-limit set contained no equilibrium, the Poincar\'e--Bendixson theorem would imply that it is a periodic orbit, contradicting \eqref{eq:left-negative-divergence}. It must
therefore contain an equilibrium of the limiting system. Since
$(\mu_L,y_L)$ is the only such equilibrium in the trapping region, the omega-limit set contains $(\mu_L,y_L)$.

Because $(\mu_L,y_L)$ is asymptotically stable, there exists a
positively invariant neighborhood contained in its basin of attraction.
The trajectory enters this neighborhood along a sequence of times
tending to infinity and thereafter remains in it. Consequently,
\begin{equation}
(\mu(\eta),y(\eta))
\longrightarrow
(\mu_L,y_L)
\qquad
\text{as }\eta\to+\infty.
\label{eq:left-mu-y-convergence}
\end{equation}
Together with \eqref{eq:left-z-to-zero}, this proves
\begin{equation}
(\mu(\eta),y(\eta),z(\eta))
\longrightarrow
(\mu_L,y_L,0)
\qquad
\text{as }\eta\to+\infty.
\label{eq:left-full-convergence}
\end{equation}
Since $\eta=-\zeta$, the limit $\eta\to+\infty$ corresponds to
$\zeta\to-\infty$. Thus
\begin{equation}
\rho(\zeta)\longrightarrow\rho_L,
\qquad
u(\zeta)\longrightarrow u_L,
\qquad
v(\zeta)\longrightarrow v_L
\qquad
\text{as }\zeta\to-\infty.
\end{equation}
This proves \eqref{eq:left-H-limit}. Finally, $\zeta\to0^-$ corresponds to $\eta\to0^+$. From \eqref{eq:left-local-y}, \eqref{eq:left-local-mu}, and $\eta=-\zeta$,
we obtain
\begin{equation}
u(\zeta)-c
=-B_L\zeta+o(|\zeta|) \ \  \text{and} \ \  
\frac{1}{\rho(\zeta)}
=
\frac{B_L}{2}\zeta^2+o(\zeta^2).
\end{equation}
Moreover, by \eqref{eq:left-v-zeta-explicit},
\begin{align}
v(\zeta)-v_\delta=
(v_\delta-v_L)
\left(
e^{\rho_Ly_L\zeta}-1
\right)=
\rho_Ly_L(v_\delta-v_L)\zeta
+
o(|\zeta|)
=
\Gamma_L\zeta+o(|\zeta|).
\label{eq:left-v-asymptotic-proof}
\end{align}
Therefore,
$
\rho(\zeta)\longrightarrow+\infty,
\
u(\zeta)\longrightarrow c,
\
v(\zeta)\longrightarrow v_\delta
\
\text{as }\zeta\to0^-,
$
and the asymptotic relations
\eqref{eq:left-inner-asymptotics-rho-u} and
\eqref{eq:left-inner-asymptotics-v} follow. This completes the proof.
\end{proof}
\begin{figure}[H]
    \centering
    \begin{minipage}[b]{0.48\textwidth}
        \centering
        \includegraphics[width=\linewidth]{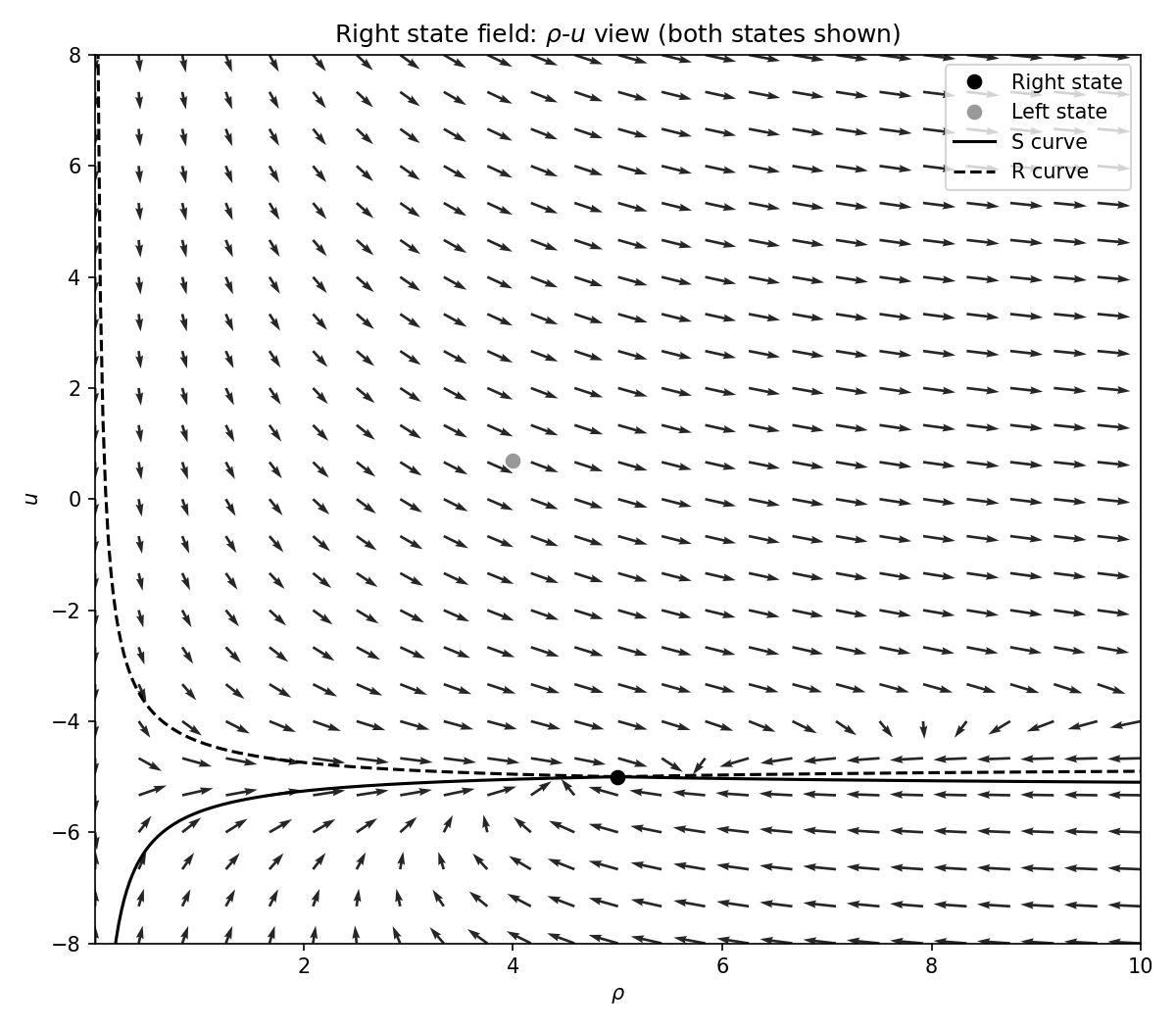}
        \caption{Right state invariant region of Region 6, Case 1: $(u_R,  \rho_R, v_R) = (-5, 5, -5)$.}
        \label{fig:combinedright}
    \end{minipage}
    \hfill
    \begin{minipage}[b]{0.48\textwidth}
        \centering
        \includegraphics[width=\linewidth]{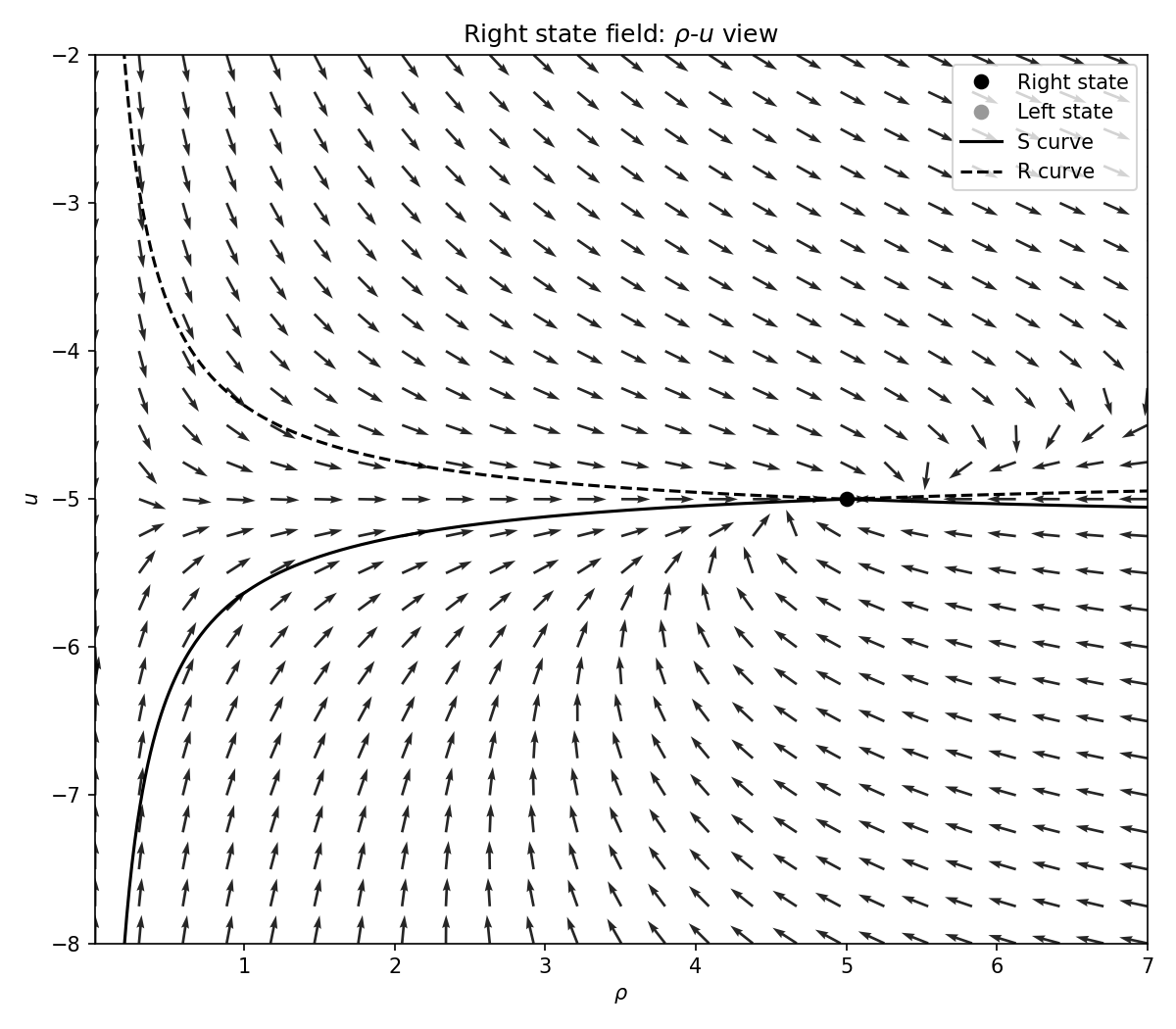}
        \caption{Zoom view of right state invariant region of Case 1.}
        \label{fig:zoomright}
    \end{minipage}
\end{figure}
\begin{prop}[Existence of the right reduced outer orbit]
\label{prop:right-reduced-orbit}
Assume that $0<\alpha\leq1$ and
\begin{equation}
A(v)=(1-v)^{-p},
\qquad
0<p\leq\alpha,
\label{eq:A-form-right}
\end{equation}
with $v_R<1$ and $v_\delta<1$. Suppose that $y_R:=c-u_R>0$ and
\begin{equation}
B_R
:=
\rho_R(c-u_R)^2
-
\frac{A(v_R)}{\rho_R^\alpha}
>0.
\label{eq:BR-positive}
\end{equation}
The latter inequality follows from condition in Lemma~\ref{lem:sign_xiw}. Let $
W_R=F(H_R)-cH_R,$ then the reduced system
\begin{equation}
\frac{dH}{d\tau}
=
F(H)-cH-W_R
\label{eq:right-reduced-H}
\end{equation}
has a trajectory
\begin{equation}
H_R(\zeta)
=
\bigl(\rho_R(\zeta),u_R(\zeta),v_R(\zeta)\bigr),
\qquad
0<\zeta<+\infty,
\end{equation}
such that
\begin{equation}
H_R(\zeta)\longrightarrow
(\rho_R,u_R,v_R)
\qquad
\text{as }\zeta\to+\infty.
\label{eq:right-H-limit}
\end{equation}
Moreover,
\begin{equation}
\rho_R(\zeta)\longrightarrow+\infty,
\qquad
u_R(\zeta)\longrightarrow c,
\qquad
v_R(\zeta)\longrightarrow v_\delta
\qquad
\text{as }\zeta\to0^+.
\label{eq:right-inner-limit-original}
\end{equation}
More precisely,
\begin{equation}
c-u_R(\zeta)
=
B_R\zeta+o(\zeta),
\qquad
\frac{1}{\rho_R(\zeta)}
=
\frac{B_R}{2}\zeta^2+o(\zeta^2),
\label{eq:right-inner-asymptotics-rho-u}
\end{equation}
and
\begin{equation}
v_R(\zeta)-v_\delta
=
\Gamma_R\zeta+o(\zeta)
\label{eq:right-inner-asymptotics-v}
\end{equation}
as $\zeta\to0^+$, where
\begin{equation}
\Gamma_R
=
v_\delta w_{1,R}-w_{3,R}
=
\rho_R(u_R-c)(v_\delta-v_R).
\label{eq:right-Gamma-identity}
\end{equation}

\end{prop}
\begin{proof}

Writing \eqref{eq:right-reduced-H} in primitive variables gives
\begin{align}\label{eq:right-rho-fast}
\displaystyle
\frac{d\rho}{d\tau}
=
\rho(u-c)-w_{1,R},
\ \ \ 
\frac{du}{d\tau}
=
\frac{uw_{1,R}-w_{2,R}}{\rho}
-\frac{A(v)}{\rho^{\alpha+1}},
\ \ \ 
\frac{dv}{d\tau}=
\frac{vw_{1,R}-w_{3,R}}{\rho}.
\end{align}
Since $W_R=F(H_R)-cH_R,$ we have
\begin{equation}
w_{1,R}
=
\rho_R(u_R-c)
=
-\rho_Ry_R.
\label{eq:right-w1}
\end{equation}
Moreover, by the definition of $B_R$,
\begin{equation}
cw_{1,R}-w_{2,R}=-B_R.
\label{eq:right-D-minus-B}
\end{equation}
Consequently,
\begin{align}
uw_{1,R}-w_{2,R}=
(u-c)w_{1,R}
+cw_{1,R}-w_{2,R}
=-\rho_Ry_R(u-c)-B_R=
\rho_Ry_R(c-u)-B_R.
\label{eq:right-momentum-reduction}
\end{align}
Similarly,
\begin{align}
v_\delta w_{1,R}-w_{3,R}
&=
\rho_R(u_R-c)(v_\delta-v_R)=
-\rho_Ry_R(v_\delta-v_R)=
\Gamma_R.
\label{eq:right-Gamma-computation}
\end{align}
Using
\begin{equation}
vw_{1,R}-w_{3,R}
=
(v-v_\delta)w_{1,R}
+
v_\delta w_{1,R}-w_{3,R},
\end{equation}
together with \eqref{eq:right-Gamma-computation}, we obtain
\begin{align}
vw_{1,R}-w_{3,R}=
-\rho_Ry_R(v-v_\delta)+\Gamma_R
=-\rho_Ry_R(v-v_R).
\label{eq:right-transport-reduction}
\end{align}
Therefore \eqref{eq:right-rho-fast} becomes
\begin{align}\label{eq:right-rho-tau}
\frac{d\rho}{d\tau}=
\rho(u-c)+\rho_Ry_R,
\ \ \
\frac{du}{d\tau}=
\frac{\rho_Ry_R(c-u)-B_R}{\rho}
-\frac{A(v)}{\rho^{\alpha+1}},\ \ \ 
\frac{dv}{d\tau}=-\frac{\rho_Ry_R(v-v_R)}{\rho}.
\end{align}
Introduce a new independent variable $\zeta$ by
\begin{equation}
\frac{d\tau}{d\zeta}=\rho.
\label{eq:right-time-rescaling}
\end{equation}
Since $\rho>0$, this change preserves the orientation of trajectories.
Equivalently,
\begin{equation}
\frac{d}{d\zeta}
=
\rho\frac{d}{d\tau}.
\end{equation}
Equations \eqref{eq:right-rho-tau} therefore become
\begin{align}\label{eq:right-rho-zeta}
\frac{d\rho}{d\zeta}
=\rho^2(u-c)+\rho\rho_Ry_R,
\ \ \
\frac{du}{d\zeta}
=\rho_Ry_R(c-u)-B_R
-\frac{A(v)}{\rho^\alpha},
\ \ \
\frac{dv}{d\zeta}=
-\rho_Ry_R(v-v_R).
\end{align}
The third equation is linear. The solution satisfying $v(0)=v_\delta$
is
\begin{equation}
v(\zeta)
=
v_R+(v_\delta-v_R)e^{-\rho_Ry_R\zeta}.
\label{eq:right-v-zeta-explicit}
\end{equation}
In particular, $v(\zeta)\longrightarrow v_\delta
\qquad
\text{as }\zeta\to0^+,$ and $v(\zeta)\longrightarrow v_R
\qquad
\text{as }\zeta\to+\infty.$ We now introduce
\begin{equation}
\mu=\frac{1}{\rho},
\qquad
y=c-u,
\qquad
\mu_R=\frac{1}{\rho_R},
\label{eq:right-variable-change}
\end{equation}
and define
\begin{equation}
z
=
e^{-\rho_Ry_R\zeta}
=
e^{-\frac{y_R}{\mu_R}\zeta}.
\label{eq:right-z-definition}
\end{equation}
Then
\begin{equation}
v(z)
=
v_R+(v_\delta-v_R)z
\label{eq:right-v-of-z}
\end{equation}
and
\begin{equation}
B_R
=
\frac{y_R^2}{\mu_R}
-
A(v_R)\mu_R^\alpha.
\label{eq:right-B-new}
\end{equation}
From \eqref{eq:right-rho-zeta}, we find
\begin{align}
\frac{d\mu}{d\zeta}
&=
-\frac{1}{\rho^2}\frac{d\rho}{d\zeta}=
c-u-\rho_Ry_R\mu=
y-\frac{y_R}{\mu_R}\mu.
\label{eq:right-mu-zeta}
\end{align}
Since $y=c-u$, equation \eqref{eq:right-rho-zeta} gives
\begin{align}
\frac{dy}{d\zeta}
&=
-\frac{du}{d\zeta}=
B_R+A(v(z))\mu^\alpha
-
\frac{y_R}{\mu_R}y.
\label{eq:right-y-zeta}
\end{align}
Finally, $\frac{dz}{d\zeta}
=
-\frac{y_R}{\mu_R}z.$
Thus the transformed right system is
\begin{equation}
\frac{d\mu}{d\zeta}
=y-\frac{y_R}{\mu_R}\mu,
\ \ \ 
\displaystyle
\frac{dy}{d\zeta}
=B_R+A(v(z))\mu^\alpha
-\frac{y_R}{\mu_R}y,
\ \ \
\displaystyle
\frac{dz}{d\zeta}
=-\frac{y_R}{\mu_R}z.
\label{eq:right-transformed-system}
\end{equation}
The singular endpoint corresponds to
\begin{equation}
(\mu(0),y(0),z(0))=(0,0,1),
\label{eq:right-transformed-initial}
\end{equation}
whereas the prescribed right state corresponds to
\begin{equation}
(\mu(\zeta),y(\zeta),z(\zeta))
\longrightarrow
(\mu_R,y_R,0)
\qquad
\text{as }\zeta\to+\infty.
\label{eq:right-transformed-target}
\end{equation}
System \eqref{eq:right-transformed-system} is identical in form to the left transformed system \eqref{eq:left-reversed-system} under the substitution
\begin{equation}
(\mu_L,y_L,B_L,v_L,\eta)
\longmapsto
(\mu_R,y_R,B_R,v_R,\zeta).
\label{eq:left-right-substitution}
\end{equation}
Indeed, under \eqref{eq:left-right-substitution}, the initial condition
\eqref{eq:left-reversed-initial} becomes
\eqref{eq:right-transformed-initial}, the target equilibrium
\eqref{eq:left-reversed-target} becomes
\eqref{eq:right-transformed-target}, and the function $v_L+(v_\delta-v_L)z$ appearing in the left system becomes precisely $v_R+(v_\delta-v_R)z.$ The local existence argument, the construction of the appropriate positively invariant region, the global continuation argument, and the asymptotically autonomous convergence argument in the proof of Proposition~\ref{prop:left-reduced-orbit} therefore apply verbatim.

More precisely, when $v_\delta>v_R$, the substitution
\eqref{eq:left-right-substitution} transforms the corresponding left
trapping region into
\begin{equation}
\mathcal{R}_R
=
\left\{
(\mu,y,z):
0\leq z\leq1,\quad
0\leq\mu\leq M_R(z),\quad
\frac{y_R}{\mu_R}\mu\leq y\leq y_R
\right\},
\label{eq:right-region-one}
\end{equation}
where $M_R(z)
=\mu_R
\left(\frac{A(v_R)}{A(v(z))}
\right)^{1/\alpha}.$ When $v_\delta\leq v_R$, it transforms the corresponding rectangular
region into
\begin{equation}
\widetilde{\mathcal{R}}_R
=
\left\{
(\mu,y,z):
0\leq z\leq1,\quad
0\leq\mu\leq\mu_R,\quad
0\leq y\leq y_R
\right\}.
\label{eq:right-region-two}
\end{equation}
The inward-pointing calculations established in the proof of
Proposition~\ref{prop:left-reduced-orbit} remain unchanged after
\eqref{eq:left-right-substitution}. Hence at least one local solution through \eqref{eq:right-transformed-initial} enters the appropriate
trapping region and remains there for all $\zeta\geq0$. Because the
solution remains in a compact subset of the domain of the vector field, it extends globally for $\zeta\geq0$. Furthermore,
\begin{equation}
z(\zeta)
=
\exp\left(
-\frac{y_R}{\mu_R}\zeta
\right)
\longrightarrow0
\qquad
\text{as }\zeta\to+\infty.
\end{equation}
The limiting planar system is
\begin{equation}
\frac{d\mu}{d\zeta}
=y-\frac{y_R}{\mu_R}\mu,
\ \ \ 
\frac{dy}{d\zeta}
=B_R+A(v_R)\mu^\alpha
-\frac{y_R}{\mu_R}y.
\label{eq:right-limiting-planar}
\end{equation}
Under \eqref{eq:left-right-substitution}, this is the same limiting
planar system considered in the proof of
Proposition~\ref{prop:left-reduced-orbit}. That proof shows that
$(\mu_R,y_R)$ is its unique equilibrium in the relevant trapping
region, that no periodic orbit can occur there, and that the
omega-limit set of the asymptotically autonomous system consists of
this equilibrium alone. Consequently,
\begin{equation}
(\mu(\zeta),y(\zeta),z(\zeta))
\longrightarrow
(\mu_R,y_R,0)
\qquad
\text{as }\zeta\to+\infty.
\end{equation}
Returning to the original variables gives
\begin{equation}
\rho(\zeta)\longrightarrow\rho_R,
\qquad
u(\zeta)\longrightarrow u_R,
\qquad
v(\zeta)\longrightarrow v_R
\qquad
\text{as }\zeta\to+\infty.
\end{equation}
This proves \eqref{eq:right-H-limit}. It remains to establish the asymptotic behavior at the singular
endpoint. Evaluating \eqref{eq:right-transformed-system} at
$(\mu,y,z)=(0,0,1)$ gives
\begin{equation}
\frac{d\mu}{d\zeta}\bigg|_{\zeta=0}=0,
\qquad
\frac{dy}{d\zeta}\bigg|_{\zeta=0}=B_R,
\qquad
\frac{dz}{d\zeta}\bigg|_{\zeta=0}
=
-\frac{y_R}{\mu_R}.
\end{equation}
Therefore, $y(\zeta)=B_R\zeta+o(\zeta)$ and $z(\zeta)
=1-\frac{y_R}{\mu_R}\zeta+o(\zeta)$ as $\zeta\to0^+$. Since
\begin{equation}
\frac{d\mu}{d\zeta}
=
y-\frac{y_R}{\mu_R}\mu
\end{equation}
and $\mu(0)=0$, integration of the preceding expansion for $y$ yields
\begin{equation}
\mu(\zeta)
=
\frac{B_R}{2}\zeta^2+o(\zeta^2).
\label{eq:right-local-mu}
\end{equation}
Recalling that $y=c-u$ and $\mu=1/\rho$, we obtain
\begin{equation}
c-u(\zeta)
=
B_R\zeta+o(\zeta),
\qquad
\frac{1}{\rho(\zeta)}
=
\frac{B_R}{2}\zeta^2+o(\zeta^2).
\end{equation}
Finally, from \eqref{eq:right-v-zeta-explicit},
\begin{align}
v(\zeta)-v_\delta
&=
(v_\delta-v_R)
\left(
e^{-\rho_Ry_R\zeta}-1
\right)=
-\rho_Ry_R(v_\delta-v_R)\zeta
+
o(\zeta)=
\rho_R(u_R-c)(v_\delta-v_R)\zeta
+
o(\zeta)
=
\Gamma_R\zeta+o(\zeta).
\end{align}
Thus \eqref{eq:right-inner-limit-original},
\eqref{eq:right-inner-asymptotics-rho-u}, and
\eqref{eq:right-inner-asymptotics-v} hold. This completes the proof.
\end{proof}

\subsection{The singular reduced configuration}
\label{subsec:singular-reduced-configuration}

Propositions~\ref{prop:left-reduced-orbit} and
\ref{prop:right-reduced-orbit} provide reduced trajectories connecting the
constant left and right states to the singular limiting state
\begin{equation}
\rho\longrightarrow+\infty,
\qquad
u\longrightarrow c,
\qquad
v\longrightarrow v_\delta.
\end{equation}
We now construct the middle component of the singular configuration. This
component describes the passage through the region in which the physical
density becomes unbounded and the variable $W$ undergoes an order-one change from $W_L$ to $W_R$, while the combinations $D$ and $\Gamma$ remain invariant by Lemma~\ref{lem:outer-W-RH}. The construction proceeds in five steps. First, we introduce a distinguished scaling near the singular region. Second, we derive the reduced middle system obtained in the singular limit. Third, we solve this reduced system explicitly. Fourth, we define entry and exit sections away
from the singular boundary and compute the corresponding transition.
Finally, we reconstruct the change in $W$ and verify that it has the
direction prescribed by the generalized Rankine--Hugoniot relations.

The analysis in this subsection concerns the reduced problem. The
persistence of the resulting singular configuration for positive values
of $\varepsilon$ will be addressed separately.

\subsubsection{Distinguished inner scaling}
\label{subsubsec:distinguished-inner-scaling}
Recall that $\rho_0=\varepsilon\rho$ is the rescaled density introduced before the spherical blow-up. The
spherical blow-up
\begin{equation}
\rho_0=r\bar\rho,
\qquad
u-\xi=r\bar u,
\qquad
v-v_\delta=r\bar v,
\qquad
\sqrt{\varepsilon}=r\bar\varepsilon
\label{eq:spherical-blowup-recalled}
\end{equation}
resolves the common singular set
\begin{equation}
\rho_0=0,
\qquad
u-\xi=0,
\qquad
v-v_\delta=0,
\qquad
\varepsilon=0.
\end{equation}

The scaling introduced below does not replace
\eqref{eq:spherical-blowup-recalled}. It is an additional inner
rescaling used to describe the middle passage between the entry and exit
points on the blown-up boundary.

The outer trajectories satisfy $\frac{1}{\rho}=O(\zeta^2),
\quad u-c=O(\zeta),
\quad v-v_\delta=O(\zeta)$ near the singular endpoint. This quadratic-linear relation motivates
the use of a squared variable for the rescaled density. We introduce
\begin{equation}
\varepsilon=\delta^3,
\qquad
\rho_0=\delta z^2,
\qquad
u-\xi=\delta\frac{y}{z},
\qquad
v-v_\delta=\delta\frac{h}{z},
\label{eq:distinguished-scaling}
\end{equation}
together with the stretched independent variable $s=\delta\tau$, where
$\delta=\varepsilon^{1/3}$ is constant along each trajectory. The choice
of powers in \eqref{eq:distinguished-scaling} is justified by the
transformed equations derived below. In particular, it yields $\frac{dk}{ds}=-z^2,$ where $k=w_1$, so that $k$ undergoes an order-one change over an $O(1)$ interval in the stretched variable $s$.

The relation with the spherical blow-up follows directly from
\eqref{eq:distinguished-scaling}. When $z$, $y/z$, and $h/z$ remain
bounded,
\begin{equation}
\rho_0=O(\delta),
\qquad
u-\xi=O(\delta),
\qquad
v-v_\delta=O(\delta),
\qquad
\sqrt{\varepsilon}=O(\delta^{3/2}).
\end{equation}
Consequently, the radial variable in
\eqref{eq:spherical-blowup-recalled} satisfies $r=O(\delta),$
whereas $\bar\varepsilon
=\frac{\sqrt{\varepsilon}}{r}
=O(\delta^{1/2})
\longrightarrow0.$
Thus the reduced middle orbit lies on the part of the blown-up boundary
given by $r=0,\ \bar\varepsilon=0.$

The positive-$\varepsilon$ passage approaches this boundary orbit as
$\delta\to0^+$. The directional coordinates introduced later are needed
to analyze the entry and exit endpoints, where $z\to0$ and the inner
coordinates in \eqref{eq:distinguished-scaling} become singular.

Since $\varepsilon=\delta^3$, the singular limit
$\varepsilon\to0^+$ is equivalent to $\delta\to0^+$. Define
\begin{equation}
k=w_1,
\qquad
D=\xi w_1-w_2,
\qquad
\Gamma=v_\delta w_1-w_3.
\label{eq:KDGamma-def}
\end{equation}
Next, introduce the stretched variable $s$ by
\begin{equation}
\frac{d}{d\tau}
=
\delta\frac{d}{ds},
\label{eq:tau-s-derivative}
\end{equation}
which desingularizes the vector field. A prime will denote
differentiation with respect to $s$. On the region $z>0$, where the change of variables
(\ref{eq:distinguished-scaling}) is regular, direct substitution of into \eqref{GSPT_eps_sys}, followed by the above time rescaling,
yields the desingularized system
\begin{equation}
\begin{cases}\displaystyle
z'
&=
\frac{y}{2}
-
\frac{\delta k}{2z},
\\[0.8ex]\displaystyle
y'
&=
\frac{D+\frac12y^2}{z}
+
\frac{\delta ky}{2z^2}
-
\delta z
-
\delta^{2\alpha}
\frac{
A\!\left(v_\delta+\delta h/z\right)
}{
z^{2\alpha+1}
},
\\[0.8ex]\displaystyle
h'
&=
\frac{\Gamma+\frac12hy}{z}
+
\frac{\delta kh}{2z^2},
\\[0.8ex]\displaystyle
k'
&=
-z^2,
\\\displaystyle
D'
&=
\delta zy+\delta^2k,
\\[0.8ex]\displaystyle
\Gamma'
&=
\delta zh,
\\[0.8ex]\displaystyle
\xi'
&=
\delta^2,
\\[0.8ex]\displaystyle
\delta'
&=
0.
\end{cases}
\label{eq:full-middle-system}
\end{equation}
For each fixed compact subset of the region $z>0$, the right-hand side of \eqref{eq:full-middle-system} has a well-defined limit as
$\delta\to0^+$. The resulting reduced system describes the interior of the middle passage. The limiting behavior as $z\to0$ is not obtained by setting $\delta=0$ uniformly in \eqref{eq:full-middle-system}; it will instead be recovered from the explicit reduced orbit and subsequently analyzed in directional blow-up coordinates.

\subsubsection{The reduced middle connection}
\label{subsubsec:reduced-middle-connection}
Setting $\delta=0$ in \eqref{eq:full-middle-system}, on the interior
region $z>0$, gives
\begin{equation}
z'=\frac{y}{2},
\ \ \ 
y'=\frac{D+\frac12y^2}{z},
\ \ \ 
h'=\frac{\Gamma+\frac12hy}{z},
\ \ \ 
k'=-z^2,
\ \ \ 
D'=0,
\ \ \ 
\Gamma'=0,
\ \ \
\xi'=0.
\label{eq:reduced-middle-system}
\end{equation}
Thus $D$, $\Gamma$, and $\xi$ are constant along every reduced middle trajectory. By Lemma~\ref{lem:sign_xiw}, $D<0.$ Define
\begin{equation}
a:=\sqrt{-2D}>0.
\label{eq:middle-a-definition}
\end{equation}
We now construct the reduced orbit joining the incoming and outgoing
directions selected by the outer trajectories.
\begin{prop}[Singular middle connection]
\label{prop:singular-middle-connection}

Assume that $D<0$, and let $a$ be defined by
\eqref{eq:middle-a-definition}. For every $\lambda>0$, the reduced
middle system \eqref{eq:reduced-middle-system} has, up to
translation in $s$, a unique trajectory satisfying
\begin{equation}
y^2+\lambda^2z^2=a^2
\label{eq:middle-ellipse}
\end{equation}
and
\begin{equation}
Dh-\Gamma y=0.
\label{eq:middle-invariant-relation}
\end{equation}
This trajectory joins, in its $(z,y,h)$-components, the points
\begin{equation}
P_{\mathrm{in}}^0
=
\left(
0,\,
a,\,
\frac{\Gamma}{D}a
\right)
\label{eq:middle-entry-zero}
\end{equation}
and
\begin{equation}
P_{\mathrm{out}}^0
=
\left(
0,\,
-a,\,
-\frac{\Gamma}{D}a
\right).
\label{eq:middle-exit-zero}
\end{equation}
More precisely, after fixing $s_-\in\mathbb{R}$ and setting $s_+:=s_-+\frac{2\pi}{\lambda},$ the trajectory is given for $s_-\leq s\leq s_+$ by
\begin{align}
z(s)
&=
\frac{a}{\lambda}
\sin\left(
\frac{\lambda}{2}(s-s_-)
\right),
\label{eq:middle-z-explicit}
\\
y(s)
&=
a
\cos\left(
\frac{\lambda}{2}(s-s_-)
\right),
\label{eq:middle-y-explicit}
\\
h(s)
&=
\frac{\Gamma}{D}a
\cos\left(
\frac{\lambda}{2}(s-s_-)
\right).
\label{eq:middle-h-explicit}
\end{align}
For every
\begin{equation}
0<\eta<\frac{a}{\lambda},
\label{eq:middle-eta-range}
\end{equation}
define the transverse sections $\Sigma_{\mathrm{in}}^\eta
=\{z=\eta,\ y>0\}$ and $\Sigma_{\mathrm{out}}^\eta=\{z=\eta,\ y<0\}.$
If
\begin{equation}
\theta_\eta
:=
\arcsin\left(
\frac{\lambda\eta}{a}
\right)
\in
\left(
0,\frac{\pi}{2}
\right),
\label{eq:middle-theta-eta}
\end{equation}
then the corresponding entry and exit points have $(z,y,h)$-components
\begin{equation}
P_{\mathrm{in}}^\eta
=
\left(
\eta,\,
\sqrt{a^2-\lambda^2\eta^2},\,
\frac{\Gamma}{D}
\sqrt{a^2-\lambda^2\eta^2}
\right)
\label{eq:middle-entry-eta}
\end{equation}
and
\begin{equation}
P_{\mathrm{out}}^\eta
=
\left(
\eta,\,
-\sqrt{a^2-\lambda^2\eta^2},\,
-\frac{\Gamma}{D}
\sqrt{a^2-\lambda^2\eta^2}
\right).
\label{eq:middle-exit-eta}
\end{equation}
The change in $k$ along the truncated orbit from
$P_{\mathrm{in}}^\eta$ to $P_{\mathrm{out}}^\eta$ is
\begin{equation}
k_{\mathrm{in}}^\eta-k_{\mathrm{out}}^\eta
=
\frac{-2D}{\lambda^3}
\left(
\pi-2\theta_\eta+\sin(2\theta_\eta)
\right).
\label{eq:middle-truncated-k-jump}
\end{equation}
Letting $\eta\to0^+$ gives
\begin{equation}
k_L-k_R
=
\frac{-2\pi D}{\lambda^3}.
\label{eq:middle-complete-k-jump}
\end{equation}
Consequently, whenever $k_L-k_R>0$, the unique value of $\lambda$
producing the prescribed change in $k$ is
\begin{equation}
\lambda_*
=
\left(
\frac{-2\pi D}{k_L-k_R}
\right)^{1/3}.
\label{eq:middle-lambda-selection}
\end{equation}
Moreover, this matching condition is nondegenerate:
\begin{equation}
\left.
\frac{d}{d\lambda}
\left(
\frac{-2\pi D}{\lambda^3}
\right)
\right|_{\lambda=\lambda_*}
=
\frac{6\pi D}{\lambda_*^4}
\neq0.
\label{eq:middle-lambda-nondegeneracy}
\end{equation}
Finally, the reconstructed variable $W$ satisfies
\begin{equation}
W_R-W_L
=
(k_R-k_L)
\begin{pmatrix}
1\\
c\\
v_\delta
\end{pmatrix}.
\label{eq:middle-W-jump}
\end{equation}

\end{prop}

\begin{proof}
Since $D$ is constant, the first two equations of
\eqref{eq:reduced-middle-system} imply
\begin{align}
\frac{d}{ds}\left(2D+y^2\right)=
2yy'=
\frac{2y}{z}
\left(
D+\frac12y^2
\right)=
\frac{y}{z}
\left(
2D+y^2
\right),
\label{eq:middle-first-integral-numerator}
\end{align}
whereas
\begin{equation}
\frac{d}{ds}(z^2)=zy.
\label{eq:middle-z-square-derivative}
\end{equation}
It follows that
\begin{equation}
\frac{d}{ds}
\left(
\frac{2D+y^2}{z^2}
\right)
=
0
\label{eq:middle-first-integral-ratio}
\end{equation}
on every interval on which $z>0$. Hence
\begin{equation}
2D+y^2=Cz^2
\label{eq:middle-first-integral}
\end{equation}
for some constant $C$. The desired orbit passes from $y=a$ to $y=-a$ and therefore contains a
point at which $y=0$. At such a point, $z>0$, and
\eqref{eq:middle-first-integral} gives $Cz^2=2D<0.$ Thus $C<0$. We may therefore write $C=-\lambda^2,
\ \lambda>0.$ Since $a^2=-2D$, equation \eqref{eq:middle-first-integral} becomes $y^2+\lambda^2z^2=a^2.$ Introduce $\theta
=\frac{\lambda}{2}(s-s_-).$ The upper half of the ellipse $z\geq0$ is parametrized by
\begin{equation}
z=\frac{a}{\lambda}\sin\theta,
\qquad
y=a\cos\theta,
\qquad
0\leq\theta\leq\pi.
\end{equation}
Because $d\theta/ds=\lambda/2$, we have $z'
=\frac{a}{2}\cos\theta
=\frac{y}{2}.$ Furthermore, using $D=-a^2/2$,
\begin{align}
\frac{D+\frac12y^2}{z}
&=
\frac{
-\frac12a^2+\frac12a^2\cos^2\theta
}{
(a/\lambda)\sin\theta
}=
-\frac{a\lambda}{2}\sin\theta=
y'.
\end{align}
Thus \eqref{eq:middle-z-explicit} and
\eqref{eq:middle-y-explicit} solve the first two equations of
\eqref{eq:reduced-middle-system}. They satisfy
\begin{equation}
(z,y)(s_-)=(0,a),
\qquad
(z,y)(s_+)=(0,-a).
\end{equation}
Next, define
\begin{equation}
Q:=Dh-\Gamma y.
\end{equation}
Since $D$ and $\Gamma$ are constant, the reduced equations give
\begin{align}
Q'
&=
Dh'-\Gamma y'=
D\frac{\Gamma+\frac12hy}{z}
-
\Gamma\frac{D+\frac12y^2}{z}
=
\frac{y}{2z}
\left(
Dh-\Gamma y
\right)=
\frac{y}{2z}Q.
\label{eq:middle-Q-equation}
\end{align}
Therefore the hypersurface $Q=0$, equivalently
\begin{equation}
Dh-\Gamma y=0,
\end{equation}
is invariant. Since $D\neq0$, it follows that $h=\frac{\Gamma}{D}y,$
which gives \eqref{eq:middle-h-explicit} and the endpoint values
\eqref{eq:middle-entry-zero}--\eqref{eq:middle-exit-zero}.

For a fixed $\eta$ satisfying \eqref{eq:middle-eta-range}, the entry
intersection occurs at $\theta=\theta_\eta$, whereas the exit
intersection occurs at $\theta=\pi-\theta_\eta$. Substitution into the explicit solution gives \eqref{eq:middle-entry-eta} and
\eqref{eq:middle-exit-eta}. Since $z'
=\frac{y}{2}$ is nonzero at both intersections, the sections
$\Sigma_{\mathrm{in}}^\eta$ and $\Sigma_{\mathrm{out}}^\eta$ are
transverse to the reduced flow. Along the middle orbit, $k'=-z^2.$ Using
\begin{equation}
z=\frac{a}{\lambda}\sin\theta,
\qquad
ds=\frac{2}{\lambda}\,d\theta,
\end{equation}
we obtain
\begin{align}
k_{\mathrm{in}}^\eta-k_{\mathrm{out}}^\eta
&=
\int_{s_{\mathrm{in}}}^{s_{\mathrm{out}}}
z(s)^2\,ds
=
\frac{2a^2}{\lambda^3}
\int_{\theta_\eta}^{\pi-\theta_\eta}
\sin^2\theta\,d\theta
=
\frac{a^2}{\lambda^3}
\left(
\pi-2\theta_\eta+\sin(2\theta_\eta)
\right).
\end{align}
Since $a^2=-2D$, this proves
\eqref{eq:middle-truncated-k-jump}. Letting $\eta\to0^+$ proves
\eqref{eq:middle-complete-k-jump}. Because $D<0$, the right-hand side
of \eqref{eq:middle-complete-k-jump} is strictly decreasing from
$+\infty$ to $0$ as $\lambda$ increases from $0$ to $+\infty$.
Therefore, if $k_L-k_R>0$, equation
\eqref{eq:middle-complete-k-jump} has the unique solution
\eqref{eq:middle-lambda-selection}. Differentiation gives
\eqref{eq:middle-lambda-nondegeneracy}.
Finally, from the definitions of $k$, $D$, and $\Gamma$,
\begin{equation}
w_1=k,
\qquad
w_2=\xi k-D,
\qquad
w_3=v_\delta k-\Gamma.
\label{eq:middle-W-reconstruction}
\end{equation}
Since $\xi=c$, $D$, and $\Gamma$ remain constant along the reduced
middle orbit,
\begin{align}
W_R-W_L
&=
\begin{pmatrix}
k_R-k_L\\
c(k_R-k_L)\\
v_\delta(k_R-k_L)
\end{pmatrix}
=
(k_R-k_L)
\begin{pmatrix}
1\\
c\\
v_\delta
\end{pmatrix}.
\end{align}
This proves \eqref{eq:middle-W-jump} and completes the proof.
\end{proof}
\begin{prop}[Limits of the outer trajectories on the blown-up boundary]
\label{prop:outer-sphere-limits}
The left reduced outer trajectory, expressed in the blow-up coordinates,
converges as $\zeta\to0^-$ to
\begin{equation}
P_L
=
\left(
0,\,
0,\,
-\frac{D}{R},\,
-\frac{\Gamma}{R},\,
0,\,
W_L,\,
c
\right).
\label{eq:P_L}
\end{equation}
The right reduced outer trajectory, expressed in the blow-up coordinates,
converges as $\zeta\to0^+$ to
\begin{equation}
P_R
=
\left(
0,\,
0,\,
\frac{D}{R},\,
\frac{\Gamma}{R},\,
0,\,
W_R,\,
c
\right).
\label{eq:P_R}
\end{equation}
The coordinates in \eqref{eq:P_L} and \eqref{eq:P_R} are ordered as
$(r,\bar\rho,\bar u,\bar v,\bar\varepsilon,W,\xi).$
\end{prop}

\begin{proof}
By Lemma~\ref{lem:outer-W-RH},
$D_L=D_R=D$ and $\Gamma_L=\Gamma_R=\Gamma$. Hence, as
$\zeta\to0^-$,
\begin{align}
u(\zeta)-c
&=
D\zeta+o(|\zeta|),
\\
v(\zeta)-v_\delta
&=
\Gamma\zeta+o(|\zeta|).
\end{align}
Moreover,
\begin{equation}
W(\zeta)\equiv W_L,
\qquad
\xi(\zeta)\equiv c,
\qquad
\varepsilon=0.
\label{eq:left-outer-constants}
\end{equation}
Since $\rho_0=\varepsilon\rho=0$ for every $\zeta\neq0$ and 
$r(\zeta)>0$ away from the limiting point, the blow-up relation
$\rho_0=r\bar\rho$ gives
\begin{equation}
\bar\rho(\zeta)
=\frac{\rho_0(\zeta)}{r(\zeta)}
=0,
\qquad\bar\varepsilon(\zeta)
=\frac{\sqrt{\varepsilon}}{r(\zeta)}
=0,
\end{equation}
for every $\zeta<0$ sufficiently close to zero. The blow-up transformation therefore
gives
\begin{equation}
r^2
=
\left(u-c\right)^2
+
\left(v-\frac{k_3}{k_1}\right)^2.
\end{equation}
Using \eqref{eq:left-inner-asymptotics-rho-u} and
\eqref{eq:left-inner-asymptotics-v}, we obtain
\begin{align}
r(\zeta)^2
&=
\left(D_L^2+\Gamma_L^2\right)\zeta^2
+
o(\zeta^2)=
R_L^2\zeta^2+o(\zeta^2).
\end{align}
Hence
\begin{equation}
r(\zeta)
=
R_L|\zeta|+o(|\zeta|).
\end{equation}
Because $\zeta\to0^-$, one has $|\zeta|=-\zeta$. Consequently,
\begin{align}
\bar u(\zeta)
&=
\frac{u(\zeta)-c}{r(\zeta)}
=
\frac{D_L\zeta+o(|\zeta|)}
     {R_L|\zeta|+o(|\zeta|)}
\longrightarrow
-\frac{D_L}{R_L},
\\
\bar v(\zeta)
&=
\frac{v(\zeta)-k_3/k_1}{r(\zeta)}
=
\frac{\Gamma_L\zeta+o(|\zeta|)}
     {R_L|\zeta|+o(|\zeta|)}
\longrightarrow
-\frac{\Gamma_L}{R_L}.
\end{align}
Together with \eqref{eq:left-outer-constants}, these limits give
\eqref{eq:P_L}. For the right reduced outer trajectory,
Proposition~\ref{prop:right-reduced-orbit} gives, as $\zeta\to0^+$,
\begin{align}
u(\zeta)-c
&=
D_R\zeta+o(\zeta),
\label{eq:right-u-asymptotic}
\\
v(\zeta)-\frac{k_3}{k_1}
&=
\Gamma_R\zeta+o(\zeta).
\label{eq:right-v-asymptotic}
\end{align}
Also, $W(\zeta)\equiv W_R, \ \xi(\zeta)\equiv c, \ \rho_0(\zeta)=0, \ \varepsilon=0.$
It follows that $r(\zeta)
=R_R\zeta+o(\zeta),$ because $\zeta>0$ on the right. Therefore,
\begin{align}
\bar u(\zeta)
&=
\frac{u(\zeta)-c}{r(\zeta)}
\longrightarrow
\frac{D_R}{R_R},
\\
\bar v(\zeta)
&=
\frac{v(\zeta)-k_3/k_1}{r(\zeta)}
\longrightarrow
\frac{\Gamma_R}{R_R},
\end{align}
while $\bar\rho(\zeta)=0,
\ \bar\varepsilon(\zeta)=0$ for every $\zeta>0$ sufficiently close to zero. Therefore the lifted right outer trajectory converges to $P_R$, which proves
\eqref{eq:P_R}.
\end{proof}

\begin{prop}[Reduced outer--middle compatibility]
\label{prop:reduced-outer-middle-compatibility}
Let $P_L$ and $P_R$ be the limiting points of the lifted left and right
outer trajectories. Under the generalized Rankine--Hugoniot relation,
their directional components are $(\bar u_L,\bar v_L)
=
\left(
-\frac{D}{R},
-\frac{\Gamma}{R}
\right)$ and
$(\bar u_R,\bar v_R)
=
\left(
\frac{D}{R},
\frac{\Gamma}{R}
\right).$ In particular, $D\bar v_L-\Gamma\bar u_L=0,
\ D\bar v_R-\Gamma\bar u_R=0.$ Thus both outer limiting directions lie on the relation $D\bar v-\Gamma\bar u=0$ selected by the reduced middle orbit.
\end{prop}
\begin{proof}
At the left endpoint,
\begin{align}
D\bar v_L-\Gamma\bar u_L
&=D\left(-\frac{\Gamma}{R}\right)
-\Gamma\left(-\frac{D}{R}\right)=0.
\end{align}
At the right endpoint,
\begin{align}
D\bar v_R-\Gamma\bar u_R
&=D\left(\frac{\Gamma}{R}\right)-
\Gamma\left(\frac{D}{R}\right)=0.
\end{align}
Therefore both limiting points lie on the relation determined
by $D$ and $\Gamma$.
\end{proof}
\begin{lemma}[Coordinate compatibility]
\label{lem:middle-coordinate-compatibility}
On the overlap between the spherical blow-up coordinates and the
distinguished middle coordinates, one has $\frac{\bar v}{\bar u}
=\frac{h}{y}$ whenever $\bar u\neq0$ and $y\neq0$. Consequently, the relation $D\bar v-\Gamma\bar u=0$ is equivalent to $Dh-\Gamma y=0.$
\end{lemma}
\begin{proof}
Equating the two representations of $u-\xi$ and $v-v_\delta$ gives
$r\bar u=\delta\frac{y}{z},
\ r\bar v=\delta\frac{h}{z}.$ Dividing the second identity by the first yields $\frac{\bar v}{\bar u}=\frac{h}{y}.$
Cross-multiplication gives $y\bar v=\bar u h.$ Therefore,
\begin{align}
y\left(D\bar v-\Gamma\bar u\right)
&=
D y\bar v-\Gamma y\bar u=
D\bar u h-\Gamma y\bar u=
\bar u\left(Dh-\Gamma y\right).
\end{align}
Hence, whenever $\bar u\neq0$ and $y\neq0$, $D\bar v-\Gamma\bar u=0$
if and only if $Dh-\Gamma y=0.$
\end{proof}
By Propositions~\ref{prop:left-reduced-orbit}
and~\ref{prop:right-reduced-orbit}, together with
Lemma~\ref{lem:outer-W-RH},
the left and right reduced outer trajectories determine the same values
of the invariants $D$ and $\Gamma$. In
Subsection~\ref{subsubsec:outer-corner-transversality}, we show that,
under the weighted blow-down transformation, the asymptotic limits of
the reduced outer trajectories correspond to the incoming and outgoing
corner equilibria of the weighted middle problem. The remaining
mismatch is therefore encoded solely by the scalar variable $k$.
Since the reduced middle orbit satisfies $k_L-k_R=\frac{-2\pi D}{\lambda^3},$ there exists a unique positive value $\lambda=
\left(
\frac{-2\pi D}{k_L-k_R}
\right)^{1/3},$ for which the reduced middle orbit joins the two limiting equilibria. Consequently, the left reduced outer orbit, the reduced middle orbit, and the right reduced outer orbit form the desired singular concatenated orbit.
\subsection{Equilibrium branches on the blown-up boundary}
\label{subsec:equilibrium-branches}
Along the reduced flow on the blown-up sphere, the variables
$w_1,w_2,w_3,$ and $\xi$ remain constant. Define
\begin{equation}
D:=\xi w_1-w_2,
\qquad
\Gamma:=v_\delta w_1-w_3,
\qquad
R:=\sqrt{D^2+\Gamma^2},
\label{eq:D-Gamma-R}
\end{equation}
where $v_\delta=k_3/k_1$. We restrict attention to the region
$R>0$. Introduce
\begin{equation}
g:=\Gamma\bar u-D\bar v.
\label{eq:g-reduced}
\end{equation}

\begin{lemma}
\label{lem:reduced-invariant-relation}
The set
\begin{equation}
\mathcal M_0
=
\left\{
r=0,\;
\Gamma\bar u-D\bar v=0
\right\}
\label{eq:M0}
\end{equation}
is invariant under the reduced flow on the blown-up sphere.
\end{lemma}
\begin{proof}
Since $D$ and $\Gamma$ are constant along the reduced flow, direct
differentiation gives
\begin{align}
\frac{dg}{d\tau}
&=
\Gamma\frac{d\bar u}{d\tau}
-
D\frac{d\bar v}{d\tau}\\
&=
\Gamma
\left[
(1-\bar u^2)D-\bar u\bar v\Gamma
\right]
-
D
\left[
-\bar u\bar vD+(1-\bar v^2)\Gamma
\right]=
-(D\bar u+\Gamma\bar v)
(\Gamma\bar u-D\bar v)=
-(D\bar u+\Gamma\bar v)g.
\label{eq:g-evolution}
\end{align}
Thus $g=0$ is preserved by the reduced flow.
\end{proof}
The reduced problem contains two distinguished equilibrium branches.
\begin{prop}[Equilibrium branches on the blown-up sphere]
\label{prop:equilibrium-branches}
Assume $R>0$. Define
\begin{equation}
\mathcal E_-
=
\left\{
\begin{aligned}
&r=0,\qquad
\bar\rho=0,\qquad
\bar\varepsilon=0,
&\bar u=-\frac{D}{R},\qquad
\bar v=-\frac{\Gamma}{R}
\end{aligned}
\right\},
\label{eq:E-minus}
\end{equation}
and
\begin{equation}
\mathcal E_+
=
\left\{
\begin{aligned}
&r=0,\qquad
\bar\rho=0,\qquad
\bar\varepsilon=0,
&\bar u=\frac{D}{R},\qquad
\bar v=\frac{\Gamma}{R}
\end{aligned}
\right\}.
\label{eq:E-plus}
\end{equation}
The variables $(w_1,w_2,w_3,\xi)$ are free parameters on each set.
Consequently, $\mathcal E_-$ and $\mathcal E_+$ are four-dimensional
sets of equilibria of the reduced blown-up system. If $D<0$, then
\begin{equation}
-\frac{D}{R}>0,
\qquad
\frac{D}{R}<0.
\label{eq:sign-E-branches}
\end{equation}
Hence $\mathcal E_-$ lies in the positive-$\bar u$ region represented
by the entry coordinates, whereas $\mathcal E_+$ lies in the
negative-$\bar u$ region represented by the exit coordinates.
\end{prop}
\begin{proof}
On the set
\begin{equation}
r=0,
\qquad
\bar\rho=0,
\qquad
\bar\varepsilon=0,
\label{eq:equatorial-set}
\end{equation}
the spherical constraint reduces to
\begin{equation}
\bar u^2+\bar v^2=1.
\label{eq:equatorial-sphere}
\end{equation}
The additional relation
\begin{equation}
\Gamma\bar u-D\bar v=0
\label{eq:equatorial-invariant-relation}
\end{equation}
implies that $(\bar u,\bar v)$ is parallel to $(D,\Gamma)$. Therefore,
there exists $\lambda\in\mathbb R$ such that
$
(\bar u,\bar v)=\lambda(D,\Gamma).
$ Using \eqref{eq:equatorial-sphere}, we obtain
\begin{equation}
1
=
\lambda^2(D^2+\Gamma^2)
=
\lambda^2R^2,
\end{equation}
and hence $\lambda=\pm\frac{1}{R}.$ Thus the only points satisfying
\eqref{eq:equatorial-set}--\eqref{eq:equatorial-invariant-relation}
are $(\bar u,\bar v)
=\left(
-\frac{D}{R},
-\frac{\Gamma}{R}
\right),$ 
and $(\bar u,\bar v)
=\left(
\frac{D}{R},
\frac{\Gamma}{R}
\right).$ Substitution into the reduced equations shows that all directional derivatives vanish at these points. Moreover,
$w_1,w_2,w_3,$ and $\xi$ are constant in the reduced problem.
Therefore, each choice of $(w_1,w_2,w_3,\xi)$ determines one point on
each equilibrium branch, so both branches are four-dimensional.

Finally, if $D<0$, then $R>0$ gives $-D/R>0$ and $D/R<0$.
The first branch therefore lies in the positive-$\bar u$ region
represented by the entry coordinates, while the second lies in the
negative-$\bar u$ region represented by the exit coordinates.
\end{proof}
\subsection{Persistence of the singular configuration}
\label{subsec:persistence-singular-configuration}
The preceding analysis constructs the reduced outer trajectories and the
explicit reduced middle passage. We now study the persistence and matching
of these orbit segments for sufficiently small $\varepsilon>0$.

The argument separates into three parts. First, the outer unstable and
stable manifolds persist up to fixed transverse sections lying away from
the singular corner. Second, the blown-up chart dynamics provide local
transition maps through the entry, central, and exit regions. Third, the
resulting maps are matched by means of the nondegenerate scalar condition
that determines the parameter $\lambda$ of the middle orbit.

The reduced singular construction above is valid under the hypotheses
of the corresponding outer-orbit propositions.  The positive-
$\varepsilon$ persistence argument developed below is carried out for
$\alpha\in\mathbb{Q}\cap(0,1].$ For $\alpha=\frac12$ and $\alpha=1$, the central weighted vector field is smooth in the original weighted corner variables.  For a general rational exponent, we introduce a ramified weighted coordinate that converts the fractional pressure term into an integer power. The standard smooth Corner Lemma is not invoked here for irrational values of $\alpha$.

\begin{prop}[Persistence of the outer orbit segments]
\label{prop:outer-persistence}
Let the hypotheses of
Propositions~\ref{prop:left-reduced-orbit} and
\ref{prop:right-reduced-orbit} hold, and let $c$ denote the
singular-shock speed. Then there exist compact intervals $I_L$ and
$I_R$, each containing $c$ in its interior, such that
\begin{equation}
I_L
\Subset
\left(
-\infty,\lambda_1(H_L)
\right),
\qquad
I_R
\Subset
\left(
\lambda_3(H_R),+\infty
\right).
\label{eq:outer-xi-intervals}
\end{equation}
Define
\begin{equation}
S_{0,L}
=
\left\{
(H,V,\xi):
H=H_L,\;
V=0,\;
\xi\in I_L
\right\}
\label{eq:S0L-compact}
\end{equation}
and
\begin{equation}
S_{2,R}
=
\left\{
(H,V,\xi):
H=H_R,\;
V=0,\;
\xi\in I_R
\right\}.
\label{eq:S2R-compact}
\end{equation}

There exists $\kappa_0>0$ such that, for every fixed
$0<\kappa<\kappa_0$, the sections
\begin{equation}
\Sigma_L^\kappa
=
\left\{
u-\xi=\kappa
\right\},
\qquad
\Sigma_R^\kappa
=
\left\{
\xi-u=\kappa
\right\}
\label{eq:outer-kappa-sections}
\end{equation}
lie in the outer region and in the domains of the entry and exit
coordinate charts, respectively.

For every sufficiently small $\varepsilon>0$, the locally invariant
perturbation of $S_{0,L}$ possesses a local unstable manifold whose
forward image intersects $\Sigma_L^\kappa$ transversely. Similarly,
the locally invariant perturbation of $S_{2,R}$ possesses a local
stable manifold whose backward image intersects
$\Sigma_R^\kappa$ transversely. The resulting intersection manifolds
depend $C^1$-smoothly on $\varepsilon$.

At $\varepsilon=0$, the distinguished left and right reduced outer
trajectories intersect their respective sections at unique points
$Q_L^0$ and $Q_R^0$.
\end{prop}

\begin{proof}
The overcompressivity inequalities imply $\lambda_3(H_R)<c<\lambda_1(H_L).$ Since both inequalities are strict, compact intervals $I_L$ and $I_R$ satisfying \eqref{eq:outer-xi-intervals} may be chosen with
$c$ in their interiors.

The fast phase space is
\begin{equation}
\mathcal P
=
\left\{
(H,V,\xi):
H\in\mathbb R^3,\;
V\in\mathbb R^3,\;
\xi\in\mathbb R
\right\},
\end{equation}
and hence
\begin{equation}
\dim\mathcal P=7.
\label{eq:outer-phase-dimension}
\end{equation}
Each of the sets $S_{0,L}$ and $S_{2,R}$ is one-dimensional, since
only the variable $\xi$ is free.

At a point of $S_{0,L}$, the three eigenvalues in the normal fast
directions are $\lambda_j(H_L)-\xi,\ j=1,2,3.$ For $\xi\in I_L$,
\begin{equation}
\lambda_j(H_L)-\xi
\geq
\lambda_1(H_L)-\sup I_L
>0,
\qquad
j=1,2,3.
\label{eq:left-uniform-gap}
\end{equation}
Thus $S_{0,L}$ is uniformly normally hyperbolic and has three
unstable normal directions. Consequently,
\begin{equation}
\dim\mathcal W_0^u(S_{0,L})
=
\dim S_{0,L}+3
=
4.
\label{eq:left-unstable-dimension}
\end{equation}
Similarly, at a point of $S_{2,R}$,
\begin{equation}
\lambda_j(H_R)-\xi
\leq
\lambda_3(H_R)-\inf I_R
<0,
\qquad
j=1,2,3.
\label{eq:right-uniform-gap}
\end{equation}
Hence $S_{2,R}$ is uniformly normally hyperbolic and has three
stable normal directions. Therefore,
\begin{equation}
\dim\mathcal W_0^s(S_{2,R})
=
\dim S_{2,R}+3
=
4.
\label{eq:right-stable-dimension}
\end{equation}
Fenichel theory now yields, for all sufficiently small
$\varepsilon>0$, locally invariant perturbations
$S_{0,L}^\varepsilon$ and $S_{2,R}^\varepsilon$, together with their
local unstable and stable manifolds $\mathcal W_\varepsilon^u(S_{0,L}^\varepsilon),
\ \mathcal W_\varepsilon^s(S_{2,R}^\varepsilon).$ On compact subsets these manifolds depend $C^1$-smoothly on
$\varepsilon$.

We first consider the left reduced orbit. By
Proposition~\ref{prop:left-reduced-orbit},
\begin{equation}
u(\zeta)-c
=
-B_L\zeta+o(|\zeta|)
\label{eq:left-section-asymptotic}
\end{equation}
as $\zeta\to0^-$, where $B_L>0$. Hence $\frac{d}{d\zeta}(u-c)
=-B_L+o(1).$ There exists $\zeta_L<0$ such that
\begin{equation}
\frac{d}{d\zeta}(u-c)
<
-\frac{B_L}{2}
<0
\label{eq:left-section-monotonicity}
\end{equation}
whenever $\zeta_L<\zeta<0$.

Moreover,
\begin{equation}
u(\zeta)-c\longrightarrow0
\qquad
\text{as }\zeta\to0^-.
\end{equation}
After decreasing $\kappa_0$ if necessary, it follows from strict
monotonicity that, for every $0<\kappa<\kappa_0$, there exists a
unique $\zeta_L^\kappa\in(\zeta_L,0)$ satisfying
\begin{equation}
u(\zeta_L^\kappa)-c=\kappa.
\end{equation}
Since $\xi=c$ along the reduced outer orbit, the corresponding point
$Q_L^0$ belongs to $\Sigma_L^\kappa$. Let $\psi_L(H,V,\xi)=u-\xi-\kappa,$
so that $\Sigma_L^\kappa
=\left\{
\psi_L=0
\right\}.$ Since
$\psi_L(H,V,\xi)=u-\xi-\kappa,$ its gradient is nonzero. Therefore
$\Sigma_L^\kappa=\{\psi_L=0\}$ is a smooth codimension-one
hypersurface of $\mathcal P$. In particular, $\dim\Sigma_L^\kappa=6.$

Let $X_0$ denote the reduced fast vector field. Along the
distinguished left reduced orbit, $\xi=c$, so
\begin{equation}
\frac{d}{d\zeta}\psi_L
=
\frac{d}{d\zeta}(u-c).
\end{equation}
Hence, at $\zeta=\zeta_L^\kappa$,
\begin{equation}
\frac{d}{d\zeta}\psi_L
\Big|_{\zeta=\zeta_L^\kappa}
<
-\frac{B_L}{2}
<0.
\label{eq:left-vector-field-crossing}
\end{equation}
It follows that
\begin{equation}
X_0(Q_L^0)
\notin
T_{Q_L^0}\Sigma_L^\kappa.
\label{eq:left-vector-not-section-tangent}
\end{equation}

On the other hand, the distinguished left reduced orbit is contained
in $\mathcal W_0^u(S_{0,L})$. Since this unstable manifold is
invariant under the reduced flow, its tangent space contains the
reduced vector field at every point of that orbit. Thus
\begin{equation}
X_0(Q_L^0)
\in
T_{Q_L^0}\mathcal W_0^u(S_{0,L}).
\label{eq:left-vector-unstable-tangent}
\end{equation}

Combining
\eqref{eq:left-vector-not-section-tangent} and
\eqref{eq:left-vector-unstable-tangent}, we see that
$T_{Q_L^0}\mathcal W_0^u(S_{0,L})$ contains a direction that does
not belong to $T_{Q_L^0}\Sigma_L^\kappa$. Since
$T_{Q_L^0}\Sigma_L^\kappa$ has codimension one in
$T_{Q_L^0}\mathcal P$, the addition of any vector outside it spans
the missing normal direction. Consequently,
\begin{equation}
T_{Q_L^0}\mathcal W_0^u(S_{0,L})
+
T_{Q_L^0}\Sigma_L^\kappa
=
T_{Q_L^0}\mathcal P.
\label{eq:left-outer-section-transversality}
\end{equation}
Thus $\mathcal W_0^u(S_{0,L})$ intersects
$\Sigma_L^\kappa$ transversely.

The dimension of the transverse intersection is therefore
\begin{equation}
\dim\left(
\mathcal W_0^u(S_{0,L})
\cap
\Sigma_L^\kappa
\right)
=
4+6-7
=
3.
\label{eq:left-intersection-dimension}
\end{equation}
The uniqueness asserted in the proposition concerns the point at
which the distinguished reduced orbit meets this three-dimensional
intersection manifold, not uniqueness of the full manifold
intersection.

Choose a compact segment of the distinguished reduced left orbit
joining a neighborhood of the local unstable manifold of
$S_{0,L}$ to the point $Q_L^0\in\Sigma_L^\kappa$. Since $\kappa>0$ is fixed, this orbit segment remains a positive
distance from the singular endpoint $\rho=+\infty$. Consequently,
the reduced vector field is smooth on a compact neighborhood of the
segment.

Choose a transverse local section near the initial endpoint of the
segment (intersecting the local unstable manifold) and another
transverse section through $Q_L^0$. The finite-time flow between
these two sections is well defined on a neighborhood of the initial section and depends $C^1$-smoothly on the initial point and on $\varepsilon$. The local
unstable manifold
$\mathcal W_\varepsilon^u(S_{0,L}^\varepsilon)$ is
$C^1$-close to $\mathcal W_0^u(S_{0,L})$. Its forward image under
the finite-time flow is therefore $C^1$-close to the corresponding
reduced image near $Q_L^0$.

Since $\mathcal W_\varepsilon^u(S_{0,L}^\varepsilon)$ depends
$C^1$-smoothly on $\varepsilon$, its tangent spaces vary
continuously with $\varepsilon$. Therefore the transversality
condition \eqref{eq:left-outer-section-transversality} persists for
all sufficiently small $\varepsilon>0$. Hence
$\mathcal W_\varepsilon^u(S_{0,L}^\varepsilon)$ intersects
$\Sigma_L^\kappa$ transversely. By the parameter-dependent implicit
function theorem, the resulting three-dimensional intersection
manifold depends $C^1$-smoothly on $\varepsilon$.

We next consider the right reduced orbit. By
Proposition~\ref{prop:right-reduced-orbit},
\begin{equation}
c-u(\zeta)
=
B_R\zeta+o(\zeta)
\label{eq:right-section-asymptotic}
\end{equation}
as $\zeta\to0^+$, where $B_R>0$. Thus $\frac{d}{d\zeta}(c-u)
=B_R+o(1).$ There exists $\zeta_R>0$ such that
\begin{equation}
\frac{d}{d\zeta}(c-u)
>
\frac{B_R}{2}
>0
\label{eq:right-section-monotonicity}
\end{equation}
for $0<\zeta<\zeta_R$.

After decreasing $\kappa_0$ once more if necessary, strict
monotonicity implies that, for every $0<\kappa<\kappa_0$, there is a
unique $\zeta_R^\kappa\in(0,\zeta_R)$ satisfying $c-u(\zeta_R^\kappa)=\kappa.$ Because $\xi=c$ along the reduced right orbit, the corresponding point $Q_R^0$ lies on $\Sigma_R^\kappa$. Define
$\psi_R(H,V,\xi)=\xi-u-\kappa.$ Then $\Sigma_R^\kappa
=\left\{\psi_R=0
\right\}$ is a smooth codimension-one hypersurface of $\mathcal P$, and hence $\dim\Sigma_R^\kappa=6.$ At $Q_R^0$,
\begin{equation}
d\psi_R(X_0(Q_R^0))
=
\frac{d}{d\zeta}(c-u)
>
\frac{B_R}{2}
>0.
\label{eq:right-vector-field-crossing}
\end{equation}
Therefore,
\begin{equation}
X_0(Q_R^0)
\notin
T_{Q_R^0}\Sigma_R^\kappa.
\end{equation}

The distinguished right reduced orbit is contained in
$\mathcal W_0^s(S_{2,R})$. Since the stable manifold is invariant
under the reduced flow,
\begin{equation}
X_0(Q_R^0)
\in T_{Q_R^0}\mathcal W_0^s(S_{2,R}).
\end{equation}
Thus the stable tangent space contains a direction transverse to the
codimension-one section, and consequently
\begin{equation}
T_{Q_R^0}\mathcal W_0^s(S_{2,R})+T_{Q_R^0}\Sigma_R^\kappa=T_{Q_R^0}\mathcal P.
\label{eq:right-outer-section-transversality}
\end{equation}
Hence $\mathcal W_0^s(S_{2,R})$ intersects
$\Sigma_R^\kappa$ transversely, and
\begin{equation}
\dim\left(
\mathcal W_0^s(S_{2,R})
\cap
\Sigma_R^\kappa
\right)
=
4+6-7
=
3.
\label{eq:right-intersection-dimension}
\end{equation}

Choose a compact segment of the distinguished reduced right orbit
joining $Q_R^0$ to a neighborhood of the local stable manifold of
$S_{2,R}$. This segment also remains a positive distance from the
singular endpoint. Applying the same finite-time flow,
Fenichel-persistence, openness-of-transversality, and
parameter-dependent implicit-function arguments as on the left shows
that the backward image of
$\mathcal W_\varepsilon^s(S_{2,R}^\varepsilon)$ intersects
$\Sigma_R^\kappa$ transversely and that the resulting
three-dimensional intersection manifold depends $C^1$-smoothly on
$\varepsilon$.

This completes the proof.
\end{proof}

\begin{remark}
\label{rem:outer-persistence-not-corner}
Proposition~\ref{prop:outer-persistence} constructs only the
persistent outer pieces of the singular concatenation. The perturbed
unstable and stable manifolds are followed only to the fixed sections
$\Sigma_L^\kappa$ and $\Sigma_R^\kappa$. Since $\kappa>0$ is fixed,
these sections remain a positive distance from the singular corner.

No application of the Corner Lemma is made in the proof above. The
continuation from the fixed outer sections to the inner overlap
sections requires a separate analysis in the entry and exit charts
and will be established below using Schecter's Corner Lemma.
\end{remark}
\subsubsection{Reduced dynamics on the blown-up boundary}
\label{subsubsec:reduced-boundary-dynamics}
The blow-up transformation separates the singular point into the sphere $S^3$. To understand the limiting behavior of the trajectories, we first consider the reduced problem obtained by setting $r=0$, which describes the leading-order dynamics in the singular layer and serves as the starting point for the geometric construction of the viscous profile.

Assume that $\bar{\rho}>0$ and $\bar{\varepsilon}>0$. Setting $r=0$ and rescaling time by 
$\frac{\overline{\rho}}{\overline{\eps}^2},$
(which amounts to multiplying the right-hand side of the system by
$\bar{\rho}/\bar{\varepsilon}^2$), and, for notational simplicity,
continuing to denote the new time variable by $\tau$, yields
\begin{align}
\begin{cases}
\displaystyle
\frac{d\overline{\rho}}{d\tau}
= -\overline{u}\,\overline{\rho}
(\xi w_1-w_2)
-\overline{v}\,\overline{\rho}
\left(\frac{k_3}{k_1}w_1-w_3\right),
\\[0.8em]
\displaystyle
\frac{d\overline{u}}{d\tau}
= (1-\overline{u}^2)(\xi w_1-w_2)
-\overline{u}\,\overline{v}
\left(\frac{k_3}{k_1}w_1-w_3\right),
\\[0.8em]
\displaystyle
\frac{d\overline{v}}{d\tau}
= -\overline{v}\,\overline{u}
(\xi w_1-w_2)
+(1-\overline{v}^2)
\left(\frac{k_3}{k_1}w_1-w_3\right),
\\[0.8em]
\displaystyle
\frac{d\overline{\eps}}{d\tau}
= -\overline{u}\,\overline{\eps}
(\xi w_1-w_2)
-\overline{\eps}\,\overline{v}
\left(\frac{k_3}{k_1}w_1-w_3\right),
\\[0.8em]
\displaystyle
\frac{dw_1}{d\tau}
=
\frac{dw_2}{d\tau}
=
\frac{dw_3}{d\tau}
=
\frac{d\xi}{d\tau}
=0.
    \end{cases}
    \label{gen_u_r=0}
\end{align}
Along the reduced flow,
$w_1,\;w_2,\;w_3,\;\xi$
remain constant. Consequently, the dynamics are completely determined by
$(\bar\rho,\bar u,\bar v,\bar\varepsilon),$
with the remaining variables acting as parameters inherited from the outer solution. Our objective is to identify a distinguished orbit of the reduced system that connects the incoming and outgoing states associated with the $\delta$-shock. To accomplish this, we first identify an invariant manifold on which the reduced dynamics simplify considerably.

\subsubsection{Entry-region passage}
\label{subsubsec:entry-region-passage}

The reduced outer trajectory constructed above approaches the blown-up
boundary through the region where $u-\xi>0$. We therefore introduce the
entry directional chart, denoted by $K_{\mathrm{in}}$, corresponding to
the positive-$\bar u$ direction. This chart resolves the approach to
$(\rho_0,\varepsilon)=(0,0)$ from the side $u-\xi>0$.

The coordinates in $K_{\mathrm{in}}$ are defined by
\begin{equation}
u-\xi=r_0,\ \ \ 
\rho_0=r_0\bar\rho_0,\ \ \ 
v-v_\delta=r_0\bar v_0,\ \ \ 
\varepsilon=r_0^2\bar\varepsilon_0^2.
\label{eq:entry-coordinates}
\end{equation}
\begin{lemma}[Left outer endpoint in the entry chart]
\label{lem:left-outer-endpoint-entry}

The left outer endpoint $P_L$ is represented in the entry chart
$K_{\mathrm{in}}$ by
\begin{equation}
P_L^{\mathrm{in}}
=
\left(
0,\,
0,\,
\frac{\Gamma_L}{D_L},\,
0,\,
W_L,\,
c
\right),
\label{eq:PL-entry-chart}
\end{equation}
where the coordinates are ordered as $(r_0,\bar\rho_0,\bar v_0,\bar\varepsilon_0,W,\xi).$
\end{lemma}
\begin{proof}
The spherical and entry-chart coordinates are related by
\begin{equation}
r_0=r\bar u,
\qquad
\bar\rho_0=\frac{\bar\rho}{\bar u},
\qquad
\bar v_0=\frac{\bar v}{\bar u},
\qquad
\bar\varepsilon_0=\frac{\bar\varepsilon}{\bar u}.
\label{eq:spherical-to-entry}
\end{equation}
By Proposition~\ref{prop:outer-sphere-limits}, at $P_L$,
$r=0, \ \bar\rho=0, \ \bar u=-\frac{D}{R}, \ \bar v=-\frac{\Gamma}{R}, \ \bar\varepsilon=0, \ W=W_L, \ \xi=c.$ Since $D<0$, one has $-D/R>0$, so $P_L$ belongs to
$K_{\mathrm{in}}$. Moreover,
$\bar v_0
=\frac{\bar v}{\bar u}=
\frac{-\Gamma/R}{-D/R}
=\frac{\Gamma}{D}.$
Also, $r_0=0,
\ \bar\rho_0=0,
\ \bar\varepsilon_0=0.$
This proves \eqref{eq:PL-entry-chart}.
\end{proof}
For $\alpha=1$, the resulting vector field is smooth in $r_0$ up to
$r_0=0$. For $\alpha=\frac12$, the pressure term contains
$r_0^{1/2}$, and a weighted radial coordinate will be introduced below.
We first record the entry-region equations for both exponents.

Rewriting system~\eqref{GSPT_eps_sys} in the entry coordinates and
applying the chain rule yields \begin{equation}
\begin{cases}
\displaystyle
\frac{dr_0}{d\tau}
=
r_0\,\Phi_1,
\\[2ex]
\displaystyle
\frac{d\bar\varepsilon_0}{d\tau}
=
-\bar\varepsilon_0\,\Phi_1,
\\[2ex]
\displaystyle
\frac{d\xi}{d\tau}
=
r_0^2\bar\varepsilon_0^2,
\\[2ex]
\displaystyle
\frac{du}{d\tau}
=
r_0^2\bar\varepsilon_0^2
+r_0\,\Phi_1,
\\[2ex]
\displaystyle
\frac{d\bar\rho_0}{d\tau}
=
-\bar\rho_0\,\Phi_1
+r_0\bar\rho_0
-r_0\bar\varepsilon_0^2w_1,
\\[2ex]
\displaystyle
\frac{d\bar v_0}{d\tau}
=
-\bar v_0\,\Phi_1
+\frac{\bar\varepsilon_0^2}{\bar\rho_0}
\left[
\left(\frac{k_3}{k_1}+r_0\bar v_0\right)w_1-w_3
\right],
\\[2ex]
\displaystyle
\frac{dw_1}{d\tau}
=
-r_0\bar\rho_0,
\\[2ex]
\displaystyle
\frac{dw_2}{d\tau}
=
-r_0\bar\rho_0(\xi+r_0),
\\[2ex]
\displaystyle
\frac{dw_3}{d\tau}
=
-r_0\bar\rho_0
\left(\frac{k_3}{k_1}+r_0\bar v_0\right),
\end{cases}\label{eq:chart1_full}
\end{equation}
where
\begin{equation}\label{eq:Phi1}
\Phi_1
=
\frac{\bar\varepsilon_0^2}{\bar\rho_0}
\left[(\xi+r_0)w_1-w_2\right]
-
A\!\left(\frac{k_3}{k_1}+r_0\bar v_0\right)
\frac{r_0^\alpha\bar\varepsilon_0^{2\alpha+2}}
{\bar\rho_0^{\alpha+1}}
-r_0\bar\varepsilon_0^2 .
\end{equation}
\begin{lemma}[Regularity of the entry-region vector field]
\label{lem:entry-regularity}

Let
$$
\alpha=\frac{p}{q}\in\mathbb Q\cap(0,1],
\qquad
\gcd(p,q)=1.
$$
Restrict attention to a compact set on which $\bar\rho_0$ is bounded
away from zero. Introduce the ramified radial coordinate
\begin{equation}
r_0=s_0^q,
\qquad
s_0\geq0.
\label{eq:r0-s0}
\end{equation}
Then the entry-region system, written in the variables $(s_0,\bar\rho_0,\bar v_0,\bar\varepsilon_0,
w_1,w_2,w_3,\xi),$ extends smoothly to $s_0=0$.
\end{lemma}

\begin{proof}
Under the substitution $r_0=s_0^q$, the fractional power appearing
in \eqref{eq:Phi1} becomes
$r_0^\alpha
=\left(s_0^q\right)^{p/q}
=s_0^p.$ Accordingly, define
\begin{equation}
\widetilde\Phi_1
=
\frac{\bar\varepsilon_0^2}{\bar\rho_0}
\left[
(\xi+s_0^q)w_1-w_2
\right]
-
A\left(
v_\delta+s_0^q\bar v_0
\right)
\frac{s_0^p\bar\varepsilon_0^{2\alpha+2}}
{\bar\rho_0^{\alpha+1}}
-
s_0^q\bar\varepsilon_0^2.
\label{eq:Phi1-weighted}
\end{equation}
Since $\frac{dr_0}{d\tau}
=r_0\Phi_1,$ we have $q s_0^{q-1}\frac{ds_0}{d\tau}
=s_0^q\widetilde\Phi_1,$ and therefore
\begin{equation}
\frac{ds_0}{d\tau}
=
\frac1q s_0\widetilde\Phi_1.
\label{eq:s0-entry}
\end{equation}
All remaining equations are obtained from
\eqref{eq:chart1_full} by replacing $r_0$ with $s_0^q$ and
$r_0^\alpha$ with $s_0^p$. Since $\bar\rho_0$ is bounded away from
zero, all right-hand sides are smooth functions of $s_0$ up to
$s_0=0$.
\end{proof}

\paragraph{Resolution of the entry endpoint.}

The directional entry coordinates regularize the vector field for
$\bar\rho_0>0$, but the endpoint selected by the reduced outer orbit
satisfies $\bar\rho_0=0,\ \bar\varepsilon_0=0.$ Consequently, the quotient
$\bar\varepsilon_0^2/\bar\rho_0$ appearing in
\eqref{eq:chart1_full} is not defined at the endpoint. We therefore
perform a second directional blow-up of
$(\bar\rho_0,\bar\varepsilon_0)$.

In the region $\bar\varepsilon_0\geq0$, introduce
\begin{equation}
\bar\varepsilon_0=q_0,
\qquad
\bar\rho_0=h_0q_0,
\qquad
q_0\geq0,
\qquad
h_0>0.
\label{eq:entry-secondary-blowup}
\end{equation}
Then
\begin{equation}
\frac{\bar\varepsilon_0^2}{\bar\rho_0}
=
\frac{q_0}{h_0}.
\label{eq:entry-secondary-quotient}
\end{equation}
Thus, the apparently singular coefficient extends continuously, and
indeed smoothly, to $q_0=0$ on compact sets on which $h_0$ remains
bounded away from zero.
\begin{lemma}[Reduced dynamics after the secondary entry blow-up]
\label{lem:entry-endpoint-desingularization}

Let $\alpha=\frac{p}{q}\in\mathbb Q\cap(0,1],
\
\gcd(p,q)=1,$ and introduce the ramified radial coordinate
$r_0=s_0^q,\ s_0\geq0,$ together with the secondary coordinates
$\bar\varepsilon_0=q_0,
\ \bar\rho_0=h_0q_0,
\ q_0\geq0,
\ h_0>0.$ On the reduced set $s_0=0$, freeze the distinguished values
\begin{equation}
W=W_L,
\qquad
\xi=c,
\end{equation}
and define
\begin{equation}
D=cw_{1,L}-w_{2,L}<0,
\qquad
\Gamma=v_\delta w_{1,L}-w_{3,L},
\qquad
z_0=\Gamma-D\bar v_0.
\end{equation}

For $q_0>0$, introduce the desingularized time $\sigma_0$ by
\begin{equation}
\frac{d\sigma_0}{d\tau}
=
\frac{q_0}{h_0},
\qquad
\frac{d}{d\sigma_0}
=
\frac{h_0}{q_0}\frac{d}{d\tau}.
\label{eq:entry-endpoint-time}
\end{equation}

Then the reduced boundary equations smoothly extend to $q_0=0$ and
are
\begin{equation}
\frac{dq_0}{d\sigma_0}
=
-Dq_0,
\qquad
\frac{dh_0}{d\sigma_0}
=
0,
\qquad
\frac{dz_0}{d\sigma_0}
=
-Dz_0.
\label{eq:entry-secondary-reduced}
\end{equation}
The normal variational equation in the ramified radial direction is
\begin{equation}
\frac{ds_0}{d\sigma_0}
=
\frac{D}{q}s_0.
\label{eq:entry-secondary-radial}
\end{equation}

Thus, the $s_0$-direction is attracting, while the $q_0$- and
$z_0$-directions are repelling in forward desingularized time. The
variable $h_0$ is a center variable for the reduced boundary system.
\end{lemma}

\begin{proof}
On $s_0=0$, with $W=W_L$ and $\xi=c$, with $W=W_L$ and $\xi=c$, the
reduced entry equations are
\begin{equation}
\frac{d\bar\rho_0}{d\tau}
=
-D\bar\varepsilon_0^2,
\qquad
\frac{d\bar\varepsilon_0}{d\tau}
=
-\frac{D\bar\varepsilon_0^3}{\bar\rho_0},
\qquad
\frac{dz_0}{d\tau}
=
-\frac{D\bar\varepsilon_0^2}{\bar\rho_0}z_0.
\label{eq:entry-endpoint-starting-system}
\end{equation}
Substituting $\bar\varepsilon_0=q_0,
\ \bar\rho_0=h_0q_0$ into the second equation gives
\begin{equation}
\frac{dq_0}{d\tau}
=-\frac{Dq_0^2}{h_0}.
\label{eq:entry-q0-tau}
\end{equation}
Differentiating $\bar\rho_0=h_0q_0$ yields $q_0\frac{dh_0}{d\tau}
+h_0\frac{dq_0}{d\tau}
=-Dq_0^2.$ Using \eqref{eq:entry-q0-tau}, we obtain
$\frac{dh_0}{d\tau}=0$ for $q_0>0$, and this equation extends smoothly to $q_0=0$.
Moreover, $\frac{dz_0}{d\tau}=-\frac{Dq_0}{h_0}z_0.$ Multiplication by $h_0/q_0$ therefore gives $\frac{dq_0}{d\sigma_0}=-Dq_0,
\ \frac{dh_0}{d\sigma_0}
=0,\ \frac{dz_0}{d\sigma_0}
=-Dz_0.$

By Lemma~\ref{lem:entry-regularity}, the ramified radial equation is
$\frac{ds_0}{d\tau}
=\frac1q s_0\widetilde\Phi_1.$ Along $s_0=0$,
$\widetilde\Phi_1
=\frac{D\bar\varepsilon_0^2}{\bar\rho_0}
=\frac{Dq_0}{h_0}.$ Consequently,
$\frac{ds_0}{d\sigma_0}
=\frac{h_0}{q_0}
\frac1q s_0
\frac{Dq_0}{h_0}
=\frac{D}{q}s_0.$ Since $D<0$, the $s_0$-direction is attracting, whereas the $q_0$- and $z_0$-directions have eigenvalue $-D>0$ and are
repelling. The $h_0$-direction has a zero eigenvalue.
\end{proof}
The resolution of the exit endpoint is similar, and therefore we omit the analogous details.
\begin{remark}
\label{rem:entry-secondary-trace}

The secondary blow-up replaces the single endpoint
$P_L^{\mathrm{in}}$ by the exceptional family
\begin{equation}
\mathcal E_L
=
\left\{
r_0=0,\;
q_0=0,\;
z_0=0,\;
h_0>0
\right\},
\end{equation}
with $W=W_L$ and $\xi=c$. The reduced outer orbit satisfies
$\bar\rho_0=\bar\varepsilon_0=0$ and therefore does not, by itself,
select a particular value of $h_0$.

Since
\begin{equation}
h_0
=
\frac{\bar\rho_0}{\bar\varepsilon_0}
=
\sqrt{\varepsilon}\,\rho,
\label{eq:h0-physical}
\end{equation}
the trace of the perturbed outer unstable manifold on
$\mathcal E_L$ can be identified only after determining the limiting
behavior of $\sqrt{\varepsilon}\rho$ in the entry scaling region.
\end{remark}
Restricting to $r_0=0$ gives the reduced system
\begin{equation}
    \begin{cases}
    \displaystyle
        \frac{d\overline{\rho}_0}{d\tau} = -\overline{\eps}_0^2 \big(\xi w_1 - w_2 \big) \\[1.5ex]\displaystyle
        \frac{d\overline{v}_0}{d\tau} = -\frac{\big(\xi w_1 - w_2 \big)\overline{v}_0\overline{\eps}_0^2}{\overline{\rho}_0} + \frac{\big(\frac{k_3}{k_1}w_1 - w_3 \big)\overline{\eps}_0^2}{\overline{\rho}_0} \\[1.5ex]\displaystyle
        \frac{d\overline{\eps}_0}{d\tau} = -\frac{\big(\xi w_1 - w_2 \big)\overline{\eps}_0^3}{\overline{\rho}_0} \\[1.5ex]\displaystyle
        \frac{d r_0}{d\tau}=\frac{d\xi}{d\tau} = \frac{du}{d\tau} = \frac{dw_1}{d\tau} = \frac{dw_2}{d\tau} = \frac{dw_3}{d\tau} = 0
    \end{cases}
\end{equation}
To analyze the distinguished reduced flow in the entry coordinates,
we now restrict to $W=W_L,
\ \xi=c.$ Therefore, $D=cw_{1,L}-w_{2,L}<0,
\ \Gamma=v_\delta w_{1,L}-w_{3,L}$ are constant along the reduced entry orbit.

On $r_0=0$, these quantities are constant, and the reduced system can
be written as
\begin{equation}
\begin{cases}
\displaystyle
\frac{d\bar\rho_0}{d\tau}
=
-D\bar\varepsilon_0^2,
\\[1.2ex]
\displaystyle
\frac{d\bar v_0}{d\tau}
=
\frac{\bar\varepsilon_0^2}{\bar\rho_0}
\left(\Gamma-D\bar v_0\right),
\\[1.2ex]
\displaystyle
\frac{d\bar\varepsilon_0}{d\tau}
=
-\frac{D\bar\varepsilon_0^3}{\bar\rho_0},
\\[1.2ex]
\displaystyle
\frac{dr_0}{d\tau}
=
\frac{d\xi}{d\tau}
=
\frac{du}{d\tau}
=
\frac{dw_1}{d\tau}
=
\frac{dw_2}{d\tau}
=
\frac{dw_3}{d\tau}
=
0.
\end{cases}
\label{eq:chart1_reduced}
\end{equation}
The invariant relation $\Gamma\bar u-D\bar v=0$ takes a particularly simple form in the entry region. Since $\bar u>0$ and the entry coordinates are obtained by dividing all spherical variables
by $\bar u$, one has $\bar v_0=\frac{\bar v}{\bar u}.$ Consequently, the reduced invariant relation becomes $\Gamma-D\bar v_0=0.$ Define
$z_0:=\Gamma-D\bar v_0,$ since $D$ and $\Gamma$ are constant along the distinguished reduced
entry orbit, system
\eqref{eq:chart1_reduced} gives
\begin{align}
\frac{dz_0}{d\tau}
=
-D\frac{d\bar v_0}{d\tau}
=
-\frac{D\bar\varepsilon_0^2}{\bar\rho_0}z_0.
\label{eq:z0_reduced}
\end{align}
Thus,
\begin{equation}
\mathcal N_1
=
\left\{
r_0=0,\;
z_0=0
\right\}
=
\left\{
r_0=0,\;
\bar v_0=\frac{\Gamma}{D},
\right\}
\label{eq:N1}
\end{equation}
is invariant under the reduced the entry region system.

\begin{lemma}[Normal hyperbolicity in the entry region]
\label{lem:chart1-hyperbolicity}

Let $\alpha=\frac{p}{q}\in\mathbb Q\cap(0,1],
\ \gcd(p,q)=1,$ and let $K_1$ be a compact subset of $\mathcal N_1$ on which $\bar\rho_0>0,
\
\bar\varepsilon_0>0,
\
D<0.$ After introducing $r_0=s_0^q$, the manifold $\mathcal N_1$ is
uniformly normally hyperbolic on $K_1$. The $s_0$-direction is
attracting and the $z_0$-direction is repelling.
\end{lemma}

\begin{proof}
On $\mathcal N_1$, one has $\widetilde\Phi_1
=\frac{D\bar\varepsilon_0^2}{\bar\rho_0}$ at $s_0=0$. Therefore, linearization of
\eqref{eq:s0-entry} in the $s_0$-direction gives $\lambda_{s_0}
=\frac{D\bar\varepsilon_0^2}
{q\bar\rho_0}
<0.$ The transverse variable $z_0$ satisfies
$\lambda_{z_0}
=-\frac{D\bar\varepsilon_0^2}
{\bar\rho_0}
>0.$ Since $K_1$ is compact and
$\bar\rho_0$, $\bar\varepsilon_0$, and $-D$ are bounded away from
zero on $K_1$, both normal eigenvalues are uniformly bounded away
from zero. Hence $\mathcal N_1$ is uniformly normally hyperbolic,
with an attracting $s_0$-direction and a repelling $z_0$-direction.
\end{proof}
The directional components of the incoming endpoint are
\begin{equation}
(\bar u_{\mathrm{in}},\bar v_{\mathrm{in}})
=
\left(
-\frac{D}{R},
-\frac{\Gamma}{R}
\right).
\label{eq:Pin-direction}
\end{equation}

Since $D<0$, the $\bar u$-component of $P_{\mathrm{in}}$ is positive,
and therefore $P_{\mathrm{in}}$ is represented in the entry region. Moreover,
\begin{equation}
\frac{\bar v_{\mathrm{in}}}{\bar u_{\mathrm{in}}}
=
\frac{\Gamma}{D}.
\end{equation}
Thus, the incoming reduced orbit lies on $\mathcal N_1$.

On $\mathcal N_1$, system \eqref{eq:chart1_reduced} reduces to
\begin{equation}
\frac{d\bar\rho_0}{d\tau}
=
-D\bar\varepsilon_0^2,
\qquad
\frac{d\bar\varepsilon_0}{d\tau}
=
-\frac{D\bar\varepsilon_0^3}{\bar\rho_0},
\qquad
\bar v_0=\frac{\Gamma}{D}.
\label{eq:chart1_on_N1}
\end{equation}
Since $D<0$, it follows that
\begin{equation}
\frac{d\bar\rho_0}{d\tau}>0,
\qquad
\frac{d\bar\varepsilon_0}{d\tau}>0
\end{equation}
whenever $\bar\rho_0>0$ and $\bar\varepsilon_0>0$.

Furthermore,
\begin{align}
\frac{d}{d\tau}
\left(
\frac{\bar\rho_0}{\bar\varepsilon_0}
\right)
&=\frac{
\bar\varepsilon_0\frac{d\bar\rho_0}{d\tau}
-\bar\rho_0\frac{d\bar\varepsilon_0}{d\tau}
}{\bar\varepsilon_0^2}=0.
\end{align}
Hence the ratio $\bar\rho_0/\bar\varepsilon_0$ remains constant along
the reduced orbit in the entry region.

Fix $R>1$ sufficiently large and define the entry region overlap section
\begin{equation}
\Sigma_{\mathrm{in}}^R
=
\left\{
\bar\varepsilon_0=\frac{1}{R}
\right\}.
\label{eq:Sigma_in_R}
\end{equation}
For each fixed value $h_0>0$, the reduced equations on
$\mathcal N_1$ admits an orbit satisfying
$\frac{\bar\rho_0}{\bar\varepsilon_0}=h_0.$ Since
$\frac{d\bar\varepsilon_0}{d\tau}>0,$ each such orbit intersects the overlap section $\Sigma_{\mathrm{in}}^R
=\left\{
\bar\varepsilon_0=\frac1R
\right\}$ transversely at a unique point, which we denote by
$Q_{\mathrm{in}}^0(h_0)$. At this point,
\begin{equation}
\bar\rho_0=\frac{h_0}{R},
\qquad
\bar v_0=\frac{\Gamma}{D}.
\label{eq:entry-section-h0}
\end{equation}
Thus, the reduced entry passage determines a one-parameter family of
points $Q_{\mathrm{in}}^0(h_0)\subset\Sigma_{\mathrm{in}}^R,$
where $h_0=\frac{\bar\rho_0}{\bar\varepsilon_0}
=\sqrt{\varepsilon}\,\rho$ is constant along each reduced trajectory. Reduced dynamics alone does not determine the parameter $h_0$.
Rather, the distinguished member of this family corresponding to the
viscous profile is identified by the global matching argument developed
in the subsequent subsections.
\subsubsection{The positive-$\varepsilon$ passage}
The positive-$\varepsilon$ and positive-density charts cover
overlapping portions of the central blown-up region and play
complementary roles. The distinguished singular boundary orbit
responsible for the order-one change from $W_L$ to $W_R$ is analyzed
in the positive-density chart $K_\rho$. The positive-$\varepsilon$
chart $K_\varepsilon$ is retained because, for
$\varepsilon>0$, it provides a regular finite-time passage through
the region where the $\varepsilon$-direction dominates and where
$u-\xi$ changes sign. Thus, $K_\varepsilon$ is a transition chart; it
does not replace the positive-density middle orbit.

In this chart $\bar{\varepsilon}>0$, and the variable $\bar{u}_\varepsilon$ (see below) is no longer constrained to lie on the unit sphere. The chart coordinates are
\begin{equation}
        u - \xi = r_\varepsilon\overline{u}_\varepsilon \ \ \
        \rho_0 = r_\varepsilon\overline{\rho}_\varepsilon \ \ \ 
        v - \frac{k_3}{k_1} = r_\varepsilon\overline{v}_\varepsilon \ \ \ 
        \epsilon = r_\varepsilon^2
\end{equation}
Similarly, rewriting system \eqref{GSPT_eps_sys} in these coordinates and applying the chain rule yields the transformed system
\begin{equation}
\begin{cases}
\displaystyle
\frac{d\bar\rho_\varepsilon}{d\tau}
=
r_\varepsilon
\left(
\bar\rho_\varepsilon\bar u_\varepsilon-w_1
\right),
\\[1.5ex]
\displaystyle
\frac{d\bar u_\varepsilon}{d\tau}
=
\frac{
(\xi+r_\varepsilon\bar u_\varepsilon)w_1-w_2
}{
\bar\rho_\varepsilon
}
-
\frac{
r_\varepsilon^\alpha
A\!\left(
\frac{k_3}{k_1}+r_\varepsilon\bar v_\varepsilon
\right)
}{
\bar\rho_\varepsilon^{\alpha+1}
}
-r_\varepsilon,
\\[1.5ex]
\displaystyle
\frac{d\bar v_\varepsilon}{d\tau}
=
\frac{
\left(
\frac{k_3}{k_1}
+r_\varepsilon\bar v_\varepsilon
\right)w_1-w_3
}{
\bar\rho_\varepsilon
},
\\[1.5ex]
\displaystyle
\frac{dw_1}{d\tau}
=
-r_\varepsilon\bar\rho_\varepsilon,
\\[1.5ex]
\displaystyle
\frac{dw_2}{d\tau}
=
-r_\varepsilon\bar\rho_\varepsilon
\left(
\xi+r_\varepsilon\bar u_\varepsilon
\right),
\\[1.5ex]
\displaystyle
\frac{dw_3}{d\tau}
=
-r_\varepsilon\bar\rho_\varepsilon
\left(
\frac{k_3}{k_1}
+r_\varepsilon\bar v_\varepsilon
\right),
\\[1.5ex]
\displaystyle
\frac{d\xi}{d\tau}=r_\varepsilon^2,
\qquad
\frac{dr_\varepsilon}{d\tau}=0.
\end{cases}
\label{eq:chart3_full}
\end{equation}
The transition map from the entry region to the positive-$\varepsilon$ region is
\begin{equation}
r_\varepsilon=r_0\bar\varepsilon_0,
\qquad
\bar\rho_\varepsilon
=
\frac{\bar\rho_0}{\bar\varepsilon_0},
\qquad
\bar u_\varepsilon
=
\frac{1}{\bar\varepsilon_0},
\qquad
\bar v_\varepsilon
=
\frac{\bar v_0}{\bar\varepsilon_0}.
\label{eq:transition_1_to_3}
\end{equation}
Therefore, the section $\Sigma_{\mathrm{in}}^R$ corresponds in
the positive-$\varepsilon$ region to $\bar u_\varepsilon=R.$ At $Q_{\mathrm{in}}^0(h_0)$, the transition relations give
\begin{equation}
\bar u_\varepsilon=R,
\qquad
\bar\rho_\varepsilon=h_0,
\qquad
\bar v_\varepsilon
=
R\frac{\Gamma}{D}.
\label{eq:entry-data-epsilon-chart}
\end{equation}
Thus the reduced entry passage determines a one-parameter family of
initial data for the positive-$\varepsilon$ chart, parametrized by
$h_0>0$.
 
The exit section of the entry region lies in the overlap with the
positive-$\varepsilon$ region. For each fixed $h_0>0$, the smooth
coordinate transformation carries the point
$Q_{\mathrm{in}}^0(h_0)$ to the corresponding initial point in
$K_\varepsilon$. Hence the reduced entry passage determines a
one-parameter family of trajectories in the positive-$\varepsilon$
chart.

The positive-$\varepsilon$ chart describes the central passage through
the region in which the $\varepsilon$-direction is dominant. Its role
is different from that of the positive-density chart
$K_\rho$. The $O(1)$ change from $W_L$ to $W_R$ occurs along the
singular boundary orbit constructed in
Proposition~\ref{prop:singular-middle-connection}. By contrast, the flow in
$K_\varepsilon$ resolves the change in sign of $u-\xi$ while the
variables $W$ vary only by a small amount over a bounded passage time.

Setting $r_\varepsilon=0$ in
\eqref{eq:chart3_full} gives
\begin{equation}
\frac{d\bar u_\varepsilon}{d\tau}
=\frac{D}{\bar\rho_\varepsilon},
\ \ \ 
\frac{d\bar v_\varepsilon}{d\tau}
=\frac{\Gamma}{\bar\rho_\varepsilon},
\ \ \  
\frac{d\bar\rho_\varepsilon}{d\tau}
=0,
\ \ \
\frac{dk}{d\tau}
=\frac{dD}{d\tau}
=\frac{d\Gamma}{d\tau}
=\frac{d\xi}{d\tau}
=0.
\label{eq:epsilon-chart-reduced}
\end{equation}
Thus, at $r_\varepsilon=0$, the quantities
$\bar\rho_\varepsilon$, $k$, $D$, $\Gamma$, and $\xi$ remain
constant. In particular, this reduced chart does not itself produce
the change from $W_L$ to $W_R$.

If $D<0$ and $\bar\rho_\varepsilon>0$, then
\begin{equation}
\frac{d\bar u_\varepsilon}{d\tau}
=
\frac{D}{\bar\rho_\varepsilon}<0.
\end{equation}
Consequently, the reduced flow crosses the positive-$\varepsilon$
chart monotonically from positive to negative values of
$\bar u_\varepsilon$.

\subsubsection{The positive-density middle passage}
\label{subsubsec:rho-chart-representation}

The positive-density chart $K_\rho$ is defined by
\begin{equation}
\rho_0=r_\rho,
\qquad
u-\xi=r_\rho\bar u_\rho,
\qquad
v-v_\delta=r_\rho\bar v_\rho,
\qquad
\sqrt{\varepsilon}=r_\rho\bar\varepsilon_\rho.
\label{eq:rho-chart-coordinates}
\end{equation}

The entry and positive-density charts overlap directly. Equating
their coordinate representations gives
\begin{equation}
r_\rho=r_0\bar\rho_0,
\qquad
\bar u_\rho=\frac{1}{\bar\rho_0},
\qquad
\bar v_\rho=\frac{\bar v_0}{\bar\rho_0},
\qquad
\bar\varepsilon_\rho
=
\frac{\bar\varepsilon_0}{\bar\rho_0}.
\label{eq:entry-to-rho-transition}
\end{equation}
At the point $Q_{\mathrm{in}}^0(h_0)$, where $\bar\varepsilon_0=\frac1R,
\ \bar\rho_0=\frac{h_0}{R},
\ \bar v_0=\frac{\Gamma}{D},$ these relations become
\begin{equation}
\bar u_\rho=\frac{R}{h_0},
\qquad
\bar\varepsilon_\rho=\frac1{h_0},
\qquad
\bar v_\rho=\frac{R\Gamma}{Dh_0}.
\label{eq:entry-family-in-rho-chart}
\end{equation}
\subsubsection{Representation of the distinguished middle orbit in
the positive-density chart}

The distinguished inner scaling is related to the coordinates of
$K_\rho$ by
$r_\rho=\delta z^2,
\ \bar u_\rho=\frac{y}{z^3},
\ \bar v_\rho=\frac{h}{z^3},
\ \bar\varepsilon_\rho=\frac{\delta^{1/2}}{z^2}.$ Consequently, for every fixed $z>0$, the singular limit
$\delta\to0^+$ corresponds to
$r_\rho\to0,
\ \bar\varepsilon_\rho\to0.$ Thus the reduced orbit constructed in
Proposition~\ref{prop:singular-middle-connection} lies on the positive-density boundary
$r_\rho=0$, $\bar\varepsilon_\rho=0$ of the spherical blow-up.

The distinguished variables provide a regular parametrization of
this boundary orbit and determine the order-one change from $W_L$
to $W_R$. The chart $K_\rho$ is used only to identify this orbit
within the spherical blow-up atlas and to relate it to the
positive-$\varepsilon$ chart on their common domain. The construction and matching of the middle orbit are carried out in the distinguished variables.
\subsubsection{Weighted resolution of the central corner}
\label{subsubsec:weighted-central-corner}

The reduced middle orbit constructed in
Subsection~\ref{subsubsec:reduced-middle-connection}
describes the interior of the singular layer for fixed
$z>0$. To analyze the simultaneous limit
$z\to0,\ \delta\to0,$ we return to the desingularized middle system
\eqref{eq:full-middle-system}
and introduce the weighted corner blow-up.

The quotient terms become singular precisely when
$z\to0$ and $\delta\to0$ simultaneously.
To resolve this corner, we introduce an additional weighted
blow-up.

\paragraph{The weighted corner variable.}
To resolve the simultaneous limit
$z\to0$ and $\delta\to0$,
introduce the weighted variable $b=\frac{\delta}{z}.$ Equivalently,
\begin{equation}
\label{eq:corner-delta-zb}
\delta=zb.
\end{equation}

For each fixed $\delta>0$, trajectories satisfy
$zb=\delta.$ Hence the singular limit $\delta=0$ consists of the two invariant faces
$\{z=0\}\quad\text{and}\quad
\{b=0\},$ which meet at the central corner $\{z=0,\;b=0\}.$

Substituting $\delta=zb$ into
\eqref{eq:full-middle-system} gives
\begin{equation}\label{eq:corner-system-s-time}
\begin{cases}\displaystyle
z'
&=
\frac12
\left(
y-bk
\right),
\\\displaystyle
y'
&=
\frac{1}{z}
\left[
D+\frac12y^2+\frac12bky
-
b^{2\alpha}
A\left(
v_\delta+bh
\right)
\right]
-z^2b,
\\\displaystyle
h'
&=
\frac{1}{z}
\left[
\Gamma+\frac12hy+\frac12bkh
\right],
\\\displaystyle
k'
&=
-z^2,
\\\displaystyle
D'
&=
z^2b
\left(
y+bk
\right),
\\\displaystyle
\Gamma'
&=
z^2bh,
\\\displaystyle
\xi'
&=
z^2b^2.
\end{cases}
\end{equation}

Differentiating the identity $\delta=zb,$ and using that $\delta$ is constant along trajectories gives
\begin{equation}
\label{eq:corner-product-derivative}
z'b+zb'=0.
\end{equation}
Using the first equation of \eqref{eq:corner-system-s-time}, we obtain
\begin{equation}
\label{eq:corner-b-s-time}
b'
=
-\frac{b}{2z}
\left(
y-bk
\right).
\end{equation}
The hypersurface $z=0$ is therefore still singular.
The vector field still contains a common factor $1/z$.  We therefore
introduce the desingularized time $\sigma$ by
\begin{equation}
\label{eq:corner-desingularized-time}
\frac{ds}{d\sigma}=z, \ \ \frac{d}{d\sigma}
=z\frac{d}{ds}.
\end{equation}
For $z>0$, this is multiplication of the vector field by a positive
function and hence does not change the oriented trajectories.

A dot will denote differentiation with respect to $\sigma$.  The
desingularized weighted corner system is
\begin{equation}
\label{eq:corner-desingularized-system}
\begin{cases}\displaystyle
\dot z
&=
\frac{z}{2}
\left(
y-bk
\right),
\\[0.8ex]\displaystyle
\dot b
&=
-\frac{b}{2}
\left(
y-bk
\right),
\\[0.8ex]\displaystyle
\dot y
&=
D+\frac12y^2+\frac12bky
-
b^{2\alpha}
A\left(
v_\delta+bh
\right)
-z^3b,
\\\displaystyle
\dot h
&=
\Gamma+\frac12hy+\frac12bkh,
\\
\dot k
&=
-z^3,
\\[0.8ex]\displaystyle
\dot D
&=
z^3b
\left(
y+bk
\right),
\\[0.8ex]\displaystyle
\dot\Gamma
&=
z^3bh,
\\[0.8ex]\displaystyle
\dot\xi
&=
z^3b^2.
\end{cases}
\end{equation}

\begin{lemma}[Invariant parameter levels]
\label{lem:corner-invariant-delta-levels}

For system~\eqref{eq:corner-desingularized-system}, the product
$zb$ is constant along the trajectories.  Consequently, every
hypersurface $\mathcal E_\delta
=\left\{
zb=\delta
\right\},
\
\delta\geq0,$ is invariant.
\end{lemma}

\begin{proof}
Using the first two equations of
\eqref{eq:corner-desingularized-system}, we find
\begin{align}
\frac{d}{d\sigma}(zb)=
\dot z\,b+z\dot b=
\frac{zb}{2}
\left(
y-bk
\right)
-\frac{zb}{2}
\left(
y-bk
\right)=0.
\end{align}
Therefore, $zb$ is constant.  By
\eqref{eq:corner-delta-zb}, its value is precisely $\delta$.
\end{proof}
\paragraph{Regularity of the weighted vector field.}
All terms in \eqref{eq:corner-desingularized-system} are smooth except
possibly
\begin{equation}
\label{eq:corner-pressure-term}
b^{2\alpha}
A\left(
v_\delta+bh
\right).
\end{equation}
For $\alpha=\frac12$, this term is $bA\left(
v_\delta+bh
\right),$ and for $\alpha=1$ it is $b^2A\left(
v_\delta+bh
\right).$ Thus, if $A$ is $C^\infty$ in a neighborhood of the relevant values
of $v$, then the corner vector field is $C^\infty$ for
$\alpha\in
\left\{
\frac12,1
\right\}.$ More generally, suppose that $\alpha\in\mathbb Q\cap(0,1]$ and write $2\alpha=\frac{m}{n},$ where $m,n\in\mathbb N$ are relatively prime.  Introduce a ramified
weighted coordinate $\beta\geq0$ by $b=\beta^n.$ Then
\begin{equation}
\label{eq:corner-pressure-beta}
b^{2\alpha}
=
\left(
\beta^n
\right)^{m/n}
=
\beta^m,
\end{equation}
so the pressure term becomes
\begin{equation}
\label{eq:corner-pressure-beta-smooth}
\beta^m
A\left(
v_\delta+\beta^nh
\right),
\end{equation}
which is smooth.

Differentiating $b=\beta^n$ and using the $b$-equation in
\eqref{eq:corner-desingularized-system} gives
\begin{align}
n\beta^{n-1}\dot\beta
&=
-\frac{\beta^n}{2}
\left(
y-\beta^nk
\right).
\end{align}
For $\beta>0$, division by $n\beta^{n-1}$ gives
\begin{equation}
\label{eq:corner-beta-equation}
\dot\beta
=
-\frac{\beta}{2n}
\left(
y-\beta^nk
\right).
\end{equation}
The right-hand side extends smoothly to $\beta=0$.

Thus, in the variables $(z,\beta,y,h,k,D,\Gamma,\xi),$
the desingularized system is
\begin{equation}
\label{eq:corner-rational-system}
\begin{cases}\displaystyle
\dot z
&=
\frac{z}{2}
\left(
y-\beta^nk
\right),
\\[0.8ex]\displaystyle
\dot\beta
&=
-\frac{\beta}{2n}
\left(
y-\beta^nk
\right),
\\[0.8ex]\displaystyle
\dot y
&=
D+\frac12y^2+\frac12\beta^nky
-
\beta^m
A\left(
v_\delta+\beta^nh
\right)
-z^3\beta^n,
\\[0.8ex]\displaystyle
\dot h
&=
\Gamma+\frac12hy+\frac12\beta^nkh,
\\[0.8ex]\displaystyle
\dot k
&=
-z^3,
\\
\dot D
&=
z^3\beta^n
\left(
y+\beta^nk
\right),
\\[0.8ex]\displaystyle
\dot\Gamma
&=
z^3\beta^nh,
\\[0.8ex]\displaystyle
\dot\xi
&=
z^3\beta^{2n}.
\end{cases}
\end{equation}
The fixed-parameter relation is now
\begin{equation}
\label{eq:corner-rational-delta-level}
z\beta^n=\delta.
\end{equation}
The weighted variables are related to the original variables through the
blow-down map:
\begin{equation}
\label{eq:corner-blowdown}
\rho_0=z^3\beta^n,\ \ \
u-\xi=\beta^n y,\ \ \ 
v-v_\delta=\beta^n h,\ \ \ 
w_1= k,\ \ \ 
w_2 = \xi k-D,\ \ \
w_3 = v_\delta k-\Gamma,\ \ \ 
\varepsilon= z^3\beta^{3n}.
\end{equation}
Dividing the relations for $\rho_0$ and $\varepsilon$ gives
\begin{equation}
\label{eq:weighted-original-identities}
\rho=\beta^{-2n}.
\end{equation}
Moreover, since $z\beta^n=\delta$, one has
\begin{equation}
\label{eq:weighted-original-identities-two}
z=\delta\sqrt{\rho},
\qquad
y=\sqrt{\rho}(u-\xi),
\qquad
h=\sqrt{\rho}(v-v_\delta).
\end{equation}
Thus, on the singular level $\delta=0$, the outer regime is represented
by the face $z=0$, while the singular middle orbit is represented by the
face $\beta=0$.

\begin{remark}
\label{rem:corner-ramified-coordinate}
For $n>1$, the map $\beta\mapsto b=\beta^n$ is not a local
diffeomorphism at $\beta=0$.  It is used here as a ramified
desingularization of the nonnegative half-space $b\geq0$.  For
$\beta>0$, it is one-to-one and maps trajectories of
\eqref{eq:corner-rational-system} to trajectories of
\eqref{eq:corner-desingularized-system}.  The boundary $\beta=0$
resolves the loss of differentiability of $b^{2\alpha}$ at $b=0$.
\end{remark}

\begin{prop}[Smoothness of the rational corner system]
\label{prop:rational-corner-smoothness}

Assume that $A$ is $C^\infty$ in a neighborhood of the relevant
values of $v$.  Let $\alpha\in\mathbb Q\cap(0,1]$ and write
$2\alpha=m/n$ in lowest terms.  Then the vector field
\eqref{eq:corner-rational-system} extends to a $C^\infty$ vector
field on a neighborhood of
\begin{equation}
\label{eq:corner-rational-boundary}
\{z=0,\ \beta=0\}.
\end{equation}

For $\alpha=\frac12$ and $\alpha=1$, one may take $n=1$, so
$\beta=b$ and no ramified change of variables is needed.

\end{prop}

\begin{proof}

All terms in \eqref{eq:corner-rational-system} are polynomial in
$z$, $\beta$, $y$, $h$, $k$, $D$, $\Gamma$, and $\xi$, except for
the composition
\begin{equation}
\beta^m
A\left(
v_\delta+\beta^nh
\right).
\end{equation}
Since $m,n\in\mathbb N$ and $A$ is $C^\infty$, this composition is
$C^\infty$.  Hence the complete vector field extends smoothly to
$z=\beta=0$.
\end{proof}

\paragraph{Corner-equilibrium manifolds.}

We now work with the rational system
\eqref{eq:corner-rational-system}.  The cases
$\alpha=\frac12$ and $\alpha=1$ are included by taking $n=1$.

On the corner $z=0,\
\beta=0,$ system~\eqref{eq:corner-rational-system} reduces to
\begin{equation}
\label{eq:corner-boundary-reduced}
\dot y=D+\frac12y^2,
\ \ \
\dot h=
\Gamma+\frac12hy,
\ \ \
\dot k=\dot D
=\dot\Gamma
=\dot\xi=0.
\end{equation}
Since $D<0$, define $a(D)=\sqrt{-2D}.$ The equilibrium conditions for
\eqref{eq:corner-boundary-reduced} are
\begin{equation}
\label{eq:corner-equilibrium-conditions}
D+\frac12y^2=0,
\qquad
\Gamma+\frac12hy=0.
\end{equation}
The first equation gives
\begin{equation}
\label{eq:corner-y-plus-minus}
y=\pm a(D).
\end{equation}

If $y=a(D)$, then
\begin{align}
h=-\frac{2\Gamma}{a(D)}=
\frac{\Gamma}{D}a(D),
\label{eq:corner-h-plus}
\end{align}
where we used $D=-a(D)^2/2$.  If $y=-a(D)$, then
\begin{align}
h=\frac{2\Gamma}{a(D)}=
-\frac{\Gamma}{D}a(D).
\label{eq:corner-h-minus}
\end{align}
Define
\begin{equation}
\label{eq:corner-P-plus}
\mathcal P_+
=
\left\{
\begin{aligned}
&z=0,\quad \beta=0,\quad
y=a(D),\\
&h=\frac{\Gamma}{D}a(D),\quad D<0
\end{aligned}
\right\}
\end{equation}
and
\begin{equation}
\label{eq:corner-P-minus}
\mathcal P_-
=
\left\{
\begin{aligned}
&z=0,\quad \beta=0,\quad
y=-a(D),\\
&h=-\frac{\Gamma}{D}a(D),\quad D<0.
\end{aligned}
\right\}.
\end{equation}
The coordinates
\begin{equation}
\label{eq:corner-center-coordinates}
(k,D,\Gamma,\xi)
\end{equation}
are free on each set.  Therefore,
\begin{equation}
\label{eq:corner-P-dimensions}
\dim\mathcal P_+
=
\dim\mathcal P_-
=
4.
\end{equation}

\begin{prop}[Normal hyperbolicity of the corner manifolds]
\label{prop:corner-normal-hyperbolicity}

Let $\alpha\in\mathbb Q\cap(0,1]$, write
$2\alpha=m/n$ in lowest terms, and assume that $D$ ranges in a
compact interval contained in $(-\infty,0)$.  Then
$\mathcal P_+$ and $\mathcal P_-$ are normally hyperbolic manifolds
of equilibria for system~\eqref{eq:corner-rational-system}.

At $\mathcal P_+$, the four normal eigenvalues are
$-\frac{a}{2n},\ \frac{a}{2},
\ \  \frac{a}{2}, \ \ a,$ where $a=\sqrt{-2D}$.  Thus $\mathcal P_+$ has one stable and three unstable normal directions. At $\mathcal P_-$, the four normal eigenvalues are $\frac{a}{2n},
\ \ -\frac{a}{2},
\ \ -\frac{a}{2},
\ \ -a.$ Thus $\mathcal P_-$ has three stable and one unstable normal
directions.
\end{prop}
\begin{proof}
We first consider $\mathcal P_+$.  Fix a point
\begin{equation}
\label{eq:corner-plus-point}
p_+
=
\left(
0,0,a,h_+,k,D,\Gamma,\xi
\right)
\in\mathcal P_+,
\end{equation}
where
\begin{equation}
\label{eq:corner-h-plus-definition}
a=\sqrt{-2D},
\
h_+=\frac{\Gamma}{D}a.
\end{equation}

The equations for $z$ and $\beta$ in
\eqref{eq:corner-rational-system} contain the factors $z$ and
$\beta$, respectively.  Their linearizations at $p_+$ are
$\dot z=
\frac{a}{2}z$ and $\dot\beta
=-\frac{a}{2n}\beta.$ Thus the corresponding eigenvalues are
$\lambda_z^+=
\frac{a}{2},
\ \ \lambda_\beta^+
=-\frac{a}{2n}.$ To determine the complete normal spectrum, order the normal variables
as $(\beta,z,y,h).$ At $p_+$, the linearized normal system has block lower-triangular form
\begin{equation}
\begin{pmatrix}
\dot{\widetilde\beta}\\
\dot{\widetilde z}\\
\dot{\widetilde y}\\
\dot{\widetilde h}
\end{pmatrix}
=
\begin{pmatrix}
-\frac{a}{2n} & 0 & 0 & 0\\[1ex]
0 & \frac{a}{2} & 0 & 0\\[1ex]
c_y^+ & 0 & a & 0\\[1ex]
c_h^+ & 0 & \frac{h_+}{2} & \frac{a}{2}
\end{pmatrix}
\begin{pmatrix}
\widetilde\beta\\
\widetilde z\\
\widetilde y\\
\widetilde h
\end{pmatrix},
\label{eq:corner-plus-full-normal-linearization}
\end{equation}
where the coefficients $c_y^+$ and $c_h^+$ arise only when the terms
containing $\beta^n$ or $\beta^m$ have nonzero first derivatives at
$\beta=0$. Their precise values do not affect the spectrum.

Because the matrix is lower triangular, its eigenvalues are
$-\frac{a}{2n},
\ \frac{a}{2},
\ a,
\ \frac{a}{2}.$ The remaining four eigenvalues are zero and correspond to vectors
tangent to the equilibrium manifold $\mathcal P_+$. Therefore
$\mathcal P_+$ is normally hyperbolic, with one stable and three
unstable normal directions.

At a point $p_-\in\mathcal P_-$, the same ordering gives a block
lower-triangular normal linearization whose diagonal entries are
$\frac{a}{2n},
\ -\frac{a}{2},
\ -a,
\ -\frac{a}{2}.$ Hence $\mathcal P_-$ is normally hyperbolic with three stable and one unstable normal directions.
\end{proof}

\begin{corollary}[Dimensions of the corner invariant manifolds]
\label{cor:corner-manifold-dimensions}
The stable and unstable manifolds of the corner-equilibrium
manifolds have dimensions
\begin{equation}
\label{eq:corner-plus-manifold-dimensions}
\dim W^s(\mathcal P_+)=5,
\qquad
\dim W^u(\mathcal P_+)=7,
\end{equation}
and
\begin{equation}
\label{eq:corner-minus-manifold-dimensions}
\dim W^s(\mathcal P_-)=7,
\qquad
\dim W^u(\mathcal P_-)=5.
\end{equation}

\end{corollary}

\begin{proof}

Each of $\mathcal P_+$ and $\mathcal P_-$ has dimension four.  At
$\mathcal P_+$ there is one stable and three unstable normal
directions, while at $\mathcal P_-$ there are three stable and one
unstable normal directions.  Adding the tangent dimensions gives
the result.
\end{proof}

\paragraph{Invariant hypersurface for the Corner Lemma.}

Define $\mathcal S
=\left\{
z=0
\right\}.$ The first equation of \eqref{eq:corner-rational-system} shows that
$\mathcal S$ is invariant.  Moreover, $
\mathcal P_+\subset\mathcal S,
\quad
\mathcal P_-\subset\mathcal S.$ At $\mathcal P_+$, the $z$-direction is unstable and transverse to
$\mathcal S$.  Therefore, the restriction of the linearization to
$\mathcal S$ removes exactly one unstable eigenvalue.  In the full
eight-dimensional phase space, the normal spectrum of
$\mathcal P_+$ consists of one stable and three unstable
eigenvalues; in $\mathcal S$, it consists of one stable and two
unstable eigenvalues.

Furthermore, the stable normal eigenspace of $\mathcal P_+$ is
contained in $T\mathcal S$. Indeed, the stable eigenvalue is
$-a/(2n)$, and the associated eigenvector may contain
$(y,h)$-components when $n=1$ or $m=1$, but it has no $z$-component.
The center bundle is also tangent to $\mathcal S$. Since $\mathcal S$
is invariant, it follows that $W^s(\mathcal P_+)\subset\mathcal S$
locally. Consequently, $\mathcal P_+$ and $\mathcal S$ satisfy the
corresponding spectral and invariance hypotheses of the Corner Lemma.

At $\mathcal P_-$, the unstable normal eigenspace is contained in
$T\mathcal S$, while the $z$-direction is stable and transverse to
$\mathcal S$. After reversing time, this gives the corresponding
outgoing corner geometry.

\begin{remark}
\label{rem:corner-irrational-alpha}

When $\alpha\notin\mathbb Q$, no finite ramified substitution
$b=\beta^n$ converts $b^{2\alpha}$ into an integer power.  The
corner geometry and the formal normal eigenvalues remain unchanged,
because the pressure term vanishes on $\beta=0$, but the smoothness
hypothesis of the standard Corner Lemma is not available.  The
present persistence argument is therefore restricted to $\alpha\in\mathbb Q\cap(0,1].$ In particular, the cases $\alpha=\frac12$ and $\alpha=1$ require no
ramified coordinate and are covered directly.
\end{remark}
\subsubsection{The Corner Lemma and the matching strategy}
\label{subsubsec:corner-lemma}
\begin{lemma}[Schecter's Corner Lemma {\cite[Theorem~5.1]{Sc}}]
\label{lem:corner-lemma}

Let $X$ be a sufficiently smooth vector field on an open subset of
$\mathbb R^N$, and let $\mathcal P$ be a normally hyperbolic manifold
of equilibria. Suppose that $\mathcal S$ is a codimension-one
invariant manifold satisfying $\mathcal P\subset\mathcal S,
\ \  W^s(\mathcal P)\subset\mathcal S.$ Assume that the unstable bundle of $\mathcal P$ contains a one-dimensional subbundle transverse to $\mathcal S$.

Let $\mathcal N$ be a manifold transverse to $W^s(\mathcal P)$ along
a submanifold $\mathcal Q\subset W^s(\mathcal P)$, and assume that
$\mathcal N$ is transverse to the vector field. Let $z$ be a defining coordinate for $\mathcal S$, so that $\mathcal S=\{z=0\}.$
For sufficiently small $\eta>0$, let
\begin{equation}
\mathcal N_\eta=\mathcal N\cap\{z=\eta\}.
\end{equation}

Then the forward images of $\mathcal N_\eta$ through a fixed
neighborhood of $\mathcal P$ converge in $C^1$, as $\eta\to0^+$, to
the corresponding portion of $W^u(\mathcal Q)$.

The analogous statement holds in backward time with stable and
unstable manifolds interchanged.
\end{lemma}
\begin{remark}
The precise regularity and coordinate hypotheses are those of
Schecter's Corner Lemma. In the present application, the smooth
vector field is \eqref{eq:corner-rational-system}, the invariant
hypersurface is $\mathcal S=\{z=0\}$, and the corner manifolds are
$\mathcal P_+$ and $\mathcal P_-$.
\end{remark}

\subsubsection{Outer-to-corner transversality}
\label{subsubsec:outer-corner-transversality}

We now verify the transversality hypotheses required for the
application of the Corner Lemma. The point that must be checked is
the tangent space of the lift of the extended outer unstable manifold
in the weighted corner coordinates.

For the left state, define
\begin{equation}
\label{eq:left-outer-W-functions}
\begin{aligned}
k_L(\xi)
&=
\bigl(F(H_L)-\xi H_L\bigr)_1,
\\
D_L(\xi)
&=
\xi k_L(\xi)
-
\bigl(F(H_L)-\xi H_L\bigr)_2,
\\
\Gamma_L(\xi)
&=
v_\delta k_L(\xi)
-
\bigl(F(H_L)-\xi H_L\bigr)_3.
\end{aligned}
\end{equation}
At $\varepsilon=0$, the reduced unstable manifold associated with the
left state is
\begin{equation}
\label{eq:left-reduced-manifold-recalled}
\mathcal W_0^u(S_0(H_L))
=
\left\{
\left(
H,F(H_L)-\xi H_L,\xi,0
\right):
H\in\Omega_\xi
\right\}.
\end{equation}
Thus, before the weighted blow-up, $H$ and $\xi$ are independent local
coordinates on the reduced unstable manifold.

In the weighted corner chart, the blow-down relations are
\begin{equation}
\label{eq:left-weighted-blowdown-recalled}
\begin{aligned}
\rho=\beta^{-2n},\quad u&=\xi+\beta^ny, \quad v=v_\delta+\beta^nh, \\
\varepsilon=z^3\beta^{3n},\quad k&=w_1,\quad D=\xi w_1-w_2, \quad \Gamma=v_\delta w_1-w_3.
\end{aligned}
\end{equation}
For every $\beta>0$, the transformation
\begin{equation}
\label{eq:left-H-weighted-coordinate-map}
(\beta,y,h,\xi)
\longmapsto
H
=
\begin{pmatrix}
\beta^{-2n}
\\[0.3em]
\beta^{-2n}\bigl(\xi+\beta^ny\bigr)
\\[0.3em]
\beta^{-2n}\bigl(v_\delta+\beta^nh\bigr)
\end{pmatrix}
\end{equation}
has nonsingular Jacobian with respect to $(\beta,y,h)$.
Indeed,
\begin{equation}
\label{eq:left-H-weighted-Jacobian}
\det
\frac{\partial(\rho,u,v)}
{\partial(\beta,y,h)}
=
-2n\beta^{-2n-1}\beta^n\beta^n
=
-\frac{2n}{\beta}
\neq0.
\end{equation}
Consequently, $(\beta,y,h,\xi)$ are genuine local coordinates on the
lift of the reduced unstable manifold for $\beta>0$.

We next include the perturbation direction. Since $\varepsilon=z^3\beta^{3n},$ Fenichel theory implies that the extended unstable manifold $\widehat{\mathcal W}_L^u$ is a $C^r$ manifold depending smoothly on $\varepsilon$. After composition with the smooth weighted blow-down transformation, it follows that, for each fixed $\beta>0$, the lifted manifold admits the local graph representation
\begin{equation}
\label{eq:left-extended-weighted-graph}
\begin{aligned}
k
&=
k_L(\xi)
+
z^3\beta^{3n}
K_L(z,\beta,y,h,\xi),
\\
D
&=
D_L(\xi)
+
z^3\beta^{3n}
\mathcal D_L(z,\beta,y,h,\xi),
\\
\Gamma
&=
\Gamma_L(\xi)
+
z^3\beta^{3n}
\mathcal G_L(z,\beta,y,h,\xi),
\end{aligned}
\end{equation}
where $K_L$, $\mathcal D_L$, and $\mathcal G_L$ are smooth and
bounded near the portion of the lifted manifold under consideration.
Thus the lifted extended manifold is locally parameterized by
\begin{equation}
\label{eq:left-extended-weighted-parameterization}
\Psi_L(z,\beta,y,h,\xi)
=
\left(
z,\beta,y,h,k,D,\Gamma,\xi
\right),
\end{equation}
with $k,D,\Gamma$ given by
\eqref{eq:left-extended-weighted-graph}. The distinguished reduced left outer orbit approaches
\begin{equation}
\label{eq:left-corner-limit-point}
p_+
=
\left(
0,0,a,h_+,k_L(c),D,\Gamma,c
\right)
\in\mathcal P_+,
\end{equation}
where
\begin{equation}
\label{eq:left-corner-a-h}
a=\sqrt{-2D},
\qquad
h_+=\frac{\Gamma}{D}a.
\end{equation}
This is precisely the asymptotic behavior established in
Proposition~\ref{prop:left-reduced-orbit}, namely $
\sqrt{\rho}(u-c)\longrightarrow a,
\ \ 
\sqrt{\rho}(v-v_\delta)\longrightarrow h_+.$ 
Indeed, by the blow-down relations $y=\sqrt{\rho}(u-\xi),
\ h=\sqrt{\rho}(v-v_\delta),$ the asymptotics of Proposition~\ref{prop:left-reduced-orbit}, $\sqrt{\rho}(u-c)\to a,
\ \sqrt{\rho}(v-v_\delta)\to h_+,$ are equivalent to
$y\to a,\ h\to h_+,$ in the weighted corner coordinates. Hence the lifted reduced outer orbit
approaches the equilibrium $p_+\in\mathcal P_+.$ Fix $\beta_0>0$ and define
\begin{equation}
\label{eq:left-incoming-beta-section}
\widehat{\mathcal N}_{L,\beta_0}
=
\widehat{\mathcal W}_L^u
\cap
\{\beta=\beta_0\}.
\end{equation}
Since $\dot\beta
=
-\frac{\beta}{2n}
\left(
y-\beta^nk
\right)$ and $y\to a>0$ along the distinguished incoming orbit, the section
$\{\beta=\beta_0\}$ is transverse to the flow for all sufficiently
small $\beta_0>0$. Hence
$\widehat{\mathcal N}_{L,\beta_0}$ is four-dimensional.

Fix $\beta_0>0$. Since the section
$\{\beta=\beta_0\}$ remains away from the singular corner,
the lifted manifold
$\widehat{\mathcal N}_{L,\beta_0}$
is represented by the graph
\eqref{eq:left-extended-weighted-parameterization}.
The graph functions are $C^1$ and bounded in a neighborhood of the
distinguished orbit. Since the graph correction is multiplied by
$z^3\beta^{3n}=\varepsilon,$ its derivatives with respect to $y,h,$ and $\xi$ remain
$O(z^3\beta^{3n})$, whereas differentiation with respect to $z$
produces
$O(z^2\beta^{3n})$.
Consequently, the tangent vectors have the following asymptotic form.
Since $(\beta,y,h,\xi)$ form local coordinates on the lifted manifold for
$\beta>0$, the variables $(z,y,h,\xi)$ provide local coordinates on the
section $\{\beta=\beta_0\}$.
Differentiating
\eqref{eq:left-extended-weighted-parameterization} on
$\{\beta=\beta_0\}$ gives the tangent vectors
\begin{equation}
\label{eq:left-weighted-z-vector-beta}
\partial_z\Psi_L
=
\left(
1,0,0,0,
O(z^2\beta_0^{3n}),
O(z^2\beta_0^{3n}),
O(z^2\beta_0^{3n}),
0
\right),
\end{equation}
\begin{equation}
\label{eq:left-weighted-y-vector-beta}
\partial_y\Psi_L
=
\left(
0,0,1,0,
O(z^3\beta_0^{3n}),
O(z^3\beta_0^{3n}),
O(z^3\beta_0^{3n}),
0
\right),
\end{equation}
and
\begin{equation}
\label{eq:left-weighted-h-vector-beta}
\partial_h\Psi_L
=
\left(
0,0,0,1,
O(z^3\beta_0^{3n}),
O(z^3\beta_0^{3n}),
O(z^3\beta_0^{3n}),
0
\right).
\end{equation}
Differentiation with respect to $\xi$ gives
\begin{equation}
\label{eq:left-weighted-xi-vector-beta}
\partial_\xi\Psi_L
=
\left(
0,0,0,0,
k_L'(\xi),
D_L'(\xi),
\Gamma_L'(\xi),
1
\right)
+
O(z^3\beta_0^{3n}).
\end{equation}

On the singular level $z=0$, these become
\begin{equation}
\label{eq:left-exact-reduced-tangent-vectors}
\begin{aligned}
Z_L
&=
(1,0,0,0,0,0,0,0),
\\
Y_L
&=
(0,0,1,0,0,0,0,0),
\\
H_L^{\mathrm{tan}}
&=
(0,0,0,1,0,0,0,0),
\\
X_L
&=
\left(
0,0,0,0,
k_L'(\xi),D_L'(\xi),\Gamma_L'(\xi),1
\right).
\end{aligned}
\end{equation}
Therefore,
\begin{equation}
\label{eq:left-section-tangent-limit}
T\widehat{\mathcal N}_{L,\beta_0}
\longrightarrow
\operatorname{span}
\left\{
Z_L,Y_L,H_L^{\mathrm{tan}},X_L
\right\}
\end{equation}
as the base point tends to $p_+$ and
$\beta_0\to0^+$. The convergence is in the Grassmannian of four-dimensional subspaces of
$\mathbb R^8$.

The variables $y$ and $h$ provide independent local coordinates on
\eqref{eq:left-incoming-beta-section}.  Consequently, the corresponding
tangent vectors satisfy
\begin{equation}
\label{eq:left-yh-tangent-vectors}
Y_L
=
\left(
0,0,1,0,0,0,0,0
\right),
\qquad
H_L
=
\left(
0,0,0,1,0,0,0,0
\right).
\end{equation}
Variation of the extended perturbation parameter gives the radial
weighted direction
\begin{equation}
\label{eq:left-z-tangent-vector}
Z_L
=
\left(
1,0,0,0,0,0,0,0
\right).
\end{equation}
Finally, variation of $\xi$ gives
\begin{equation}
\label{eq:left-xi-weighted-vector}
X_L
=
\left(
0,0,0,0,
k_L'(c),D_L'(c),\Gamma_L'(c),1
\right).
\end{equation}
Thus
\begin{equation}
\label{eq:left-section-limiting-tangent}
T\widehat{\mathcal N}_{L,\beta_0}
\longrightarrow
\operatorname{span}
\left\{
Z_L,Y_L,H_L,X_L
\right\}
\end{equation}
as $\beta_0\to0^+$.

\begin{prop}[Outer-to-corner transversality]
\label{prop:outer-corner-transversality}

For all sufficiently small $\beta_0>0$, the incoming section
$\widehat{\mathcal N}_{L,\beta_0}$ is transverse to
$W^s(\mathcal P_+)$ near the distinguished left corner point.

Similarly, the outgoing section obtained from
$\widehat{\mathcal W}_R^s$ is transverse to
$W^u(\mathcal P_-)$ near the distinguished right corner point.

\end{prop}

\begin{proof}
As $\beta_0\to0^+$, the distinguished point of
$\widehat{\mathcal N}_{L,\beta_0}$ converges to the corner equilibrium
$p_+$.
Consequently, it suffices to verify transversality for the limiting
tangent spaces at $p_+$. By
Proposition~\ref{prop:corner-normal-hyperbolicity},
$T_{p_+}W^s(\mathcal P_+)
=T_{p_+}\mathcal P_+
\oplus E_+^s,$ where $E_+^s$ is the one-dimensional stable normal eigenspace. Its eigenvector has a nonzero $\beta$-component and no $z$-component.

The normal variables are ordered as $(\beta,z,y,h).$ The vectors $Z_L,Y_L,H_L^{\mathrm{tan}}$ have normal projections
$(0,1,0,0)$, $(0,0,1,0)$, $(0,0,0,1).$
The restriction of the linearization computed in
Proposition~\ref{prop:corner-normal-hyperbolicity}
to the $(z,y,h)$ variables is
\[
\begin{pmatrix}
a/2&0&0\\
0&a&0\\
0&h_+/2&a/2
\end{pmatrix}.
\]
Since the full normal linearization has exactly three unstable normal eigenvalues, this invariant subspace coincides with the unstable normal eigenspace: $E_+^u
=\operatorname{span}
\{Z_L,Y_L,H_L^{\rm tan}\}.$ Since
$T_{p_+}W^s(\mathcal P_+)
=T_{p_+}\mathcal P_+
\oplus E_+^s$ and the normal splitting is $\mathbb R^8
=T_{p_+}\mathcal P_+
\oplus E_+^s
\oplus E_+^u,$ we obtain
$T_{p_+}W^s(\mathcal P_+)
+\operatorname{span}
\left\{
Z_L,Y_L,H_L^{\mathrm{tan}}
\right\}
=
\mathbb R^8.$ Since $E_+^u
\subset
\lim_{\beta_0\to0}
T\widehat{\mathcal N}_{L,\beta_0},$ we obtain $T_{p_+}W^s(\mathcal P_+)
+\lim_{\beta_0\to0^+}
T\widehat{\mathcal N}_{L,\beta_0}
=\mathbb R^8.$ Since the tangent spaces of
$\widehat{\mathcal N}_{L,\beta_0}$
depend continuously on $\beta_0$, this transversality relation persists
for all sufficiently small $\beta_0>0$. Therefore, $\widehat{\mathcal N}_{L,\beta_0}\pitchfork W^s(\mathcal P_+).$

At $p_-$, the same calculation applies with stable and unstable
directions interchanged. The $(z,y,h)$-space is the
three-dimensional stable normal eigenspace, while the unstable normal
eigenvector has a nonzero $\beta$-component. The limiting tangent
space of the outgoing section contains $Z_R=(1,0,0,0,0,0,0,0),
\
Y_R=(0,0,1,0,0,0,0,0),
\
H_R^{\mathrm{tan}}=(0,0,0,1,0,0,0,0).$ Therefore
$T_{p_-}W^u(\mathcal P_-)
+\operatorname{span}
\left\{
Z_R,Y_R,H_R^{\mathrm{tan}}
\right\}
=
\mathbb R^8,$ and openness gives $\widehat{\mathcal N}_{R,\beta_0}
\pitchfork
W^u(\mathcal P_-)$ for all sufficiently small $\beta_0>0$. Therefore all hypotheses of
Lemma~\ref{lem:corner-lemma}
are verified at the left corner:
the manifold of equilibria
$\mathcal P_+$
is normally hyperbolic
(Proposition~\ref{prop:corner-normal-hyperbolicity}),
the invariant hypersurface
$\mathcal S=\{z=0\}$
contains
$W^s(\mathcal P_+)$,
the unstable bundle contains the transverse
$z$-direction,
and $\widehat{\mathcal N}_{L,\beta_0}
\pitchfork
W^s(\mathcal P_+).$ The corresponding backward-time hypotheses hold at
$\mathcal P_-$.
\end{proof}
\begin{corollary}[Corner passage]
\label{cor:corner-passage}

Fix $\beta_0>0$ sufficiently small and, for $\eta>0$, define
\begin{equation}
\widehat{\mathcal N}_{L,\beta_0,\eta}
=
\widehat{\mathcal N}_{L,\beta_0}
\cap
\{z=\eta\}.
\end{equation}
Then the forward images of
$\widehat{\mathcal N}_{L,\beta_0,\eta}$
through a fixed neighborhood of $\mathcal P_+$ converge in $C^1$,
as $\eta\to0^+$, to the corresponding portion of
$W^u(\mathcal Q_+)$, where $\mathcal Q_+\subset\mathcal P_+$ is the
projection of
$\widehat{\mathcal N}_{L,\beta_0}\cap W^s(\mathcal P_+)$
along the stable foliation of $W^s(\mathcal P_+)$.

Similarly, after fixing a sufficiently small right-hand section,
the corresponding backward images through the neighborhood of
$\mathcal P_-$ converge in $C^1$ to the appropriate portion of
$W^s(\mathcal Q_-)$.
\end{corollary}

\begin{proof}
Proposition~\ref{prop:corner-normal-hyperbolicity}
establishes the normal hyperbolicity of
$\mathcal P_\pm$,
while the discussion preceding
Lemma~\ref{lem:corner-lemma}
verifies the invariant-hypersurface hypotheses.
The manifold
$\widehat{\mathcal N}_{L,\beta_0}$
is transverse to the vector field because
$\dot\beta\neq0$
along the distinguished incoming orbit and therefore throughout a sufficiently small neighborhood of that orbit.
Its transversality to
$W^s(\mathcal P_+)$
was established in
Proposition~\ref{prop:outer-corner-transversality}.
Since
$\mathcal S=\{z=0\}$,
the defining transverse coordinate in
Lemma~\ref{lem:corner-lemma}
is $z$. Since
$\mathcal S=\{z=0\}$,
the defining transverse coordinate in
Schecter's Corner Lemma
(Lemma~\ref{lem:corner-lemma})
is $z$.
The conclusion therefore follows by applying
Lemma~\ref{lem:corner-lemma}
to $\widehat{\mathcal N}_{L,\beta_0,\eta}
=\widehat{\mathcal N}_{L,\beta_0}
\cap
\{z=\eta\}.$ The right-hand statement follows by reversing time.
\end{proof}

\subsubsection{Middle matching condition}
\label{subsubsec:middle-transversality}

The middle matching condition is scalar. Its nondegeneracy is established
in the proof of Proposition~\ref{prop:singular-concatenation-persistence},
where it is shown that the associated matching function satisfies
$\mathcal G'(\lambda_*)\neq0.$ This nondegeneracy provides the middle transversality condition required
in the persistence argument below.

\subsubsection{Exit-chart passage}
The right reduced outer trajectory leaves the blown-up region through
the side $u-\xi<0$. We therefore introduce the exit directional chart,
denoted by $K_{\mathrm{out}}$, corresponding to the negative-$\bar u$
direction. Its coordinates are defined by
\begin{equation}
    \begin{cases}
    \displaystyle
        u - \xi = -r_1 \\\displaystyle
        \rho_0 = r_1\overline{\rho}_1 \\\displaystyle
        v - \frac{k_3}{k_1} = r_1\overline{v}_1 \\\displaystyle
        \epsilon = r_1^2 \overline{\epsilon}_1^2
    \end{cases}
\end{equation}
\begin{lemma}[Right outer endpoint in the exit chart]
\label{lem:right-outer-endpoint-exit}

The right outer endpoint $P_R$ is represented in the exit chart
$K_{\mathrm{out}}$ by
\begin{equation}
P_R^{\mathrm{out}}
=
\left(
0,\,
0,\,
-\frac{\Gamma}{D},\,
0,\,
W_R,\,
c
\right),
\label{eq:PR-exit-chart}
\end{equation}
where the coordinates are ordered as $(r_1,\bar\rho_1,\bar v_1,\bar\varepsilon_1,W,\xi).$
\end{lemma}

\begin{proof}
The spherical and exit-chart coordinates are related by
\begin{equation}
r_1=-r\bar u,
\qquad
\bar\rho_1=-\frac{\bar\rho}{\bar u},
\qquad
\bar v_1=-\frac{\bar v}{\bar u},
\qquad
\bar\varepsilon_1=-\frac{\bar\varepsilon}{\bar u}.
\label{eq:spherical-to-exit}
\end{equation}
By Proposition~\ref{prop:outer-sphere-limits}, at $P_R$,
$$r=0,
\qquad
\bar\rho=0,
\qquad
\bar u=\frac{D}{R},
\qquad
\bar v=\frac{\Gamma}{R},
\qquad
\bar\varepsilon=0,
\qquad
W=W_R,
\qquad
\xi=c.
$$
Since $D<0$, one has $D/R<0$, so $P_R$ belongs to
$K_{\mathrm{out}}$. Moreover, $\bar v_1
=-\frac{\bar v}{\bar u}
=-\frac{\Gamma/R}{D/R}
=-\frac{\Gamma}{D}.$ Also, $r_1=0,
\ \bar\rho_1=0,
\  \bar\varepsilon_1=0.$ This proves \eqref{eq:PR-exit-chart}.
\end{proof}
Similarly, rewriting system \eqref{GSPT_eps_sys} in the exit coordinates and applying the chain rule yields the transformed system 
\begin{equation}
\begin{cases}
\displaystyle
\frac{dr_1}{d\tau}
=
-r_1\,\Phi_2,
\\[2ex]
\displaystyle
\frac{d\bar\varepsilon_1}{d\tau}
=
\bar\varepsilon_1\,\Phi_2,
\\[2ex]
\displaystyle
\frac{d\xi}{d\tau}
=
r_1^2\bar\varepsilon_1^2,
\\[2ex]
\displaystyle
\frac{d u}{d\tau}
=
r_1^2\bar\varepsilon_1^2
+r_1\,\Phi_2,
\\[2ex]
\displaystyle
\frac{d\bar\rho_1}{d\tau}
=
\bar\rho_1\,\Phi_2
-r_1\bar\rho_1
-r_1\bar\varepsilon_1^2w_1,
\\[2ex]
\displaystyle
\frac{d\bar v_1}{d\tau}
=
\bar v_1\,\Phi_2
+\frac{\bar\varepsilon_1^2}{\bar\rho_1}
\left[
\left(\frac{k_3}{k_1}+r_1\bar v_1\right)w_1-w_3
\right],
\\[2ex]
\displaystyle
\frac{dw_1}{d\tau}
=
-r_1\bar\rho_1,
\\[2ex]
\displaystyle
\frac{dw_2}{d\tau}
=
-r_1\bar\rho_1(\xi-r_1),
\\[2ex]
\displaystyle
\frac{dw_3}{d\tau}
=
-r_1\bar\rho_1
\left(\frac{k_3}{k_1}+r_1\bar v_1\right),
\end{cases}\label{eq:chart2_full}
\end{equation}
where
\begin{equation}\label{eq:Phi2}
\Phi_2
=
\frac{\bar\varepsilon_1^2}{\bar\rho_1}
\left[(\xi-r_1)w_1-w_2\right]
-
A\!\left(\frac{k_3}{k_1}+r_1\bar v_1\right)
\frac{r_1^\alpha\bar\varepsilon_1^{2\alpha+2}}
{\bar\rho_1^{\alpha+1}}
-r_1\bar\varepsilon_1^2 .
\end{equation}
\begin{lemma}[Regularity of the exit-region vector field]
\label{lem:exit-regularity}

Let $\alpha=\frac{p}{q}\in\mathbb Q\cap(0,1],
\
\gcd(p,q)=1.$ Restrict attention to a compact set on which $\bar\rho_1$ is bounded
away from zero. Introduce the ramified radial coordinate
\begin{equation}
r_1=s_1^q,
\qquad
s_1\geq0.
\label{eq:r1-s1}
\end{equation}
Then the exit-region system, written in the variables
$(s_1,\bar\rho_1,\bar v_1,\bar\varepsilon_1,
w_1,w_2,w_3,\xi),$ extends smoothly to $s_1=0$.
\end{lemma}

\begin{proof}
Under the substitution $r_1=s_1^q$, the fractional power in
\eqref{eq:Phi2} becomes $r_1^\alpha
=\left(s_1^q\right)^{p/q}
=s_1^p.$ Define
\begin{equation}
\widetilde\Phi_2
=
\frac{\bar\varepsilon_1^2}{\bar\rho_1}
\left[
(\xi-s_1^q)w_1-w_2
\right]
-
A\left(
v_\delta+s_1^q\bar v_1
\right)
\frac{s_1^p\bar\varepsilon_1^{2\alpha+2}}
{\bar\rho_1^{\alpha+1}}
-
s_1^q\bar\varepsilon_1^2.
\label{eq:Phi2-weighted}
\end{equation}
Since $\frac{dr_1}{d\tau}
=-r_1\Phi_2,$ we obtain $q s_1^{q-1}\frac{ds_1}{d\tau}
=-s_1^q\widetilde\Phi_2,$ and hence
\begin{equation}
\frac{ds_1}{d\tau}
=
-\frac1q s_1\widetilde\Phi_2.
\label{eq:s1-exit}
\end{equation}
All remaining equations are obtained from
\eqref{eq:chart2_full} by replacing $r_1$ with $s_1^q$ and
$r_1^\alpha$ with $s_1^p$. Since $\bar\rho_1$ is bounded away from
zero, all right-hand sides extend smoothly to $s_1=0$.
\end{proof}
Restricting to $r_1 = 0$, gives the reduced system 
\begin{equation}
    \begin{cases}
    \displaystyle
        \frac{d\overline{\rho}_1}{d\tau} = \overline{\eps}_1^2 \big(\xi w_1 - w_2 \big) \\[1.5ex]\displaystyle
        \frac{d\overline{v}_1}{d\tau} = \frac{\overline{v}_1\overline{\eps}_1^2\big(\xi w_1 - w_2 \big)}{\overline{\rho}_1} + \frac{\overline{\eps}_1^2 \big(\frac{k_3}{k_1}w_1 - w_3 \big)}{\overline{\rho}_1} \\[1.5ex]\displaystyle
        \frac{d\overline{\eps}_1}{d\tau} = \frac{\big(\xi w_1 - w_2 \big)\overline{\eps}_1^3}{\overline{\rho}_1} \\[1.5ex]\displaystyle
        \frac{d r_1}{d\tau}=\frac{d\xi}{d\tau} = \frac{du}{d\tau} = \frac{dw_1}{d\tau} = \frac{dw_2}{d\tau} = \frac{dw_3}{d\tau} = 0
    \end{cases}
\end{equation}
To analyze the reduced flow in the exit region, recall
\begin{equation}
D=\xi w_1-w_2,
\qquad
\Gamma=\frac{k_3}{k_1}w_1-w_3.
\end{equation}
On $r_1=0$, these quantities are constant, and the reduced system can
be written as
\begin{equation}
\begin{cases}
\displaystyle
\frac{d\bar\rho_1}{d\tau}
=
D\bar\varepsilon_1^2,
\\[1.2ex]
\displaystyle
\frac{d\bar v_1}{d\tau}
=
\frac{\bar\varepsilon_1^2}{\bar\rho_1}
\left(
\Gamma+D\bar v_1
\right),
\\[1.2ex]
\displaystyle
\frac{d\bar\varepsilon_1}{d\tau}
=
\frac{D\bar\varepsilon_1^3}{\bar\rho_1},
\\[1.2ex]
\displaystyle
\frac{dr_1}{d\tau}
=
\frac{d\xi}{d\tau}
=
\frac{du}{d\tau}
=
\frac{dw_1}{d\tau}
=
\frac{dw_2}{d\tau}
=
\frac{dw_3}{d\tau}
=
0.
\end{cases}
\label{eq:chart2_reduced}
\end{equation}
The invariant relation
\begin{equation}
\Gamma\bar u-D\bar v=0
\end{equation}
also takes a simple form in the exit region. Since $\bar u<0$ and the exit coordinates are obtained by dividing the spherical variables
by $-\bar u$, one has $\bar v_1=
-\frac{\bar v}{\bar u}.$ Consequently, the reduced invariant relation becomes
\begin{equation}
\Gamma+D\bar v_1=0.
\label{eq:chart2_invariant_relation}
\end{equation}
Define $z_1:=\Gamma+D\bar v_1.$ Since $D$ and $\Gamma$ are constant when $r_1=0$, system
\eqref{eq:chart2_reduced} gives
\begin{align}
\frac{dz_1}{d\tau}=
D\frac{d\bar v_1}{d\tau}=
\frac{D\bar\varepsilon_1^2}{\bar\rho_1}z_1.
\label{eq:z1_reduced}
\end{align}
Thus,
\begin{equation}
\mathcal N_2
=
\left\{
r_1=0,\;
z_1=0
\right\}
=
\left\{
r_1=0,\;
\bar v_1=-\frac{\Gamma}{D}
\right\}
\label{eq:N2}
\end{equation}
is invariant under the reduced the exit region system.

\begin{lemma}[Normal hyperbolicity in the exit region]
\label{lem:chart2-hyperbolicity}

Let $\alpha=\frac{p}{q}\in\mathbb Q\cap(0,1],
\ \gcd(p,q)=1,$ and let $K_2$ be a compact subset of $\mathcal N_2$ on which $\bar\rho_1>0,
\ \bar\varepsilon_1>0,
\ D<0.$ After introducing $r_1=s_1^q$, the manifold $\mathcal N_2$ is
uniformly normally hyperbolic on $K_2$. The $s_1$-direction is
repelling and the $z_1$-direction is attracting.
\end{lemma}

\begin{proof}
On $\mathcal N_2$, one has $\widetilde\Phi_2
=\frac{D\bar\varepsilon_1^2}{\bar\rho_1}$ at $s_1=0$. Linearization of \eqref{eq:s1-exit} in the
$s_1$-direction therefore gives $\lambda_{s_1}
=-\frac{D\bar\varepsilon_1^2}
{q\bar\rho_1}>0.$ The transverse variable $z_1$ has eigenvalue $
\lambda_{z_1}
=\frac{D\bar\varepsilon_1^2}
{\bar\rho_1}<0.$ Since $K_2$ is compact and
$\bar\rho_1$, $\bar\varepsilon_1$, and $-D$ are bounded away from
zero on $K_2$, these eigenvalues are uniformly bounded away from
zero. Hence $\mathcal N_2$ is uniformly normally hyperbolic, with a
repelling $s_1$-direction and an attracting $z_1$-direction.
\end{proof}

By Proposition~\ref{prop:outer-sphere-limits}, the directional
components of the outgoing endpoint selected by the right reduced
outer orbit are
\begin{equation}
(\bar u_{\mathrm{out}},\bar v_{\mathrm{out}})
=
\left(
\frac{D}{R},
\frac{\Gamma}{R}
\right).
\label{eq:Pout}
\end{equation}
Since $D<0$, the $\bar u$-component of $P_{\mathrm{out}}$ is negative,
and therefore $P_{\mathrm{out}}$ is represented in the exit region. Moreover,
\begin{equation}
-\frac{\bar v(P_{\mathrm{out}})}
{\bar u(P_{\mathrm{out}})}
=
-\frac{\Gamma/R}{D/R}
=
-\frac{\Gamma}{D}.
\end{equation}
Thus the outgoing reduced orbit lies on $\mathcal N_2$.

On $\mathcal N_2$, system \eqref{eq:chart2_reduced} reduces to
\begin{equation}
\frac{d\bar\rho_1}{d\tau}
=
D\bar\varepsilon_1^2,
\qquad
\frac{d\bar\varepsilon_1}{d\tau}
=
\frac{D\bar\varepsilon_1^3}{\bar\rho_1},
\qquad
\bar v_1=-\frac{\Gamma}{D}.
\label{eq:chart2_on_N2}
\end{equation}
Since $D<0$, it follows that
\begin{equation}
\frac{d\bar\rho_1}{d\tau}<0,
\qquad
\frac{d\bar\varepsilon_1}{d\tau}<0
\end{equation}
whenever $\bar\rho_1>0$ and $\bar\varepsilon_1>0$.

Furthermore,
\begin{align}
\frac{d}{d\tau}
\left(
\frac{\bar\rho_1}{\bar\varepsilon_1}
\right)
&=
\frac{
\bar\varepsilon_1\frac{d\bar\rho_1}{d\tau}
-
\bar\rho_1\frac{d\bar\varepsilon_1}{d\tau}
}{
\bar\varepsilon_1^2
}
\nonumber\\
&=
0.
\end{align}
Hence the ratio $\bar\rho_1/\bar\varepsilon_1$ remains constant along
the reduced orbit in the exit region.

Fix $R>1$ sufficiently large and define the exit region overlap section
\begin{equation}
\Sigma_{\mathrm{out}}^R
=
\left\{
\bar\varepsilon_1=\frac{1}{R}
\right\}.
\label{eq:Sigma_out_R}
\end{equation}
The reduced orbit obtained by continuing the inner connection from
the positive-$\varepsilon$ region intersects $\Sigma_{\mathrm{out}}^R$ transversely at a unique
point. Denote this intersection by $Q_{\mathrm{out}}^0$.

The transition map from the positive-$\varepsilon$ region to the exit region is
\begin{equation}
r_1
=
-r_\varepsilon\bar u_\varepsilon,
\qquad
\bar\rho_1
=
-\frac{\bar\rho_\varepsilon}{\bar u_\varepsilon},
\qquad
\bar\varepsilon_1
=
-\frac{1}{\bar u_\varepsilon},
\qquad
\bar v_1
=
-\frac{\bar v_\varepsilon}{\bar u_\varepsilon}.
\label{eq:transition_3_to_2}
\end{equation}
Equating the exit-chart coordinates directly with those of
$K_\rho$ gives
\begin{equation}
r_\rho
=
r_1\bar\rho_1,
\qquad
\bar u_\rho
=
-\frac{1}{\bar\rho_1},
\qquad
\bar v_\rho
=
\frac{\bar v_1}{\bar\rho_1},
\qquad
\bar\varepsilon_\rho
=
\frac{\bar\varepsilon_1}{\bar\rho_1}.
\label{eq:rho-to-exit-transition}
\end{equation}
Therefore, the section $\Sigma_{\mathrm{out}}^R$ corresponds in
the positive-$\varepsilon$ region to $
\bar u_\varepsilon=-R.$ At the reduced intersection point $Q_{\mathrm{out}}^0$, one also has
$\bar v_\varepsilon
=
-R\frac{\Gamma}{B}.$ Thus, the reduced orbit exiting the positive-$\varepsilon$ region determines initial data for
the outgoing passage through the exit region.

The normal hyperbolicity described in
\label{lem:chart2-hyperbolicity} applies only on compact subsets
away from $\bar\varepsilon_1=0$. It provides the local geometric
control needed to continue trajectories between the overlap section
$\Sigma_{\mathrm{out}}^R$ and a section in the outer region. The
continuation from that outer section to the right state $U_R$ is
provided by the stable manifold constructed in
Proposition~\ref{prop:outer-persistence}.

\subsection{Transitions Between the Blow-Up Charts}
The complete blow-up atlas contains the entry chart
$K_{\mathrm{in}}$, the positive-density chart $K_\rho$, the
positive-$\varepsilon$ chart $K_\varepsilon$, and the exit chart
$K_{\mathrm{out}}$. The chart $K_\rho$
contains the singular orbit responsible for the $O(1)$ change in
$W$, while $K_\varepsilon$ resolves the central positive-$\varepsilon$
passage.

We first record the transition between $K_\rho$ and
$K_\varepsilon$. Equating the coordinate representations
\begin{equation}
\rho_0=r_\rho=r_\varepsilon\bar\rho_\varepsilon,
\qquad
u-\xi=r_\rho\bar u_\rho
=
r_\varepsilon\bar u_\varepsilon,
\end{equation}
\begin{equation}
v-v_\delta
=
r_\rho\bar v_\rho
=
r_\varepsilon\bar v_\varepsilon,
\qquad
\sqrt{\varepsilon}
=
r_\rho\bar\varepsilon_\rho
=
r_\varepsilon,
\end{equation}
gives
\begin{equation}
r_\varepsilon
=
r_\rho\bar\varepsilon_\rho,
\qquad
\bar\rho_\varepsilon
=
\frac{1}{\bar\varepsilon_\rho},
\qquad
\bar u_\varepsilon
=
\frac{\bar u_\rho}{\bar\varepsilon_\rho},
\qquad
\bar v_\varepsilon
=
\frac{\bar v_\rho}{\bar\varepsilon_\rho}.
\label{eq:rho-to-epsilon-transition}
\end{equation}
In contrast,
\begin{equation}
r_\rho
=
r_\varepsilon\bar\rho_\varepsilon,
\qquad
\bar\varepsilon_\rho
=
\frac{1}{\bar\rho_\varepsilon},
\qquad
\bar u_\rho
=
\frac{\bar u_\varepsilon}{\bar\rho_\varepsilon},
\qquad
\bar v_\rho
=
\frac{\bar v_\varepsilon}{\bar\rho_\varepsilon}.
\label{eq:epsilon-to-rho-transition}
\end{equation}
\begin{remark}
\label{rem:epsilon-passage-limited-role}

The reduced positive-$\varepsilon$ equations provide a finite-time
transition only on compact subsets on which
$\bar\rho_\varepsilon$ is finite and bounded away from zero. They do
not resolve the limiting region $\bar\rho_\varepsilon\to+\infty,
$ which corresponds, under
\eqref{eq:epsilon-to-rho-transition}, to $\bar\varepsilon_\rho\to0.
$ The singular middle orbit and the order-one change from $W_L$ to
$W_R$ lie on this positive-density boundary and therefore must be 
analyzed in $K_\rho$.
\end{remark}
\begin{remark}[Role of the two central charts]
\label{rem:central-chart-overlap}

The charts $K_\varepsilon$ and $K_\rho$ are overlapping coordinate
representations of the central blown-up region; they do not describe
two consecutive orbit segments. In their common domain, the
coordinate transformations
\eqref{eq:rho-to-epsilon-transition} and
\eqref{eq:epsilon-to-rho-transition} conjugate the two vector fields,
up to the corresponding positive time rescaling. Consequently, an
orbit represented in $K_\varepsilon$ and the corresponding orbit
represented in $K_\rho$ are the same geometric trajectory.

The chart $K_\varepsilon$ is useful on compact subsets on which
$\bar\rho_\varepsilon$ is finite and bounded away from zero. The
positive-density chart $K_\rho$ is required to resolve the limiting
regime $\bar\rho_\varepsilon\longrightarrow+\infty,
\quad
\bar\varepsilon_\rho
=\frac{1}{\bar\rho_\varepsilon}
\longrightarrow0,$ which contains the singular middle orbit. Thus, the central passage
must be constructed as one geometric passage, using
$K_\varepsilon$ where its coordinates remain regular and $K_\rho$
near the positive-density boundary. The corresponding chart
transition maps are coordinate changes, not additional dynamical
transition maps.
\end{remark}

The complete atlas consists of the entry chart $K_1$, the
positive-$\varepsilon$ chart $K_\varepsilon$, the positive-density
chart $K_\rho$, and the exit chart $K_2$. The charts
$K_\varepsilon$ and $K_\rho$ overlap in the central region and
represent the same geometric trajectories wherever both coordinate
systems are defined.

The sections $\Sigma_0^R
=\left\{\bar u_\varepsilon=R
\right\},
\ \Sigma_1^R
=\left\{
\bar u_\varepsilon=-R
\right\}$
are transverse to the reduced $K_\varepsilon$ flow on compact
subsets on which $\bar\rho_\varepsilon$ is finite and bounded away
from zero. The limiting central corner
$\bar\rho_\varepsilon\to+\infty$ is represented instead by
$\bar\varepsilon_\rho\to0$ in $K_\rho$.

\begin{prop}[Reduced singular concatenation and local chart passages]
\label{prop:reduced-concatenation-summary}
Let $\alpha=\frac{p}{q}\in\mathbb Q\cap(0,1],
\ \gcd(p,q)=1,$ and assume that $A$ is smooth on a neighborhood of the values attained by $v_\delta+s_1^q\bar v_1$.

Then the following statements hold.
\begin{enumerate}
\item
The reduced left and right outer orbit segments persist up to the
fixed outer sections constructed in
Proposition~\ref{prop:outer-persistence}.
\item
The entry and exit reduced systems admit the normally hyperbolic
invariant manifolds $\mathcal N_1$ and $\mathcal N_2$ on compact
subsets bounded away from their respective corner endpoints.
\item
The positive-density reduced system admits the singular middle
orbit of Proposition~\ref{prop:singular-middle-connection}, along which
$W$ changes from $W_L$ to $W_R$.
\item
On compact subsets with
$0<m\leq\bar\rho_\varepsilon\leq M<\infty$, the
positive-$\varepsilon$ flow defines finite-time transition maps
between fixed transverse $\bar u_\varepsilon$-sections.
\end{enumerate}
These conclusions construct the reduced singular configuration and
the local chart passages. The identification of the incoming and
outgoing traces in the weighted corner coordinates, together with
the required transversality verification and the resulting
positive-$\varepsilon$ persistence, is carried out in
Subsections~\ref{subsubsec:outer-corner-transversality} and
\ref{subsubsec:middle-transversality}.
\end{prop}
\begin{proof}
Statement (1) is Proposition~\ref{prop:outer-persistence}.
Statement (2) follows from
Lemmas~\ref{lem:chart1-hyperbolicity} and
\ref{lem:chart2-hyperbolicity}.
Statement (3) is Proposition~\ref{prop:singular-middle-connection}. For statement (4), fix compact constants $0<m<M<\infty$ and two values $R_+>R_-$. At $r_\varepsilon=0$,
equation~\eqref{eq:epsilon-chart-reduced} gives $\frac{d\bar u_\varepsilon}{d\tau}
=\frac{D}{\bar\rho_\varepsilon}.$ Since $D<0$ and
$m\leq\bar\rho_\varepsilon\leq M$, one has $\frac{d\bar u_\varepsilon}{d\tau}
\leq
\frac{D}{M}<0.$ Hence, the reduced orbit crosses each fixed
$\bar u_\varepsilon$-section transversely and reaches
$\{\bar u_\varepsilon=R_-\}$ from
$\{\bar u_\varepsilon=R_+\}$ in uniformly bounded time. On the specified compact set, the full vector field
\eqref{eq:chart3_full} is a smooth perturbation of the reduced vector
field for $\alpha\in\{\frac12,1\}$. Continuous dependence of
solutions and the implicit function theorem applied to the hitting
equation $\bar u_\varepsilon(\tau)=R_-$ therefore give the corresponding local transition map for all
sufficiently small $r_\varepsilon>0$. The remaining identification and transversality argument is supplied
by Proposition~\ref{prop:outer-corner-transversality}, while the
passage through the corner neighborhoods follows from
Corollary~\ref{cor:corner-passage}.
\end{proof}
\begin{prop}[Persistence of the singular concatenation]
\label{prop:singular-concatenation-persistence}

Assume that the Riemann data satisfy the overcompressive condition
and the hypotheses of the left and right reduced outer-orbit
propositions. Let $\alpha\in\mathbb Q\cap(0,1],
\quad
2\alpha=\frac{m}{n}$ in lowest terms, and assume that $A$ is sufficiently smooth in a
neighborhood of the values attained by the profile.

Then there exists $\varepsilon_0>0$ such that, for every
$0<\varepsilon<\varepsilon_0,$
the blown-up Dafermos system admits a heteroclinic orbit joining
the invariant manifold associated with the left state $H_L$ to the
invariant manifold associated with the right state $H_R$.

The heteroclinic orbit converges, as $\varepsilon\to0^+$, to the
singular concatenation consisting of the left reduced outer orbit,
the reduced middle orbit, and the right reduced outer orbit.
\end{prop}

\begin{proof}
Fenichel theory and Proposition~\ref{prop:outer-persistence} give
the persistent left and right outer manifolds up to fixed sections
away from the singular region. The entry and exit charts transport
these manifolds into the weighted corner neighborhoods.

By Proposition~\ref{prop:outer-corner-transversality}, the incoming
manifold is transverse to $W^s(\mathcal P_+)$ and the outgoing
manifold is transverse to $W^u(\mathcal P_-)$. Corollary
\ref{cor:corner-passage}, obtained from Schecter's Corner Lemma,
therefore propagates these manifolds through the corresponding
corner neighborhoods. Their images converge in $C^1$ to the
unstable and stable manifolds associated with the reduced middle
orbit.

The remaining matching is governed by the scalar function
$\mathcal G(\lambda)
=k_L-k_R-\frac{\pi a^2}{\lambda^3}.$ At the reduced matching value $\lambda=\lambda_*$, $\mathcal G(\lambda_*)=0,
\quad \mathcal G'(\lambda_*)
= \frac{3\pi a^2}{\lambda_*^4}
\neq0.$ Choose a fixed transverse section in the middle region on which the
propagated incoming and outgoing manifolds are represented locally as
$C^1$ graphs. Let $\mathcal G(\lambda,\varepsilon)$ denote the signed difference of these two graphs in the distinguished
matching coordinate. By the $C^1$ convergence supplied by
Corollary~\ref{cor:corner-passage}, together with the smooth finite-time
transition maps constructed in the remaining charts,
$\mathcal G(\lambda,\varepsilon)$ is $C^1$ near
$(\lambda_*,0)$ and satisfies $\mathcal G(\lambda,0)
=k_L-k_R-\frac{\pi a^2}{\lambda^3}.$ Therefore
$\mathcal G(\lambda_*,0)=0,
\
\partial_\lambda\mathcal G(\lambda_*,0)
=\frac{3\pi a^2}{\lambda_*^4}
\neq0.$ The implicit function theorem therefore yields a unique
$C^1$ function
$\lambda=\lambda(\varepsilon)$ for all sufficiently small
$\varepsilon\ge0$ satisfying $\mathcal G(\lambda(\varepsilon),\varepsilon)=0.$ Consequently, for every sufficiently small $\varepsilon>0$, the
propagated incoming and outgoing manifolds intersect.

The resulting orbit joins the persistent left outer manifold to the
persistent right outer manifold and therefore gives the required
heteroclinic connection. The $C^1$ convergence supplied by the
Corner Lemma, together with smooth dependence in the remaining
charts, implies convergence to the reduced singular concatenation
as $\varepsilon\to0^+$.
\end{proof}
\begin{theorem}[Existence of a Dafermos viscous profile]
\label{thm:dafermos-viscous-profile}

Assume the hypotheses of
Propositions~\ref{prop:left-reduced-orbit} and
\ref{prop:right-reduced-orbit}, together with the strict
overcompressive condition~\eqref{eq:overcompressive-speed}.

Then the reduced blown-up problem admits a singular concatenated
configuration consisting of the left reduced outer orbit, the
distinguished reduced middle orbit, and the right reduced outer orbit.

Suppose in addition that $\alpha\in\mathbb Q\cap(0,1],$ and that $A$ is sufficiently smooth in a neighborhood of the values
attained by the profile. Then there exists $\varepsilon_0>0$ such
that, for every $0<\varepsilon<\varepsilon_0$, the Dafermos
regularization admits a self-similar viscous profile connecting
$U_L$ to $U_R$.

As $\varepsilon\to0^+$, the corresponding blown-up heteroclinic
orbit converges to the reduced singular concatenation. After
blow-down, the viscous profiles converge to the isolated
overcompressive singular solution determined by the generalized
Rankine--Hugoniot relations. This limiting solution satisfies the
conservation laws in the sense of distributions and coincides with
the associated shadow-wave solution.
\end{theorem}

\begin{proof}
The left and right reduced outer-orbit propositions, together with
Proposition~\ref{prop:singular-middle-connection} and
Lemma~\ref{lem:outer-W-RH}, give the reduced
singular concatenation. The persistence of this concatenation for sufficiently small positive $\varepsilon$ follows from
Proposition~\ref{prop:singular-concatenation-persistence}. That proposition constructs a heteroclinic orbit of the blown-up
system joining the persistent invariant manifolds associated with
$H_L$ and $H_R$. The heteroclinic orbit of the blown-up system blows down to a solution of the Dafermos profile equations connecting the left and right states. The distributional and shadow-wave conclusions follow from Section~\ref{sec:singular}.
\end{proof}

\begin{remark}[Physical interpretation of the reduced middle orbit]
Along the interior of the reduced middle orbit one has $z>0$,
whereas $z=0$ only at its entry and exit endpoints. Since $\rho=\frac{1}{b^2},\ u=\xi+by,\ v=v_\delta+bh,$ the singular limit $b\to0$ corresponds to $\rho\longrightarrow+\infty,\ u\longrightarrow c,\ v\longrightarrow v_\delta,$ where we used that $\xi=c$ on the reduced problem. Thus the reduced middle orbit represents the infinite-density layer connecting the
limiting states of the left and right reduced outer orbits.
\end{remark}

\section{Conclusion} \label{conclusion}
In this work, we studied the Riemann problem for a three-equation generalized Chaplygin gas whose pressure coefficient depends on an additional transported variable. We derived a complete classification of Riemann solutions for $0<\alpha\le1$. Depending on the Riemann data, the solutions consist of combinations of shock waves, rarefaction waves, contact discontinuities, and overcompressive singular solutions.

We studied the singular solutions using both the classical Dirac-delta framework and the shadow-wave construction. We verified that the $\delta$-shock solutions satisfy the governing system in the sense of distributions and compared them with the corresponding shadow wave solutions. In addition, we considered the Dafermos maximum entropy dissipation principle for representative Riemann data with competing singular solutions. Although a complete analytical characterization of the entropy dissipation criterion remains open, numerical computations suggest that it may provide a useful admissibility criterion for competing singular solutions.

To investigate profiles corresponding to an isolated overcompressive
$\delta$-shock, we considered a Dafermos regularization and analyzed
the resulting singularly perturbed dynamical system by spherical
blow-up. We constructed the reduced singular concatenation and
established persistence of the outer orbit segments. A weighted central-corner blow-up resolves the loss
of normal hyperbolicity at the endpoints of the singular middle
orbit. For rational $\alpha\in(0,1]$, the ramified weighted vector
field is smooth and admits two normally hyperbolic manifolds of
corner equilibria. We verified the outer-to-corner transversality
hypotheses, applied Schecter's Corner Lemma, and used the
nondegeneracy of the scalar middle matching condition to obtain a
heteroclinic orbit for sufficiently small positive viscosity.
Consequently, the isolated overcompressive singular solution is realized as the zero-viscosity limit of a family of self-similar Dafermos viscous profiles.

Future work includes extending the viscous profile analysis to interactions between $\delta$-shock waves and classical waves, including shock waves, rarefaction waves, and contact discontinuities, as well as developing an analytical characterization of the entropy dissipation criterion for competing singular solutions.\\

\noindent
{\bf Acknowledgments.} This work is supported by the National Science Foundation under Grant Number DMS-2349040 (PI: Tsikkou). Any opinions, findings, and conclusions or recommendations expressed in this
material are those of the authors and do not necessarily
reflect the views of the National Science Foundation. \\

\noindent
The authors gratefully thank Prof. Marko Nedeljkov for suggesting the problem, for many insightful discussions and invaluable guidance throughout the development of this work, and for generously sharing his expertise on shadow waves and singular shocks. His encouragement, collaboration, and friendship over the years have been greatly appreciated. The authors also thank Evan Halloran, Ryan Lin, and Emily Peng for the contribution of their MATLAB code, which was used as the structural basis to the numerical analysis completed in this paper.\\

\noindent
{\bf Authors' contributions.}
Thomas Allen: Conceptualization (supporting); Investigation (equal); Writing - original draft (equal); Writing - review \& editing (equal). Ella Kim: Conceptualization (supporting); Data curation (lead); Investigation (equal); Writing - original draft (equal); Writing - review \& editing (equal). Jonathan Nunes: Conceptualization (supporting); Investigation (equal); Writing - original draft (equal); Writing - review \& editing (equal). Rita Xing: Conceptualization (supporting); Investigation (equal); Writing - original draft (equal); Writing - review \& editing (equal). Charis Tsikkou: Conceptualization (lead); Acquisition of funding (lead); Supervision (lead); Writing - original draft (equal); Writing - review \& editing (equal).\\

\noindent
{\bf Data Availability.} The data that support the findings of
this study are available from the corresponding author
upon reasonable request.\\

\noindent
{\bf Conflict of Interest Statement.}
The authors declare that they have no conflict of interest.
\printbibliography
\end{document}